\documentclass{article}
\usepackage{geometry}
\usepackage[utf8]{inputenc}
\usepackage[english]{babel}
\usepackage{graphicx}
\usepackage{array}
\usepackage{tabularx}
\newcolumntype{C}{>{\centering\arraybackslash}X}

\usepackage{amssymb, amsmath, amsthm, mathrsfs, yhmath}
\usepackage{enumerate}

\newtheorem{theorem}{Theorem}[section]
\newtheorem{corollary}[theorem]{Corollary}
\newtheorem{lemma}[theorem]{Lemma}
\newtheorem{proposition}[theorem]{Proposition}

\theoremstyle{definition}

\newtheorem{remark}[theorem]{Remark}

\newtheorem{examplex}[theorem]{Example}
\newenvironment{example}
{\pushQED{\qed}\examplex}
{\popQED\endexamplex}

\usepackage{chngcntr}
\counterwithin{table}{section}
\usepackage{multirow}

\renewcommand{\epsilon}{\varepsilon}
\renewcommand{\theta}{\vartheta}

\usepackage{mathtools}
\usepackage{etoolbox}
\DeclareFontFamily{U}{matha}{\hyphenchar\font45}
\DeclareFontShape{U}{matha}{m}{n}{
      <5> <6> <7> <8> <9> <10> gen * matha
      <10.95> matha10 <12> <14.4> <17.28> <20.74> <24.88> matha12
      }{}
\DeclareSymbolFont{matha}{U}{matha}{m}{n}
\DeclareFontSubstitution{U}{matha}{m}{n}

\DeclareFontFamily{U}{mathx}{\hyphenchar\font45}
\DeclareFontShape{U}{mathx}{m}{n}{
      <5> <6> <7> <8> <9> <10>
      <10.95> <12> <14.4> <17.28> <20.74> <24.88>
      mathx10
      }{}
\DeclareSymbolFont{mathx}{U}{mathx}{m}{n}
\DeclareFontSubstitution{U}{mathx}{m}{n}

\DeclareMathDelimiter{\vvvert}{0}{matha}{"7E}{mathx}{"17}
\DeclarePairedDelimiterX{\normi}[1]
  {\vvvert}
  {\vvvert}
  {\ifblank{#1}{\:\cdot\:}{#1}}

 \usepackage{todonotes}

\usepackage[
    backend=biber,
    citestyle=numeric,
    bibstyle=authoryear,
    dashed=false
]{biblatex}
\usepackage{csquotes}

\makeatletter
\input{numeric.bbx}
\makeatother

\DeclareFieldFormat[article,inbook,incollection,inproceedings,patent,thesis,unpublished, periodical]{title}{{#1\isdot}}
\DeclareFieldFormat[article,inbook,incollection,inproceedings,patent,thesis,unpublished, periodical]{volume}{\mkbibbold{#1}}

\renewbibmacro{in:}{}
\DeclareFieldFormat{pages}{#1}

\title{{Local Phase Tracking and Metastability of Planar Waves in Stochastic Reaction-Diffusion Systems}}
\author{M.\,van den Bosch$^{\,\rm a,}$\footnote{Corresponding author.\\Email addresses: \url{vandenboschm@math.leidenuniv.nl} and \url{hhupkes@math.leidenuniv.nl}.}\,\,, H.\,J.\,Hupkes$^{\,\rm a}$}
\date{\today}

\numberwithin{equation}{section}
\usepackage[hidelinks]{hyperref}

\begin{document}


\maketitle

\begin{center}\small
    \textsc{
    $^{\mathrm a}$Mathematical Institute,  Leiden University,\\ P.O. Box 9512, 2300 RA Leiden, The Netherlands}
\end{center}





\

\begin{abstract}
\noindent Planar travelling waves on   $\mathbb R^d,$ with $ d\geq 2,$ are shown to persist in systems of reaction-diffusion equations  with multiplicative noise  on significantly long timescales with high probability, provided   that the wave is orbitally stable in dimension one ($d=1$). 
While a  global phase tracking mechanism is required to determine the location of the stochastically perturbed wave in one dimension, or on a cylindrical domain, we show that the travelling wave on the full unbounded space can be controlled by keeping track of local deviations only.
In particular, the energy infinitesimally added to or withdrawn from the system by noise  dissipates almost fully into the transverse direction, leaving behind small localised phase shifts.  The noise process considered is white in time and coloured in space, possibly weighted, and either translation invariant or trace class.
\end{abstract}

\,




\

\section{Introduction}\label{sec:intro}
In this paper we investigate the multidimensional (meta)stability of planar travelling wave solutions 
to systems of stochastic reaction-diffusion equations  of the form
\begin{equation}
    \mathrm du=[D\Delta u + f(u)]\mathrm  dt+\sigma g(u)\mathrm d
W_t^Q,\quad t>0,\label{eq:u}
\end{equation}
where solutions $u=u(x,y,t)\in \mathbb R^n$, with $n\geq 1$, evolve over  $\mathbb R^d=\mathbb R\times \mathbb R^{d-1}\ni (x,y)$
for some spatial dimension $d \ge 2$. We
assume that \eqref{eq:u} in the deterministic setting $\sigma =0 $ admits a spectrally stable solution of the form
\begin{equation}
\label{eq:int:planar:wave}
    u(x,y,t) = \Phi_0( x  - c_0 t),
\end{equation}
\pagebreak and study  the behaviour of small perturbations to this planar wave within the small-noise regime $0 < \sigma \ll 1$. 
The noise is assumed to be heterogeneous and sufficiently localised with respect to the coordinate $y \in \mathbb R^{d-1}$, i.e., the direction transverse to the propagation of the wave. In particular, 
the problem cannot be reduced to a one-dimensional spatial setting and indeed the transverse dynamics play a key role in our analysis.

Our main result below 
states that any solution
to  \eqref{eq:u} that starts sufficiently close to
\eqref{eq:int:planar:wave} at initial time $t = 0$ will  stay close to this planar wave|with (exponentially) high probability|over timescales that are polynomially ($d  = 2, 3, 4)$ or exponentially ($d \ge 5$) long with respect to $\sigma^{-1}$.   
The main step in the proof is the establishment of moment bounds for the running supremum of several quantities related to the size of the perturbation, 
inspired by a stochastic freezing approach that was originally developed for one-dimensional settings \cite{hamster2019stability,hamster2020diag,hamster2020expstability,hamster2020} and has recently been extended to cylindrical domains \cite{bosch2025conditionalspeedshapecorrections,bosch2024multidimensional}. 
 We focus in our analysis on the equal diffusion setting, thus  taking $D=I_n$ where $I_n$ is the $n\times n$ identity matrix. Nonetheless, we claim that  Theorem \ref{thm:main}  also holds in the case of unequal (positive) diffusion coefficients as long as the deterministic flow does not  experience any transverse instabilities. 



\begin{theorem}[see {\S}\ref{sec:preliminaries}]\label{thm:main}
Suppose $k > d/2+1$.
Under 
technical yet rather mild assumptions,
there are \textnormal(small\textnormal) constants $\kappa_b> 0,$ $\kappa_c> 0,$ $\delta_\eta>0$, and $\delta_\sigma > 0$ so that
for all $0<\eta<\delta_\eta$,  $  0 < \sigma < \delta_\sigma$, and\footnote{We view the parameter $r$ as a means to balance the length of the timescale with the sharpness of the probability bound. Specifically, $r = 1$ provides us with the best bound, whereas the longest timescales arise by taking $r \downarrow 0$ for $2 \le d \le 4$. For $d = 5$ the longest timescale is achieved by taking $r = 1/3$, while for  $d > 5$ taking $r = 1$ leads to a timescale and probability bound which are both optimal.} $0< r\leq 1,$
the probability bound
    \begin{equation}
    \label{eq:int:bnd:p:thm}
        \mathbb P\left(\sup_{0\leq t\leq T(\eta/\sigma^2,r)}\|u(x,y,t)-\Phi_0(x-c_0t)\|_{H^k(\mathbb R^d;\mathbb R^n)}^2>\eta\right)\leq 8\exp\left(-\frac{\kappa_b \eta^r}{4\sigma^{2r}}\right)
    \end{equation}
holds, where the
timescales
are given by
\begin{equation}
    T(\chi,r)=\begin{cases}
      \exp\big( \kappa_c \chi \big),&d> 5,\\
     \exp\big( \kappa_c \chi^{\frac{1}{2}\min\{2/3,1-r\}}\big),&d=5, \\
       \kappa_c (r \chi)^{{2(1-r)}/{(5-d)}},&2\leq d\leq 4,
    \end{cases}\label{eq:timescales}
\end{equation}
 whenever   $u(x,y,0)$  is $\mathcal O(\eta)$-close to $\Phi_0(x)$. 
\end{theorem}

One of the major points in the result above is that
the phase of the planar wave $(\Phi_0, c_0)$ used to track $u(t)$ in \eqref{eq:int:bnd:p:thm} is given simply by 
$\gamma(t)=c_0t$. In particular, no corrections are required to the 
movement of the deterministic wave $(\Phi_0, c_0)$.
This is in stark contrast to the situation encountered in one-dimensional settings and on cylindrical domains, where the tracking phase $\gamma(t)$ satisfies an SDE and can deviate significantly from the deterministic position. In the present multidimensional setting, the stochastic forcing causes \textit{local} $y$-dependent perturbations to the phase, which all spread out in the transverse direction.
Our main task in this paper is to characterise and control these localised perturbations. 



\paragraph{Deterministic setting} To set the stage, we first review the deterministic approach developed by Kapitula in \cite{kapitula1997} to establish the asymptotic stability of planar waves in dimension $d\geq 2,$ improving the prior results in    \cite{levermore1992multidimensional,xin1992multidimensional} where either  $d\geq 4$ or other restrictive assumptions were required. The key idea in \cite{kapitula1997} is to decompose the solution
as
\begin{equation}
u(x+c_0t,y,t)= \Phi_0\big(x-\theta(y,t)\big) + v(x,y,t)
,\label{eq:pert_v}
\end{equation}
thus  splitting  the perturbation into a 
$y$-dependent phase-shift $\theta(\cdot,t)\in  
H^k(\mathbb R^{d-1};\mathbb R)$ and a
remainder term $v(\cdot, \cdot, t) \in H^k(\mathbb R^{d};\mathbb R^n)$ that satisfies $P_{\rm tw} v(\cdot, y, t)  = 0$ for each $y \in \mathbb R^{d-1}$ and $t \ge 0$. Here $P_{\rm tw}$ denotes the spectral projection associated to the neutral   eigenvalue of the linear operator $\mathcal L_{\rm tw}$ that describes the linearisation around the wave  in the $x$-direction only. 
The (local) existence of such a decomposition is assured by the implicit function theorem, but can also be obtained by prescribing appropriate evolution equations for $v(t)$ and $\theta(t)$.

Assuming that the one-dimensional operator $\mathcal{L}_{\rm tw}$ possesses a spectral gap, this decomposition guarantees exponential decay in time for the linear flow of the $v$-component in \eqref{eq:pert_v}.
However, this is not the case for the $\vartheta$-component.
Indeed, no spectral gap is present with regard to the  operator $\mathcal L=\mathcal L_{\rm tw}+\Delta_y$ 
associated to the full spatial problem \cite{kapitula1997,xin1992multidimensional},
where $\Delta_y$ denotes the Laplacian on $\mathbb R^{d-1}$. In particular, the boundary of the spectrum of $\mathcal L$ touches zero in a quadratic tangency, resulting in the algebraic decay rate $\vartheta(t) \sim t^{-(d-1)/4}$. For dimensions $d \in \{2, 3\}$ this decay rate is too slow to close a nonlinear stability argument. The key observation that allows one to circumvent this problem is that the critical terms in the relevant nonlinearities depend solely on $\nabla_y \vartheta$, which decay at the faster rate $\nabla_y \vartheta(t) \sim t^{-(d+1)/4}$. For these lower dimensions, the decomposition \eqref{eq:pert_v} allows  this specific structure of the nonlinearities to be exploited in a crucial fashion.

Although the results in \cite{kapitula1997} are formulated for the equal diffusion setting, the main aspects of this approach also work when the diffusion coefficients are unequal but positive and the spectrum of $\mathcal L$ touches the imaginary axis at the origin only, via a quadratic tangency. See, e.g., \cite{hoffman2015multi}, where anisotropic (discrete) diffusion affects the decay rates of $v$ and $\vartheta$. However, transversal instabilities can occur in various settings, which we discuss further in the outlook below.





\paragraph{Evolution of the perturbation}
We continue to use the decomposition \eqref{eq:pert_v}
for our analysis in this paper. Our It{\^o} computations
in {\S}\ref{sec:evol:pert} lead to the coupled
system\footnote{We are suppressing the corresponding $H^k(\mathbb R^d;\mathbb R^n)$ and $H^k(\mathbb R^{d-1};\mathbb R)$ norms in \eqref{eq:int:coupled:v:theta:ev} for readability purposes.}
\begin{equation}
\label{eq:int:coupled:v:theta:ev}
\begin{aligned}
\mathrm dv & = \big[ \mathcal{L}_{\rm tw} v  + \mathcal O (\sigma^2 + \|v\|^2 +\|v\|\|\theta\|
 + \|\nabla_y \vartheta\|^2  ) \big] \, \mathrm dt +  \mathcal O(  \sigma  ) \,  \mathrm  dW^Q_t ,
 \\[0.2cm]
\mathrm d \vartheta & =  \big[ \Delta_y \vartheta + \mathcal O (\sigma^2 + \|v\|^2 +\|v\|\|\theta\|
 + \|\nabla_y \vartheta\|^2 ) ] \, \mathrm dt
 + \mathcal O( \sigma )  \,\mathrm d W^Q_t,
\end{aligned}
\end{equation}
which we will analyse by passing to a mild formulation and utilising the semigroups $e^{\mathcal L_{\rm tw} t}$
and $e^{\Delta_y t}$. These  decay on the appropriate subspaces\footnote{The  subspaces are $H^k_\perp(\mathbb R^d;\mathbb R^n)=\{v\in H^k(\mathbb R^d;\mathbb R^n):P_{\rm tw}v=0\}$ and $H^k(\mathbb R^{d-1};\mathbb R) \cap L^1(\mathbb R^{d-1};\mathbb R)$, respectively.}  at an exponential and algebraic rate, respectively. 

One of the main technical contributions in this paper is that we obtain precise growth rates for the running suprema of the resulting stochastic convolutions. For example, we have 
\begin{equation}
\label{eq:int:stoch:conv:growth}
    \mathbb E  \sup_{0 \le t \le T}  \left\|\int_0^t e^{\Delta_y (t-s)} B(s) \,\mathrm d W^Q_s \right\|^{2p}_{H^k(\mathbb R^{d-1};\mathbb R)} \sim (p^p + \log T)^p 
     \left[ \int_0^T (1 + T - t)^{-(d-1)/2} \, \mathrm dt \right]^p ,
\end{equation}
for all $p\geq 1,$ in  case   a suitable norm of the integrand $B(t)$ is uniformly bounded.
The decay rate in the final integral is the square of the natural semigroup decay and results into extra growth at the rate
\begin{equation}
\label{eq:int:def:d:stc}
\int_0^T (1 + T - t)^{-(d-1)/2} \, \mathrm dt
\lesssim
\mathfrak d_{\rm stc}(T)=\begin{cases}
   1,&d\geq 4,\\
        \log(T),&d=3,\\
        T^{1/2},&d=2.
    \end{cases}
\end{equation}
In particular, by writing
\begin{equation}\begin{aligned}
    \theta_I(t)= \int_0^t e^{\Delta_y (t-s)}\mathcal O(\sigma)\,\mathrm dW_s^Q,\label{eq:mild:theta}
    \end{aligned}
\end{equation}
we expect the behaviour
\begin{equation}
    \mathbb E \sup_{0 \le t \le T} \| \vartheta_I(t) \|_{H^k(\mathbb R^{d-1};\mathbb R)}^2 \sim  \sigma^2 \mathfrak d_{\rm stc}(T) \log (T).
\end{equation}
This additional growth term does not occur for  stochastic convolutions against the semigroup $e^{\mathcal L_{\rm tw} t}$, which have been analysed extensively in previous works \cite{bosch2025conditionalspeedshapecorrections,bosch2024multidimensional,hamster2020expstability}. In particular, based (solely) on the $\mathcal{O}(\sigma)$ term in the first line of \eqref{eq:int:coupled:v:theta:ev} we expect the behaviour
\begin{equation}
\label{eq:int:exp:beh:v}
    \mathbb E \sup_{0 \le t \le T} \| v(t) \|_{H^k(\mathbb R^d;\mathbb R^n)}^2  \sim  \sigma^2 \log (T).
\end{equation}

Another crucial difference with respect to these prior results is that the deterministic leading order terms will  lead to even stronger growth with respect to $T$, which occurs at the rate
\begin{equation}
\label{eq:int:def:d:det}
    \int_0^T(1+T-t)^{-(d-1)/4}\mathrm dt \lesssim \mathfrak d_{\rm det}(T)=\begin{cases}
   1,&d> 5,\\
        \log(T),&d=5,\\
        T^{(5-d)/4},&2\leq d\leq 4,
    \end{cases}
\end{equation}
corresponding to the heat semigroup. In particular,
writing
\begin{equation}\begin{aligned}
    \theta_{II}(t)=\int_0^te^{\Delta_y (t-s)}\mathcal O(\sigma^2)\,\mathrm ds ,
    \end{aligned}
\end{equation}
we expect the growth behaviour
\begin{equation}
    \mathbb E \sup_{0 \le t \le T} \| \vartheta_{II}(t) \|_{H^k(\mathbb R^{d-1};\mathbb R)}^2 \sim \sigma^4 \mathfrak d_{\rm det}(T)^2. 
\end{equation}
Based on the bounds above, the maximal relevant timescales over which we can anticipate $\vartheta(t)$ and $v(t)$ to remain small are hence determined by
\begin{equation}
\label{eq:int:restr:on:T}
\sigma^2 T^{(5-d)/4} \ll 1 \hbox { for } d \in \{2, 3, 4\},
\qquad \qquad
    \sigma^2 \log (T) \ll 1 \hbox{ for } d \ge 5.
\end{equation}
For $d > 5$ these agree with those in our main theorem.

For $2 \le d \le 5$ however the timescales \eqref{eq:timescales} are shorter, which turns out to be solely attributable to the presence of the mixed term
\begin{equation}\begin{aligned}
    \theta_{III}(t)=\int_0^te^{\Delta_y (t-s)}\mathcal O( \|v\|_{H^k(\mathbb R^d;\mathbb R^n)} \| \vartheta \|_{H^k(\mathbb R^{d-1};\mathbb R)})\,\mathrm ds .
    \end{aligned}
\end{equation}
Imposing the crude a priori bound $\| \vartheta\| \le 1$ 
and using the heuristic \eqref{eq:int:exp:beh:v}
leads to the estimate
\begin{align}
\label{eq:int:crude:est:theta:mix}
\nonumber\textstyle\mathbb E\big[\sup_{0 \le t \le T}  \| \theta_{III}(t)\|^2_{H^k(\mathbb R^{d-1};\mathbb R)} 1_{\|\vartheta(t)\| \le 1}\big]
  &\textstyle \sim  \mathbb E\sup_{0\leq t\leq T}\left[\int_0^t(1+t-s)^{-\frac{d-1}{4}}\| v(s)\|_{H^k(\mathbb R^d;\mathbb R^n)} \,\mathrm ds\right]^2
  \\ 
  & \textstyle\sim  
  \mathfrak d_{\rm det}(T)^2 \mathbb E \sup_{0\leq t\leq T} \|v(t)\|_{H^k(\mathbb R^d;\mathbb R^n)}^2
  \\
  & \textstyle\sim 
     \sigma^2 \mathfrak d_{\rm det}(T)^2 \log (T) .\nonumber
\end{align}
In particular, the associated timescale restrictions are given by
\begin{equation}
\label{eq:int:restr:on:T:mix}
\sigma^2 T^{(5-d)/2} \log(T) \ll 1 \hbox { for } d \in \{2, 3, 4\},
\qquad \qquad
    \sigma^2 \log (T)^3 \ll 1 \hbox{ for } d = 5.
\end{equation}
Observe that this leads directly to the $\sigma$-scalings in \eqref{eq:timescales} upon using the parameter $r > 0$ to absorb the logarithmic term in \eqref{eq:int:restr:on:T:mix} for $d \in \{2, 3, 4\}$ into an additional small power of $T$.

It is important to
note that the timescales obtained in \eqref{eq:timescales} are not a consequence of our choice of the decomposition  \eqref{eq:pert_v}, but are in fact inherent to the original problem. Indeed, exit problems for the pair $(v, \vartheta)$ can be directly related to those for the solution
$u$. In particular, there exists a constant $K > 0$ so that for every sufficiently small $\eta$ the exit event in \eqref{eq:int:bnd:p:thm} is guaranteed to occur whenever the size of $(v, \vartheta)$ crosses through the threshold $K \eta$. There are hence no local mechanisms through which the effect on $u$ of a growing $\vartheta$ can be ``neutralised'' by altering 
$v$.

\paragraph{Numerical toy problem}
\begin{figure}[!t]
    \centering
    \includegraphics[width=0.495\linewidth, trim = 0 0 1cm 0, clip]{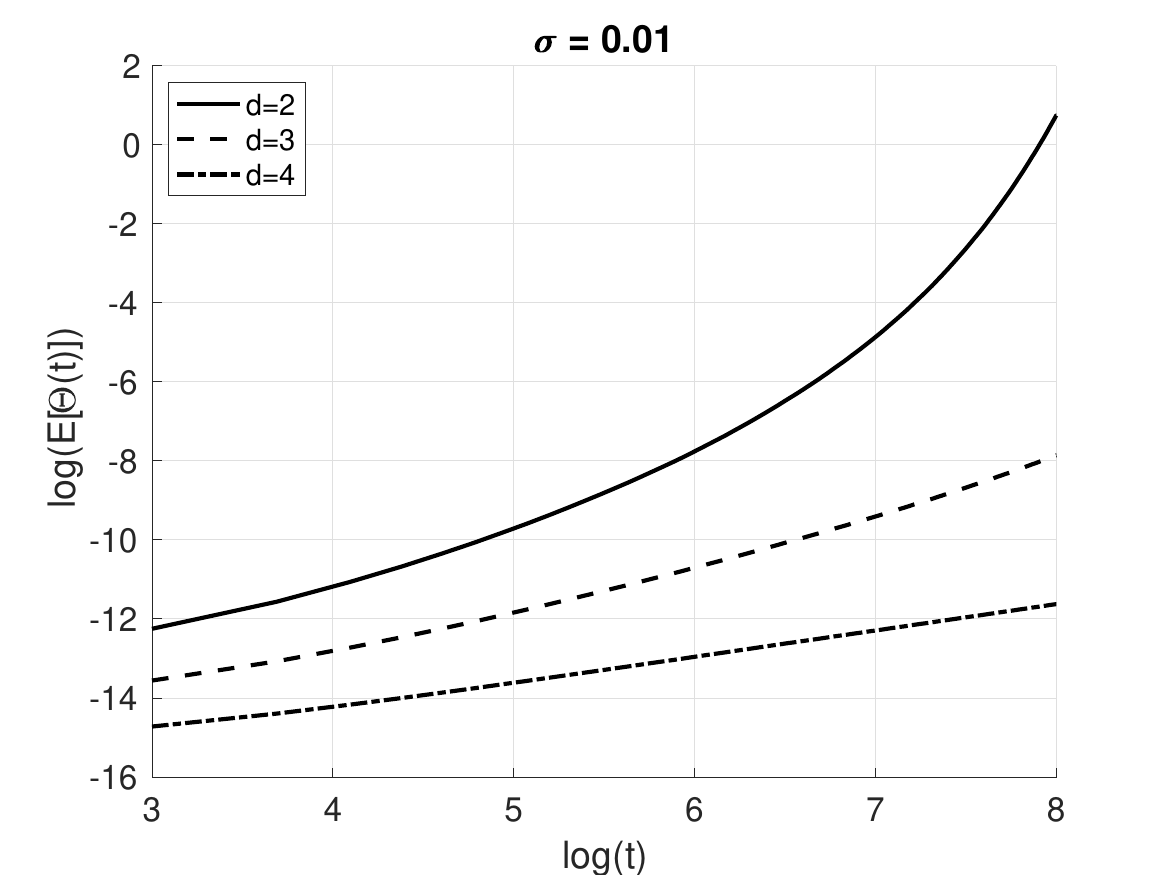}
\includegraphics[width=0.495\linewidth, trim = .5cm 0 .5cm 0, clip]{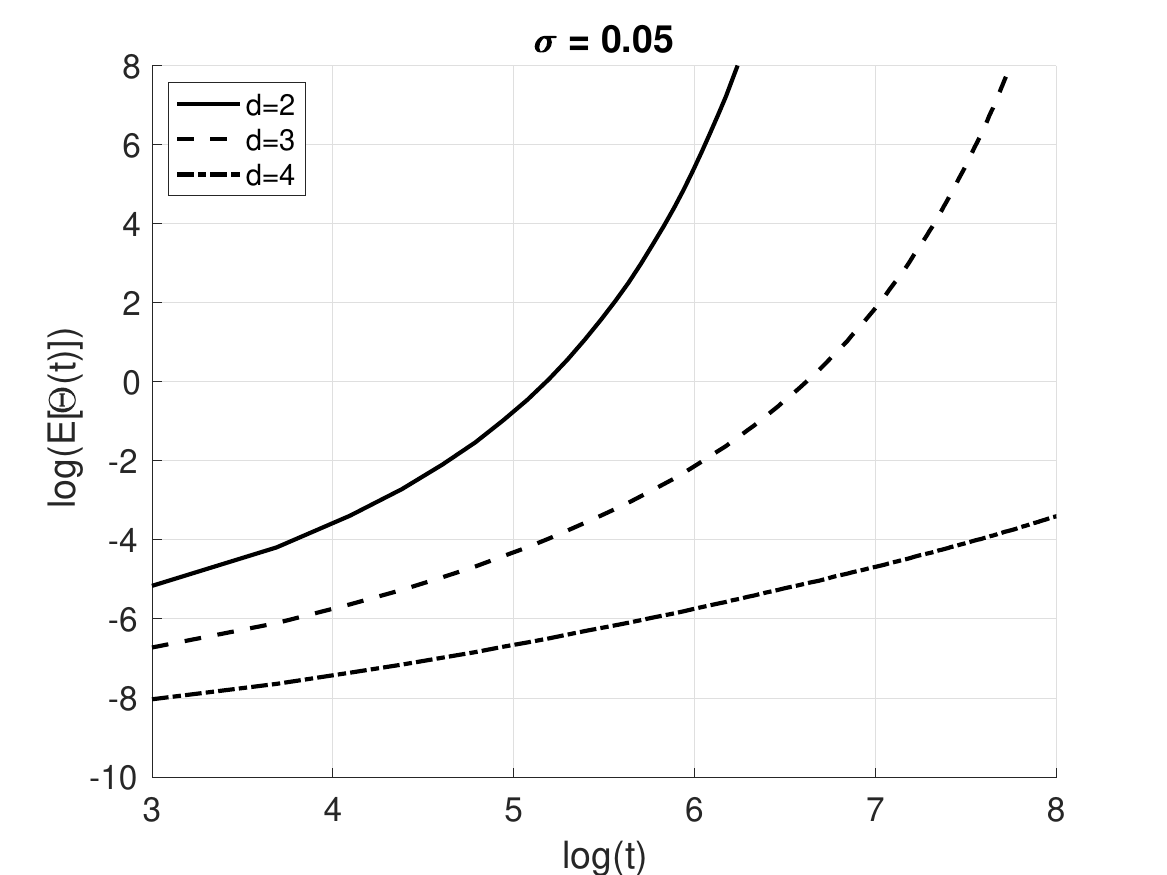}
    \caption{Average over 200 realisations of the quantity $ \Theta(t)$ defined in \eqref{eq:timescale:rel}. Since the curves exhibit super-linear growth when viewed in this double logarithmic scale, these computations suggest that
    $\mathbb  E[\Theta(T)]$ grows faster than any power of $T$.
    }
    \label{fig:theta-hat}
\end{figure}
On account of the rather brute estimate \eqref{eq:int:crude:est:theta:mix}, it is natural to wonder in what sense the timescales \eqref{eq:timescales} are optimal. To this end, we 
capture the essential terms featuring in the discussion above in the two-component SDE
\begin{equation}
    \hat \theta(t):=\int_0^t(1+t-s)^{-\frac{d-1}{4}}\big[\sigma^2+|\hat v(s)|\hat \theta(s)\big]\,\mathrm ds,\qquad     \mathrm d\hat v(t)=- \hat v(t) \,\mathrm dt+\sigma\, \mathrm d\beta(t),\label{eq:mild:scalar}
\end{equation}
which we view as a toy problem for our original physical system. Here
$(\beta(s))_{s\geq 0}$ denotes  a standard Brownian motion.
The process $\hat \theta$ is 
strongly intertwined with
the scalar Ornstein-Uhlenbeck
process 
\begin{equation}
\label{eq:int:def:Y:ou}
\hat v(t)   := \sigma Y_{\rm ou}(t),\qquad Y_{\rm ou}(t) :=  \int_0^t e^{-(t-s)} \, \mathrm d \beta(s),
\end{equation} via a convolution kernel that decays at the same algebraic rate as the heat semigroup.

We proceed  by defining the maximal process 
\begin{equation}
    \textstyle\Theta(t)=\sup_{0\leq s\leq t} |\hat \theta(s)|^2,\label{eq:timescale:rel}
\end{equation}
together with  the on average hitting time
\begin{equation}
   \tau_{\rm avg}(\eta)=\inf\{t\geq 0:\mathbb E[\Theta(t)]>\eta\}.
\end{equation}
The numerical results
in Figure \ref{fig:theta-hat} clearly suggest that
$\mathbb E[\Theta(T)]$ grows at a rate that is faster than any power of $T$, which is significantly more rapid than any of the bounds discussed above. 
In particular, it is essential to use a priori bounds on our variables to obtain slower rates such
as \eqref{eq:int:crude:est:theta:mix}, which we achieve by introducing appropriate stopping times. This can be interpreted as the first estimate in
\eqref{eq:int:crude:est:theta:mix}.
The fits in  Figure \hyperref[fig:growth]{\ref{fig:growth}(b)} confirm that the final two estimates 
in \eqref{eq:int:crude:est:theta:mix} are highly reasonable.

Furthermore, the results in Figure  \ref{fig:datafit} for $d =2$ numerically support the claim that $\tau_{\rm avg}(\eta) \ll \sigma^{-2}$ holds for small $\sigma$. This is compatible with the timescale $T \sim \sigma^{-4/3}$ in our main theorem, but also  
indicates that the timescale $T \sim \sigma^{-8/3}$ arising from \eqref{eq:int:restr:on:T} is indeed too long. In particular, the mix term
in the second line of \eqref{eq:int:coupled:v:theta:ev}
is indeed responsible for the significant and numerically confirmed shortening of the timescales over which stability can be
maintained. 


\begin{figure}[!t]
    \centering
    \includegraphics[width=0.99\linewidth]{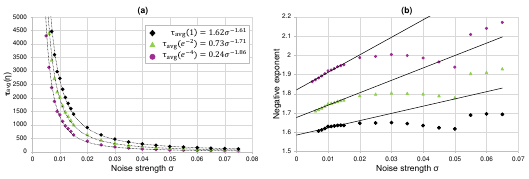}
    \caption{Power law behaviour of $\tau_{\rm avg}(\eta)$ for $d=2$ in terms of $\sigma^{-1}$, obtained by averaging over 200 realisations of $\Theta(t)$ with: $\eta=1$ (diamond); $\eta=\exp(-2)\approx 0.1$ (triangle); and  $\eta=\exp(-4)\approx 0.02$ (circle). \textbf{(a)} We start with $\sigma=0.075$ and decrease $\sigma$ until $\tau_{\rm avg}(\eta)$ is beyond the threshold 5000. A standard regression is performed on all  collected data points for fixed $\eta$. \textbf{(b)} For each $\sigma$, we plot the negative exponent of the power law obtained from a standard regression on the subset of data points in Figure \hyperref[fig:datafit]{\ref{fig:datafit}(a)}
    that lie in in the horizontal range $[\sigma, 0.075]$.
   It would be beneficial to acquire more data with smaller values of $\sigma$ to decrease the spread in the exponents, but the required running times increase prohibitively. Nevertheless, these results provide clear evidence that $\hat \theta(t)$ experiences significant stochastic forcing well before the 
     timescale $t\sim \sigma^{-2}.$ 
    }
    \label{fig:datafit}
\end{figure}


\paragraph{Algebraic decay and stochastic convolutions}
In order to appreciate the growth rate in \eqref{eq:int:stoch:conv:growth}, it is insightful
to consider the scalar process $Y=(Y(t))_{t\geq 0},$ defined by
\begin{equation}
        Y(t)=\int_0^t(1+t-s)^{-(d-1)/4}\mathrm d\beta(s),\label{eq:Y(t)}
\end{equation}
where $(\beta(s))_{s\geq 0}$ is again  a standard Brownian motion. The  chaining  principle by Talagrand  \cite{dirksen2015tail,talagrand2005generic} is a powerful technique to analyse such processes. Indeed, 
exploiting the fact that $Y$ is a Gaussian process, 
it provides estimates of the form
 \begin{align}
    \label{eq:int:dudley}
    \mathbb E\sup_{0\leq t\leq T}|Y(t)|^{2}&\leq  K_{\rm ch}^{2}\left(\int_0^\infty \sqrt{\log(N(T,d_{\rm ch},\nu))}\,\mathrm d\nu\right)^{2},
    \end{align}
for some $K_{\rm ch}>0$. The covering number $N(T, d_{\rm ch}, \nu)$ denotes the minimum number
of intervals of length $\nu$ or less  in the metric $d$ that are required to cover $[0, T]$; the metric $d_{\rm ch}$ is 
given by
\begin{equation}
d_{\rm ch}(t,s)^2= \mathbb E|Y(t)-Y(s)|^2  
=\mathbb E[Y(t)^2]+\mathbb E[Y(s)^2]-2\mathbb E[Y(t)Y(s)].
\end{equation}
 This metric is related
to the Chernoff bound
\begin{equation}
    \mathbb P(|Y(t)-Y(s)|>\theta)\leq 2\exp(-\theta^2/d_{\rm ch}(t,s)^2) 
\end{equation}
and can be evaluated explicitly using the identity
\begin{equation}
   \textstyle \mathbb E[Y(t)Y(s)]=\int_0^{t\wedge s}(1+t-u)^{-(d-1)/4}(1+s-u)^{-(d-1)/4}\,\mathrm du .
\end{equation}

\begin{figure}[!t]
    \centering
\includegraphics[width=0.495\linewidth, trim = 0 0 1cm 0, clip]{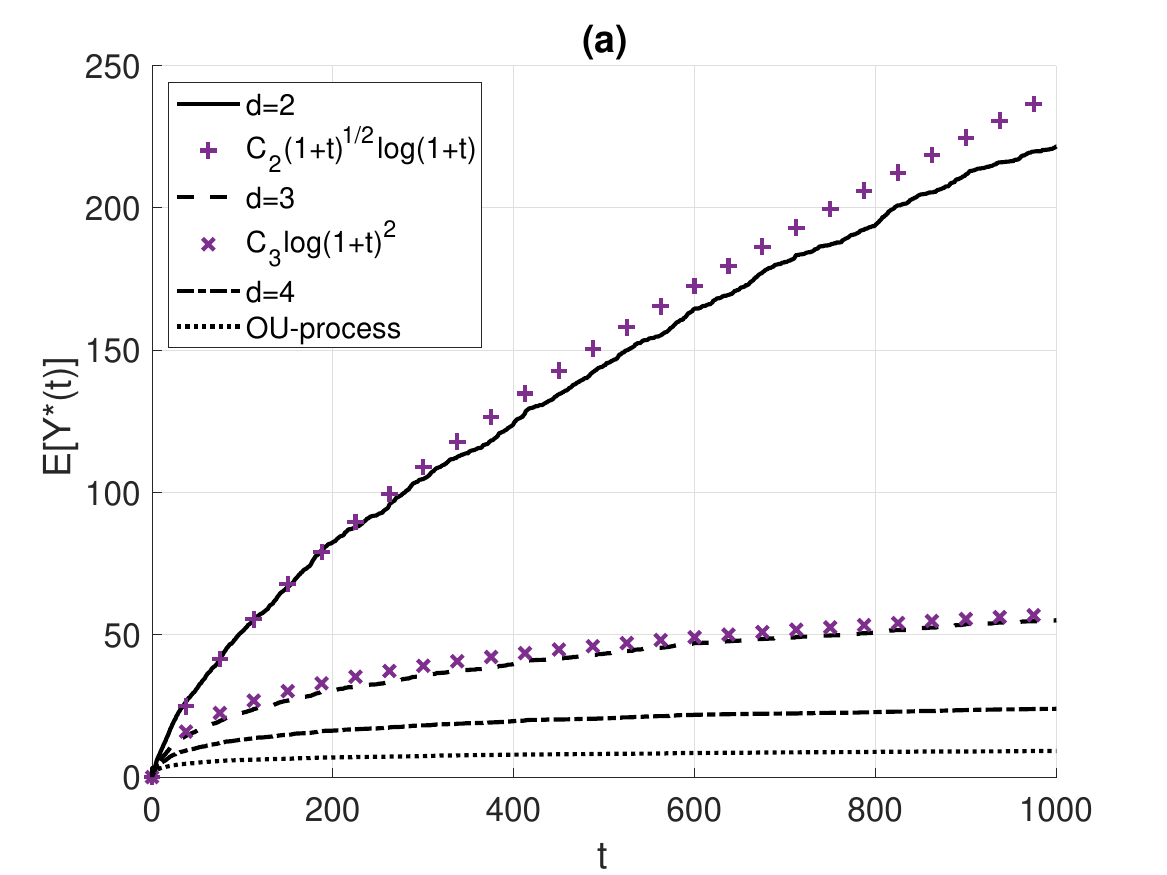}
\includegraphics[width=0.495\linewidth, trim = .5cm 0 .5cm 0, clip]{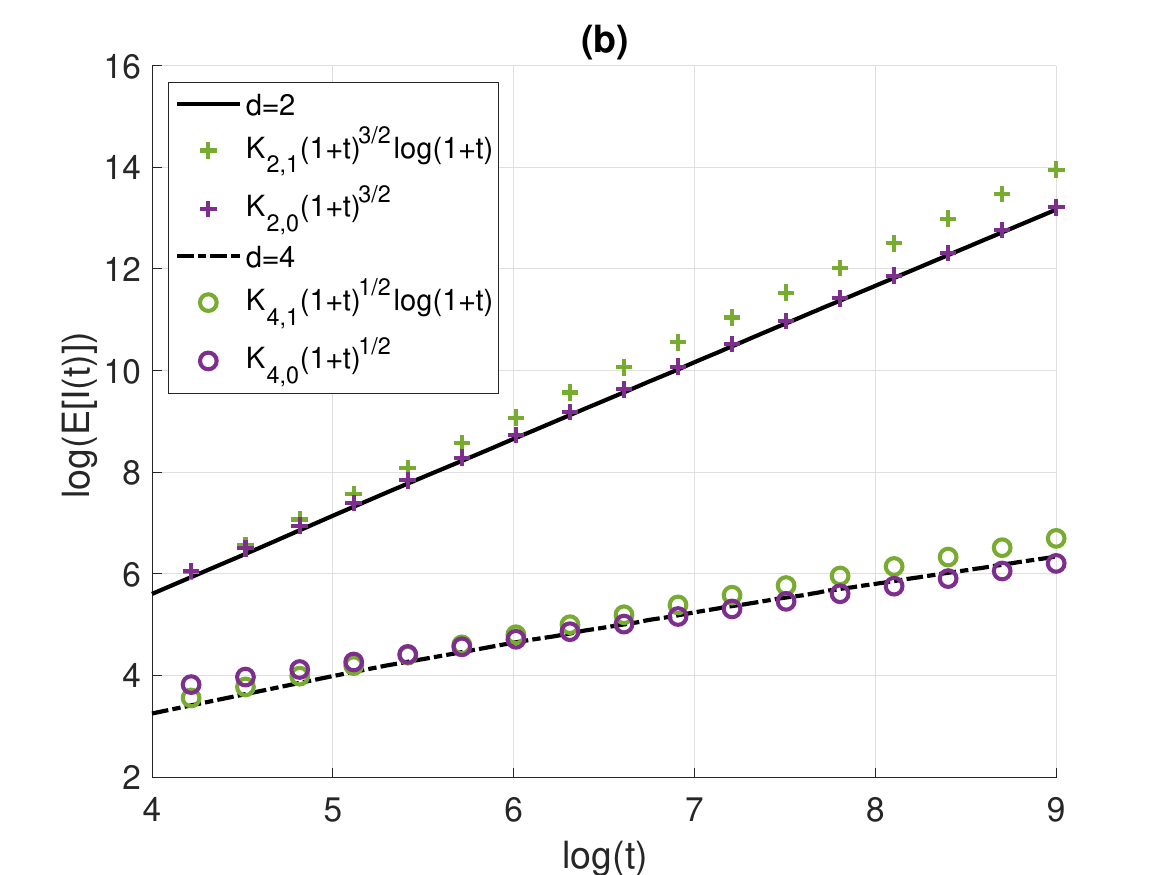}
    \caption{On the growth of stochastic convolutions and weighted integrals. \textbf{(a)} Average over 200 realisations of $Y^*(t)=\sup_{0\leq s\leq t} Y(s)^2$ 
    and $Y^*_{\rm ou}(t)=\sup_{0\leq s\leq t} Y_{\rm ou}(s)^2$;
    see \eqref{eq:Y(t)} and \eqref{eq:int:def:Y:ou}.
    The purple markers represent the estimates
    \eqref{eq:int:stoch:conv:growth} with associated constants $C_2=1.1$ and $C_3=1.2$.
     \textbf{(b)} Average over 200 realisations of $I(t)=\sup_{0\leq s\leq t}[\int_0^s (1+s-r)^{-\frac{d-1}{4}}|Y_{\rm ou}(r)|\,\mathrm dr]^2$.
     The green markers represent the $T$-dependence in 
     \eqref{eq:int:crude:est:theta:mix} with associated constants $K_{2,1}=0.175$ and $K_{4,1}=1.0$.
     The purple markers omit the logarithmic terms and feature the constants $K_{2,0}=0.75$ and $K_{4,0}=5.5.$ Notice that the latter do not appear to fully capture the growth of $I(t)$, showing that the logarithmic factors are indeed essential.
     }\label{fig:growth}
\end{figure}

For the odd cases $d \in \{3, 5\}$ this leads to insightful and tractable bounds for $d_{\rm ch}$ that can be used to estimate
\eqref{eq:int:dudley}. Indeed, picking $d = 3$ and assuming $t > s$,
we may compute
\begin{equation}
\begin{aligned}
d_{\rm ch}(t,s)^2
&=\log(t+1)+\log(s+1)-2\log\left(\frac{2+t-s-2\sqrt{1+t-s}}{2+t+s-2\sqrt{(1+t)(1+s)}}\right)\\
&\leq \log(1+\tfrac12(t-s)).
    \end{aligned}
\end{equation}
For $|t-s|\leq 2,$ this implies 
    $d_{\rm ch}(t,s)^2\leq \frac12|t-s|, $
while for $|t-s|\geq 2$ we have  $
    d_{\rm ch}(t,s)^2\leq \log(t-s).$
Subsequently, we see that the covering number $N(T, d_{\rm ch}, \nu)$
can be bounded
by $T/2\nu^2$ for $0<\nu\leq 1,$ by $T/\exp({\nu^2})$ for $1\leq \nu\leq \sqrt{\log(T)},$ and by 1 for $\nu\geq\sqrt{\log(T)}.$
This yields
    \begin{equation}
        \begin{aligned}\int_0^\infty \sqrt{\log(N(T,d_{\rm ch},\nu))}\,\mathrm d\nu&\leq  \int_0^1\sqrt{\log(T)-\log(2\nu^2)}\,\mathrm d\nu+ \int_1^{\sqrt{\log(T)}}\sqrt{\log(T)-\nu^2}\,\mathrm d\nu
    \\& \leq 2\sqrt{\log(T)}+\frac{\pi}{4}\log(T)
    \\& \leq 4\log(T),
\end{aligned}
    \end{equation}
    for $T\geq 2$, which is in line with \eqref{eq:int:stoch:conv:growth} for $p=1$.  
    
On the other hand, for    $d=5$  we similarly get 
     \begin{equation}
        \begin{aligned}
            d_{\rm ch}(t,s)^2&=\frac{t}{t+1}+\frac{s}{s+1}-2(t-s)^{-1}{\log\left((1+t-s)\frac{1+s}{1+t}\right)} \\
            &\leq \min\{|t-s|,2\},
        \end{aligned}
    \end{equation}
    which agrees (up to a constant) with the bound that can be obtained for the Ornstein--Uhlenbeck process $Y_{\rm ou}$; see
    e.g., \cite[Sec. 1]{hamster2020expstability}.
    In particular, $Y$ and $Y_{\rm ou}$ admit the same logarithmic growth rate
    \begin{equation}
    \mathbb E \sup_{0 \le t \le T} Y(t)^2 \sim \mathbb E \sup_{0 \le t \le T} Y_{\rm ou}(t)^2 \sim \log( T),
    \end{equation}
    which again agrees with \eqref{eq:int:stoch:conv:growth} for $p=1$.
    For $Y_{\rm ou}$  this
    can alternatively be established by
    examining crossing numbers \cite{pickands1969asymptotic} or analysing explicit probability distributions \cite{alili2005representations}; see also \cite{jia2020moderate}.
    

    
    
    
    

    In principle, these Dudley-type bounds can also be obtained for infinite-dimensional stochastic integrals
    such as \eqref{eq:int:stoch:conv:growth}
    by following the approach in \cite{bosch2024multidimensional,hamster2020expstability}. Nevertheless, in {\S}\ref{sec:supremum} we proceed with the more streamlined approach that was developed in \cite{bosch2025conditionalspeedshapecorrections},
    which leads to the same growth bounds. We refer
    to Figure \hyperref[fig:growth]{3(a)} for numerical 
    evidence in support of these results.
    

    \paragraph{Moments of maximal processes} 
    Recalling the decomposition \eqref{eq:pert_v},
    our main objective is to control the size of the quantity
    \begin{equation}
\begin{aligned}
    N_{k}(t)&=\|v(t)\|_{H^k(\mathbb R^d;\mathbb R^n)}^2+\|\theta(t)\|_{H^k(\mathbb R^{d-1};\mathbb R)}^2+\int_0^t(1+t-s)^{-(d-1)/4}\|v(s)\|_{H^{k+1}(\mathbb R^d;\mathbb R^n)}^2\mathrm ds\\&\qquad+\int_0^t(1+t-s)^{-(d-1)/4}\|v(s)\|_{H^{k}(\mathbb R^{d};\mathbb R^n)}\|\theta(s)\|_{H^{k}(\mathbb R^{d-1};\mathbb R)}\mathrm ds\\&\qquad\qquad+\int_0^t(1+t-s)^{-(d-1)/4}\|\nabla _y\theta(s)\|_{H^{k}(\mathbb R^{d-1};\mathbb R)}^2\mathrm ds,\label{eq:size:main}
    \end{aligned}
\end{equation}
which will allow us to study the related stopping time
\begin{equation}
t_{\rm st}(\eta;k)=\inf\{t\geq 0:N_{k}(t)>\eta\}.
\end{equation}

Let us stress the fact that  the second and third integral  terms in \eqref{eq:size:main} are  necessary, as these nontrivial quadratic terms appear explicitly in \eqref{eq:int:coupled:v:theta:ev}.  Furthermore, the $H^k$-control on  $\nabla_y \theta$  combined with the $H^{k+1}$-control on $v$ is required to prevent the blow-up of solutions
to \eqref{eq:int:coupled:v:theta:ev}, for which only local existence results are available; see \cite{agresti2024criticalNEW} and {\S}\ref{sec:evol:pert}. 
Generally speaking, the presence of the integrals in \eqref{eq:size:main} is related to the fact that the natural $H^k \to H^{k+1}$ semigroup bounds have a $\mathcal O(t^{-1/2})$-singularity for $0 < t \ll 1$ that is not integrable when squared, as is required when analysing the stochastic integrals. In {\S}\ref{sec:estimates:max:reg} we obtain several maximal regularity estimates where these weighted integrals are essential to remove the singularities. We emphasise that we cannot replace the $\| \nabla \theta\|_{H^k}$ term by  $\| \theta\|_{H^{k+1}}$,  since   estimates in dimensions 2 and 3 would then fail.
For the lower dimensions $2\leq d \le 5$, it is worth observing that even $H^{k+1}$-supremum bounds on $v$ and $\vartheta$ would not be enough to obtain $t$-uniform control on these integrals on account of the growth in \eqref{eq:int:def:d:det}.

In {\S}\ref{sec:stability} 
we establish the moment bound stated in Theorem \ref{thm:general}, which can be considered as the main result of this paper, and reads
 \begin{equation}
 \begin{aligned}\label{eq:24}
     \mathbb E\left[\sup_{0\leq t\leq t_{\rm st}(\eta;k)\wedge T}|N_{k}(t)|^p\right]&\leq K^p\Big[\|v(0)\|_{H^{k}(\mathbb R^{d};\mathbb R^n)}^{2p}+\|\theta(0)\|_{H^{k}(\mathbb R^{d-1};\mathbb R)}^{2p}\\&\qquad\qquad\qquad\qquad  
     +\sigma^{2p}\mathfrak d_{\rm det}(T)^{2p}(p^p+\log(T)^p)\Big],
     \end{aligned}
 \end{equation}
 for all integers $p \ge 1$.
 By exploiting Markov-type inequalities  in \cite{bosch2025conditionalspeedshapecorrections,bosch2024multidimensional,hamster2020expstability}, we can show that the exit probability increases  logarithmically for $d\geq 5$ and (at least) polynomially in time for $2\leq d\leq 4$, hence resulting into Theorem \ref{thm:main}.

    \paragraph{Outlook} 
    The main issue in dimensions $d \in\{2, 3, 4\}$ that limits the timescale over which stability can be maintained is that the phase $\vartheta(t)$ experiences|at lowest order in $\sigma$|deterministic growth.
    In principle, this lowest order behaviour leads to the formation of a localised deformed profile that can be explicitly computed. A possible next step is to linearise around this profile rather than the (rigid) planar wave \eqref{eq:int:planar:wave}. Naturally, this spawns many additional (time-dependent) terms which could be problematic to control, but it has the potential to open up the route towards longer timescales.  One could also attempt to incorporate weights into quantities like \eqref{eq:size:main} to study how much faster $\theta(t)$\linebreak grows compared to $v(t)$.
    

Returning to the setting of unequal diffusion coefficients, we are especially interested in the appearance of transverse instabilities \cite{hagberg1994complex}, which are attracting major interest since they
can lead to intriguing phenomena such as fingering \cite{carter2024stabilizing,carter2023criteria}.
Near the onset of these instabilities we expect that stochastic fluctuations could play an important role to (de)stabilise the underlying system \cite{appleby2008stabilization,mohammed2014multi,salins2023solutions}. In particular, 
bifurcations may now occur in different regions of the parameter space, which could lead to novel insights. Investigating this could be done either on the whole space 
$\mathbb R^d$
  or on a cylindrical domain, where the size of the torus may serve as an additional bifurcation parameter. 
The above connects to exit problems \cite{salins2023solutions}   and also 
has possible applications in optimal control \cite{de2023stochastic}.
Ultimately, one could extend this study to curved fronts or other  interfaces \cite{berestycki2002influence,funaki1999singular,guo2021curved,hagberg1994complex,niu2018curved,wang2012traveling}.

    \paragraph{Organisation} 
    We start in  {\S}\ref{sec:preliminaries} by stating our assumptions regarding the 
     nonlinearities $f$ and $g$ and the (spectral) properties of the planar  wave \eqref{eq:int:planar:wave}. We particularly focus on giving a rigorous definition for solutions to system \eqref{eq:u}, which is followed by our main result Theorem \ref{thm:general}.
    In {\S}\ref{sec:phase} we introduce our local phase tracking mechanism and provide the  relevant growth and Lipschitz estimates. In {\S}\ref{sec:evol:pert} we prove existence and uniqueness of the  perturbation system \eqref{eq:int:coupled:v:theta:ev} and relate this back to the original system. Working towards the nonlinear stability proof, we establish  growth rates for  the running supremum of  stochastic convolutions with respect to  the heat semigroup $e^{\Delta_yt}$ and $e^{(\mathcal L_{\rm tw}+\Delta_y)t}$ in {\S}\ref{sec:supremum}. In addition, we deduce  maximal regularity estimates to control polynomially weighted integrals
    such as those that appear in \eqref{eq:size:main}.
    As a final step, we derive in {\S}\ref{sec:det:bounds:critical} bounds for the corresponding deterministic convolutions,
    which are essential for the nonlinear stability proof that we provide in {\S}\ref{sec:stability}. 


\section{Main results}\label{sec:preliminaries}
The results in this paper concern the SPDE
\begin{equation}
    \mathrm du=[\Delta u+f(u)]\mathrm dt+\sigma g(u)\mathrm dW_t^Q,\label{eq:reaction:diff}
\end{equation}
which we consider in the Cartesian coordinate system $(x,y)\in\mathbb R\times \mathbb R^{d-1}=\mathbb R^d$ for some integer $d \ge 2$. We denote by $\mathrm dx$   the Lebesgue measure on $\mathbb R$ with respect to the first coordinate and $\mathrm dy$ the Lebesgue measure on $\mathbb R^{d-1}$ with respect to the transverse direction $y.$ Further, set $\mathrm d\mathbf x=\mathrm dx\mathrm dy$. We shall encounter the following  spaces and utilise the shorthands
\begin{equation}
    L^p_x= L^p(\mathbb R,\mathrm dx)^n=L^p(\mathbb R;\mathbb R^n),\quad L^p= L^p(\mathbb R^{d},\mathrm d\mathbf x)^n=L^p(\mathbb R^d;\mathbb R^n),\qquad 1\leq p\leq \infty,
\end{equation}
together with 
\begin{equation}
    L^p_y= L^p(\mathbb R^{d-1},\mathrm dy)=L^p(\mathbb R^{d-1};
    \mathbb R),\qquad  1\leq p\leq \infty,
\end{equation}
and likewise we define the Sobolev spaces $H^k_x$, $H^k$, and $H^k_y.$ Lastly, we also introduce the right-shift operators
\begin{equation}
    [T_{\varrho}z](x,y)=z(x-\varrho, y),  \qquad  \varrho \in \mathbb R,
\end{equation}
which hence act only on the $x$-coordinate.

We outline our main assumptions in {\S}\ref{sec:assumptions} for the deterministic terms and in {\S}\ref{subsec:mr:st} for the stochastic terms. Subsequently, we formulate our main results regarding existence and uniqueness of solutions in {\S}\ref{sec:character} which is followed by the persistence and  metastability of planar travelling waves in {\S}\ref{sec:main:result}.

\subsection{Deterministic assumptions}\label{sec:assumptions}
For our purposes here it is sufficient to impose local Lipschitz conditions on the nonlinearity $f$  and its derivatives. The quantities $u_\pm$
appearing in \eqref{eq:mr:f:vanish:u:pm}
will correspond to the limits of the waveprofile $\Phi_0$. In addition, the integer $k \ge 0$ is a smoothness 
parameter whose value will become clear from the context.

\begin{itemize}
    \item[(Hf)]  We have  $f\in C^{k+2}(\mathbb R^n;\mathbb R^n)$  
    with 
    \begin{equation}
    \label{eq:mr:f:vanish:u:pm}
        f(u_-)=f(u_+)=0
    \end{equation}
    for some pair $u_\pm\in\mathbb R^n.$ In addition, for any $N>0$ there is a  constant $K_f^N>0$  such that  
    \begin{equation}
        |f(u_A)-f(u_B)|+\ldots+|D^{k+2}f(u_A)-D^{k+2}f(u_B)|\leq K_f^N|u_A-u_B|\label{eq:mr:lipschitz:f}
    \end{equation} 
     for all  $u_A,u_B\in\mathbb R^{n}$ with $|u_A|\leq N$ and $|u_B|\leq N$.
\end{itemize}
We remark that in the prior works \cite{bosch2025conditionalspeedshapecorrections,bosch2024multidimensional,hamster2019stability,hamster2020diag,hamster2020expstability,hamster2020} global Lipschitz conditions
were imposed on the nonlinearities. This distinction is somewhat cosmetic, because our metastability results provide pointwise control on our solutions which means that we can safely modify the nonlinearities outside the region of interest to enforce global Lipschitz bounds.

Travelling wave solutions $u = \Phi_0(x - c_0 t)$ to system \eqref{eq:reaction:diff} 
with $\sigma = 0$
must satisfy the travelling wave ODE
\begin{equation}
    \label{eq:mr:trv:wave:ode}
    \Phi_0''+c_0\Phi_0'+f(\Phi_0)=0.
    \end{equation}
Linearising  \eqref{eq:mr:trv:wave:ode}  around the  wave $(\Phi_0,c_0)$ leads to the linear operator
\begin{equation}
\label{eq:mr:def:l:tw}
    \mathcal L_{\rm tw}:H^2_x\to L^2_x,\quad 
    [\mathcal L_{\rm tw}u](x)=u''(x)+c_0u'(x)+Df(\Phi_0(x))u(x),
\end{equation}
which has the standard kernel element $\mathcal L_{\rm tw}\Phi_0'=0$ associated to translational invariance. Its associated adjoint operator is given by
\begin{equation}
    \mathcal L_{\rm tw}^{\rm adj}:H^2_x\to L^2_x,\quad
    [\mathcal L_{\rm tw}^{\rm adj}w](x)=w''(x)-c_0w'(x)+Df(\Phi_0(x))^\top w(x),
\end{equation}
and it is easily verified that $\langle \mathcal L_{\rm tw}v,w\rangle_{L^2_x}=\langle v,\mathcal L_{\rm tw}^{\rm adj}w\rangle_{L^2_x}$ holds for $v,w\in H^2_x.$ 

We continue with a standard existence and spectral stability condition 
for the wave $(c_0, \Phi_0)$. 
In particular, we assume that $\Phi_0$ approaches its limits $u_\pm \in \mathbb R^n$ at an exponential rate;
a common assumption connected to asymptotic hyperbolicity \cite{kapitula2013spectral,sandstede2002stability} which holds in many applications. 
Further, we demand that the translational eigenvalue for $\mathcal{L}_{\rm tw}$ at zero is an isolated simple eigenvalue and that the remainder of the spectrum can be strictly bounded to the left of the imaginary axis, i.e., there is a spectral gap. 

\begin{itemize}
    \item[(HTw)] There exists a waveprofile $\Phi_0 \in C^{2}(\mathbb R;\mathbb R^n)$  and a wavespeed $c_0\in\mathbb R$ 
    that satisfy \eqref{eq:mr:trv:wave:ode}.
    In addition,     
    there is a constant $K>0$ together with exponents $\nu_\pm>0$ so that the bound
    \begin{equation}
        |\Phi_0(x)-u_-|+|\Phi_0'(\xi)|\leq Ke^{-\nu_-|x|}
    \end{equation}
    holds for all $x\leq 0,$ whereas the bound
    \begin{equation}
        |\Phi_0(x)-u_+|+|\Phi_0'(\xi)|\leq Ke^{-\nu_+|x|}
    \end{equation}
    holds for all $x\geq 0.$  Finally,
the linear operator $\mathcal L_{\rm tw}$ defined in \eqref{eq:mr:def:l:tw} has a simple eigenvalue at $\lambda=0$ and there exists a constant $\beta>0$ so that  $\mathcal L_{\rm tw}-\lambda$ is invertible for all $\lambda \in \mathbb C$ satisfying $\rm Re\, \lambda\geq -2\beta$. 
\end{itemize}

Assuming (Hf) and (HTw), 
we point out as in \cite{bosch2024multidimensional} that the operators $\mathcal L_{\rm tw}$ and $\mathcal L_{\rm tw}^{\rm adj}$ map  $H^{\ell+2}_x$ into $H^\ell_x$, for any  $0 \le \ell \le k+1 $.  The resolvent set of $\mathcal L_{\rm tw}$ restricted to the domain $H^{\ell+2}_x$  contains the resolvent set of $\mathcal L_{\rm tw}$ seen as the original operator with domain $H^2_x.$  Consequently, writing $P_{\mathrm{tw}}:H^k_x\to H^k_x$ for the spectral projection \cite{hamster2020expstability,lunardi2004linear} onto this neutral eigenvalue for $\mathcal{L}_{\mathrm{tw}}:H^{k+2}_x\to H^k$, we  obtain the identifications
\begin{equation}
\operatorname{ker}\left(\mathcal{L}_{\mathrm{tw}}\right)=\operatorname{span}\left\{\Phi_0^{\prime}\right\}\subset H_x^{k+3}, \quad \operatorname{ker}(\mathcal{L}_{\mathrm{tw}}^{\rm adj})=\operatorname{span}\left\{\psi_{\mathrm{tw}}\right\}\subset H_x^{k+3}, \quad P_{\mathrm{tw}} v=\left\langle v, \psi_{\mathrm{tw}}\right\rangle_{L^2_x} \Phi_0^{\prime},
\end{equation}
for some $\psi_{\rm tw}$ such that  
$\left\langle\Phi_0^{\prime}, \psi_{\mathrm{tw}}\right\rangle_{L^2_x}=1.$
In particular, we have  $|\Phi_0^{(\ell+1)}(x)|\to 0$  
 and $|\psi_{\rm tw}^{(\ell)}(x)|\to 0$ exponentially fast  as $|x|\to \infty$ for any $0\leq \ell\leq k+3$.  

Since $\mathcal L_{\rm tw}$ is a lower order perturbation of the  diffusion operator $\partial_x^2$, we see that 
$\mathcal L_{\rm tw}$ is sectorial in $L^2_x$ and $H^k_x$ (the latter  after restriction), and hence generates analytic semigroups
on these spaces, which we all denote by $(S_{\rm tw}(t))_{t\geq 0}$ since they agree where they overlap \cite{engel2000one,lunardi2018interpolation,book:Triebel}. Furthermore, we introduce  the linear operator
\begin{equation}\mathcal L:H^{2}\to L^2,\quad [\mathcal Lv](x,y)= [\mathcal L_{\rm tw}v(\cdot,y)](x)+[\Delta_y v(x,
    \cdot)](y),
    \end{equation}
where $\Delta_y$ denotes the Laplacian on $\mathbb R^{d-1}.$  Again, $\mathcal L$ is sectorial in both $L^2$ and $H^k,$ and its analytic semigroup (possibly after restriction) is denoted by $(S_L(t))_{t\geq 0}.$
In contrast to what is done in \cite{bosch2025conditionalspeedshapecorrections,bosch2024multidimensional}, yet in line with \cite{kapitula1997}, we trivially extend the projection operator $P_{\rm tw}$ into the transverse direction. We denote this operator again by $P_{\rm tw}$,  which acts on elements $v(x,y)$  on the first coordinate only, i.e.,
$[P_{\rm tw}v](x,y)=\langle v(\cdot, y),\psi_{\rm tw}\rangle_{L^2_x}\Phi_0'(x)$. In particular, $P_{\rm tw}$ is a bounded operator on $H^k.$    
We refer to  Appendix \ref{sec:prelim} for further properties of these spatially extended operators and their semigroups, including essential decay estimates. 



\subsection{Stochastic assumptions} 
\label{subsec:mr:st}
Let $\mathcal W$ be a real separable Hilbert space   with orthonormal basis $
(e_j)_{j\geq 0}$ and assume $Q\in\mathscr L(\mathcal W)$ to be a non-negative symmetric
operator. 
Consider the Hilbert space $\mathcal W_Q=Q^{1/2}(\mathcal W)$ endowed with its  natural inner product
 \begin{equation}
     \langle v,w\rangle_{\mathcal W_Q}=\langle Q^{-1/2}v,Q^{-1/2}w\rangle_{\mathcal W},
 \end{equation}
 which has $(\sqrt Qe_j)_{j\geq 0}$ as an orthonormal basis. Throughout this work, we neglect any possible zero element of the set $(\sqrt Qe_j)_{j\geq 0}$ caused by the fact that $Q$ is only a non-negative and not a positive operator.
 Following \cite{da2014stochastic,gawarecki2010stochastic,hairer2009introduction,hamster2020,karczewska2005stochastic,prevot2007concise}, we let $(\Omega,\mathcal F,\mathbb F,\mathbb P)$ be a filtered probability space
 and set to construct a cylindrical $Q$-Wiener process $W^Q=(W^Q_t)_{t\geq 0}$ that is adapted to the filtration $\mathbb F$. We consider a set $(\beta_j)_{j\geq 0}$ of independent standard Brownian motions adapted to $\mathbb F$ and write
 \begin{equation}
W_t^Q=\sum_{j=0}^\infty\sqrt{Q}e_j\beta_j(t),
\end{equation}
which converges in $L^2(\Omega;\mathcal{W}_{\rm ext})$ for some larger (abstract) space $\mathcal{W}\subset \mathcal{W}_{\rm ext}$ that is guaranteed to exist by the discussion in \cite{hairer2009introduction}; 
see also \cite[Sec. 5.1]{hamster2020} for  additional background information.

 For any Hilbert space $\mathcal{H}$ and $p \ge 2$,
we introduce the class of processes
\begin{equation}
\begin{aligned}
    \mathcal N^p([0,T];\mathbb F;HS(\mathcal W_Q;\mathcal{H}))&=\{B\in L^p(\Omega ; L^2([0, T] ; HS(\mathcal W_Q; \mathcal{H}))) :\\&\quad\quad\quad\quad\quad B\text{ has a progressively measurable version}\},
    \end{aligned}
\end{equation} 
 for which It\^o integrals with respect to $W^Q$ can be defined. Here the space $HS(\mathcal W_Q;\mathcal H)$ denotes the space of operators $B:\mathcal W_Q\to \mathcal H$ that are Hilbert--Schmidt, i.e., for which
\begin{equation}
    \|B\|_{HS(\mathcal W_Q;\mathcal H)}^2=\sum_{j=0}^\infty \|B[\sqrt Qe_j]\|_{\mathcal H}^2<\infty.
\end{equation}
Notice that $\|\cdot \|_{HS(\mathcal W_Q;\mathcal H)}$ is  independent of the orthonormal basis $(e_j)_{j\geq 0}$ and defines a norm on $HS(\mathcal W_Q;\mathcal H)$.
Lastly, 
we have the identity
\begin{equation}
    \int_0^tB(s)\mathrm dW_s^Q=\lim_{n\to\infty}\sum_{j=0}^n \int_0^t B(s)[\sqrt Qe_j]\mathrm d\beta_j(s),
\end{equation}
where the convergence is in $L^p(\Omega;\mathcal{H })$ \cite{karczewska2005stochastic}. 

We are now ready to   formulate an abstract condition to ensure that the stochastic terms in  \eqref{eq:u} and \eqref{eq:int:coupled:v:theta:ev} are well-defined. 
  That is, we make  sure that $Q$ is a covariance operator and that
  \begin{equation}
      g(u):\mathcal W_Q\to H^k,\quad \xi\mapsto g(u)[\xi],
  \end{equation}
  has the desired Hilbert--Schmidt properties for all $u$ in the affine spaces
\begin{equation}
    \mathcal U_{H^k}=\Phi_{\rm ref}+H^k,
\end{equation}
where we simply take $\Phi_{\rm ref}=\Phi_0$ as our reference function   for convenience.

\begin{itemize}
    \item[(HgQ)] The operator $Q\in \mathscr L (\mathcal W)$ is  non-negative and symmetric. The commutation relation 
     $    T_{\gamma}g(u)[\xi]=g(T_\gamma u)[T_{\gamma}\xi]$  
    holds for all $\gamma\in \mathbb R$ and $u\in\mathcal U_{H^k}$. In addition,  for any $N>0$  there exists a constant $K_{g}^N>0$ so that the bounds 
    \begin{equation}
        \|g(T_\gamma \Phi_0+v)\|_{HS(\mathcal W_Q;H^k)}\leq K_{g}^N \label{eq:Hg}
    \end{equation}
    and
    \begin{equation}
        \|g(T_\gamma \Phi_0+v)-g(T_\gamma \Phi_0+w)\|_{HS(\mathcal W_Q;H^k)}\leq K_{g}^N\|v-w\|_{H^k}\label{eq:HgLip}
    \end{equation}
    hold for all $\gamma \in \mathbb R$ and
    each pair $v,w\in H^k$ that satisfies $\|v\|_{H^k}\leq N$ and $\|w\|_{H^k}\leq N.$ 
\end{itemize}
Besides this $H^k$-based condition, we also need
to impose an integrability condition that will allow us to obtain the $L^1$-estimate in Proposition \ref{prop:nl:bnds}. The latter is necessary for the  stability analysis as it will enable us to exploit
the decay of the heat semigroup; see Corollary \ref{cor:algebraic:decay}.
\begin{itemize}
\item[(H$L^1$)] For any $N>0$ there exists a constant $K_{\rm int}^N>0$ for which we have the bound
     \begin{equation}
        \sum_{j=0}^\infty \|g(T_{\gamma}\Phi_0+v)[\sqrt Q{e_j}]^{\color{black}\top}T_{\gamma}\psi_{\rm tw}\|^2_{L^1\color{black}{(\mathbb R^d;\mathbb R)}}<K_{\rm int}^N,\label{eq:int:cond}
    \end{equation}
  for all $\gamma\in \mathbb R$, each  $v\in H^k$ satisfying $\|v\|_{H^k}\leq N,$ and  any\footnote{Actually, for our intents and purposes, it suffices to have this condition for at least one orthonormal basis, but we prefer to work with estimates that do not depend on the choice of the basis.} 
  orthonormal basis $(e_j)_{j\geq 0}$ of $\mathcal W$.
\end{itemize}

 Note that the value of \eqref{eq:int:cond} is basis dependent and may as well even be infinite. We introduce therefore for any Hilbert space $\mathcal X$ and Banach space $\mathcal Y$ the space of $2$-summing operators $\Pi_2(\mathcal X,\mathcal Y)$$ $, which consists of  all linear operators $T:\mathcal X\to \mathcal Y$ that satisfy
\begin{equation}
    \|T\|_{\Pi_2(\mathcal X;\mathcal Y)}=\sup\left\{\sqrt{\sum_{j=0}^\infty \|Tx_j\|_{\mathcal Y}^2}:(x_j)_{j\geq 0}\text{ is an orthonormal basis of $\mathcal X$}\right\}<\infty.\label{eq:nonorm}
\end{equation}
Observe that $\|\cdot\|_{\Pi_2(\mathcal X;\mathcal Y)}$ defines a norm on $\Pi_2(\mathcal X;\mathcal Y)$ and makes it a Banach space. 
This definition for the 2-summing norm coincides with the definitions in \cite[Sec. 8]{albiac2006topics} and \cite[Sec. 1.10]{johnson2001handbook}, since we restrict ourselves to the setting where $\mathcal X$ is Hilbert.   In particular,  $\Pi_2(\mathcal X,\mathcal Y)=HS(\mathcal X,\mathcal Y)$ is the space of Hilbert--Schmidt operators whenever $\mathcal Y$ is a Hilbert space. 

Thanks to the terminology above,  condition  (H$L^1$) can be reformulated as: For any $N>0$ there exists a constant $K_{\rm int}^N>0$ such that
\begin{equation}\|g(T_\gamma\Phi_0+v)[\,\cdot\,]^\top T_\gamma\psi_{\rm tw}\|_{\Pi_2(\mathcal W_Q;L^1(\mathbb R^d;\mathbb R))}^2<K_{\rm int}^N\end{equation} holds for all $\gamma\in \mathbb R$ and $v\in H^k$ satisfying $\|v\|_{H^k}\leq N.$
We refer to Appendix \ref{sec:about:noise} for several examples of triplets $(\mathcal{W}, Q,g)$
that satisfy the hypotheses (HgQ) and (H$L^1)$.


\subsection{Characterisation of solutions}\label{sec:character}
In order to rigorously discuss the concept of solutions to \eqref{eq:reaction:diff}, it turns out to be convenient to take a rather indirect approach. In particular, our main focus will be on the pair $(v, \vartheta)$ featuring in the decomposition \eqref{eq:pert_v},  which we directly prescribe by the non-autonomous evolution system
\begin{equation}
    \begin{aligned}\label{eq:pert:system}
        \mathrm dv&=[\mathcal Lv+\mathcal N_\sigma(v,\theta,\gamma)]\,\mathrm dt+\sigma\mathcal M(v,\theta)T_{-\gamma}\mathrm dW_t^Q,\\
        \mathrm  d\mathrm \theta&=[\Delta_y \theta+N_\sigma(v,\theta,\gamma)]\,\mathrm dt+\sigma M(v,\theta)T_{-\gamma}\mathrm dW_t^Q,
    \end{aligned}
\end{equation}
where $\gamma(t) = c_0 t$.
Recall  $\mathcal L=\mathcal L_{\rm tw}+\Delta_y$. The nonlinearities $\mathcal N_\sigma$, $\mathcal M$, $N_\sigma$, and $M$ are defined in {\S}\ref{sec:phase}, where we also  derive  system \eqref{eq:pert:system} in a natural|but non-rigorous|fashion from \eqref{eq:reaction:diff}
by imposing the constraint $P_{\rm tw}v(t)=\langle v(t),\psi_{\rm tw}\rangle _{L^2_x}=0$.

Our first result  shows that solutions to system \eqref{eq:pert:system} exist in a variational sense; we characterise this fully in {\S}{\ref{sec:evol:pert}}.
The dimension restriction $k>d/2+1$ enables us to fit our problem into the critical variational framework of Agresti and Veraar \cite{agresti2024criticalNEW}.
Note that the extra smoothness requirement is no big surprise as something similar is also needed in the deterministic approach  \cite{kapitula1997}.

\begin{proposition}[Perturbation system; see {\S}\ref{sec:evol:pert}]
\label{prop:var:higher:short}
Let $k>d/2+1$
and suppose  that   \textnormal{(Hf)}, \textnormal{(HTw)}, and \textnormal{(HgQ)} hold.
Further, fix $T>0$ and $0\leq \sigma\leq 1$. Then for any initial condition 
\begin{equation}
(v_0,\theta_0)\in H^k\times H^k_y,
\end{equation}
there exists a stopping time $\tau_{\infty}$
together with a variational solution $(v(t),\theta(t))$ to  \eqref{eq:pert:system} that is defined on $0\leq t<\tau_{\infty}\leq T$. 
Whenever  $\tau_{\infty}(\omega)<T$ holds
for some $\omega \in \Omega$,
the finite time  blow-up 
\begin{equation}
    \sup_{0\leq t<\tau}\big[\|v(t,\omega)\|_{H^k}^2+\|\theta(t,\omega)\|_{H^k}^2\big]+\int_0^\tau\big[\|v(s,\omega)\|_{H^{k+1}}^2+\|\theta(s,\omega)\|_{H^{k+1}_y}^2\big]\,\mathrm ds\to \infty
\end{equation}
occurs as $\tau\uparrow \tau_{\infty}(\omega)$.
\end{proposition}

Completely in line with  \cite{bosch2025conditionalspeedshapecorrections,bosch2024multidimensional,hamster2019stability,hamster2020diag,hamster2020expstability,hamster2020}, we remark  that within our ``frozen'' frame \eqref{eq:pert_v}, the system $(v(t),\theta(t))$ is  decoupled|in some statistical sense|from the global phase process $\gamma(t)=c_0t$ whenever $Q$ is translation invariant in the $x$-direction. Indeed, in this case $T_{-\gamma}\mathrm dW_t^Q$ and $\mathrm dW_t^Q$ are indistinguishable and $\mathcal N_\sigma$ and $N_\sigma$ will no longer explicitly depend on $\gamma$.

To justify the formally obtained system \eqref{eq:pert:system}, we 
combine the pair $(v,\theta)$  obtained from Proposition \ref{prop:var:higher:short} to introduce the function
\begin{equation}
\label{eq:mr:def:u:from:v:theta}
u(t)=T_{\gamma(t)+\theta(t)}\Phi_0+T_{\gamma(t)}v(t).
\end{equation}
We use rigorous It\^o computations in {\S}\ref{sec:evol:pert} to show that $u$ solves \eqref{eq:reaction:diff} in an analytically weak sense.

\begin{proposition}[Original system; see {\S}\ref{sec:evol:pert}]
\label{prp:mr:an:sol:org:sys}
Consider the setting  of Proposition \ref{prop:var:higher:short}.   Then for any $\zeta\in C_c^\infty(\mathbb R^d;\mathbb R^n)$, the identity
\begin{equation}\begin{aligned}
\langle u(t),\zeta\rangle_{H^k}=\langle u(0),&\,\zeta\rangle_{H^k}+\int_0^t  \big[\langle u(s),\Delta \zeta\rangle_{H^k} +\langle f(u(s)),\zeta\rangle_{H^k}\big]\mathrm ds +\sigma\int_0^t \mathcal \langle g(u(s))\,\mathrm d  W_s^Q,\zeta\rangle_{H^k}\label{eq:weak:orig:u}
\end{aligned}
\end{equation}
holds $\mathbb P$-a.s.\ for all $0\leq t< \tau_{\infty}. $
\end{proposition}

  We cannot expect to interpret $u(t)$ directly as a solution  to \eqref{eq:reaction:diff} in the variational sense; this is due to the fact that  $f(\Phi_{0})$ is not integrable in the transverse direction,  let alone in $H^k(\mathbb R^d;\mathbb R^n).$ In case one is unsatisfied with this notion of a solution, we remark 
  that one can define 
  the stochastic process
\begin{equation}
    z(t) = T_{\gamma(t)}v(t)+T_{\gamma(t)+\theta(t)}\Phi_0-
    T_{\gamma(t)} \Phi_0,
\end{equation}
which is hence related to \eqref{eq:mr:def:u:from:v:theta} via
$u(t) = z(t) + T_{\gamma(t)} \Phi_0$,
and establish that $z(t)$ is a variational solution of
\begin{equation}
   \mathrm  dz = [\Delta z + f(T_{\gamma(t) }\Phi_0 + z ) - f(T_{\gamma(t)} \Phi_0) ] \, \mathrm dt + g(T_{\gamma(t)} \Phi_0 + z)\, \mathrm dW^Q_t.
\end{equation}
We do not proceed further in this direction.


\subsection{Moment bounds and probability tail estimates}\label{sec:main:result} The main result of this paper is that we can  control the size of the quantity
    \begin{equation}
\begin{aligned}
    N_{\mu;k}(t)&=\|v(t)\|_{H^k}^2+\|\theta(t)\|_{H^k_y}^2+\int_0^t(1+t-s)^{-\mu}\|v(s)\|_{H^{k+1}}^2\mathrm ds\\&\qquad+\int_0^t(1+t-s)^{-\mu}\|v(s)\|_{H^{k}}\|\theta(s)\|_{H^{k}}\mathrm ds+\int_0^t(1+t-s)^{-\mu}\|\nabla_y\theta(s)\|_{H^{k}_y}^2\mathrm ds,\label{eq:size:main2}
    \end{aligned}
\end{equation}
in terms of maximal moments, for any $\mu>0$ that satisfies the conditions in Table \ref{table:main:hfdst2}. This  allows us to formulate probability tail estimates for the behaviour of the associated stopping time
\begin{equation}
    t_{\rm st}(\eta;k)=\inf\{t\geq 0:N_{\mu;k}(t)>\eta\}.\label{eq:st:main}
\end{equation}
As explained in {\S}\ref{sec:intro},
both the first and third integral terms are required to prevent possible $L^2([0,T];$ $H^{k+1} \times H^{k+1}_y)$ blow-up of the perturbation pair $(v, \vartheta)$, while the second and third   play a key role in our nonlinear stability analysis. 
In particular,  we do not need the $H^{k+1}$-control on the perturbation $v$ for the latter, in contrast to the situation in our prior works \cite{bosch2025conditionalspeedshapecorrections,bosch2024multidimensional,hamster2020expstability} where the global stochastic phase shift required such additional regularity on $v$.

\begin{table}[t!]
\centering
\renewcommand{\arraystretch}{2} 

\begin{tabularx}{0.6\textwidth}{|C|C|C|C|C|C|}
\hline
 & $d>5$ & $d=5$ & $d=4$ & $d=3$ & $d=2$ \\[.1cm] \hline
$\mu$ & $\left(1, \frac{d-1}{4} \right]$ & $1$ & $\frac{3}{4}$ & $\frac{1}{2}$ & $\frac{1}{4}$ \\[.15cm] \hline
$\mathfrak d_{\rm stc}(T)$ & $1$ & $1$ & $1$ & $\log(T)$ & $T^{1/2}$ \\[.15cm] \hline
$\mathfrak d_{\rm det}(T)$ & $1$ & $\log(T)$ & $T^{1/4}$ & $T^{1/2}$ & $T^{3/4}$ \\[.15cm] \hline
\end{tabularx}
\caption{The parameter $\mu$ controls the algebraic decay rate of both $v(t)$ and $\theta(t)$ in \eqref{eq:size:main2}. The terms $\mathfrak d_{\rm stc}(T)$ and $\mathfrak d_{\rm det}(T)$  are defined and have been discussed in {\S}\ref{sec:intro}; they capture the growth of  $\sup_{0\leq t\leq T}\int_0^t(1+t-s)^{-2\mu}\mathrm ds$ and $\sup_{0\leq t\leq T}\int_0^t(1+t-s)^{-{\mu}}\mathrm ds$ with respect to $T$.} 
\label{table:main:hfdst2}
\end{table}

Our main result
here provides logarithmic or polynomial growth bounds
for the expectation of the maximal value that $N_{\mu;k}(t)$ attains as we increase $T$. 
This should be seen as a first counterpart to the results found in \cite[Thm. 1.1]{hamster2020expstability} as well as \cite[Prop. 2.5 and Thm. 2.6]{bosch2024multidimensional}, where (planar) waves evolve over the real line and a cylindrical domain, respectively. Notice that the bounds in \cite{hamster2020expstability} have been improved in \cite{bosch2025conditionalspeedshapecorrections,bosch2024multidimensional} so that it is in line with large deviations theory  and other comparable results regarding metastability (in reaction-diffusion systems); see \cite{adams2022isochronal,da2014stochastic,dembo2009large,eichinger2022multiscale,freidlin2012random,funaki1999singular,gnann2024solitary,inglis2016general,kruger2014front,kuehn2020travelling,kuehn2022stochastic,lang2016multiscale,lang2016l2,lord2012computing,maclaurin2023phase,salins2021systems,salins2021metastability,sowers1992large,stannat2013stability,stannat2014stability,swiȩch2009pde,theewis2024large,van2024noncommutative}  and the references therein. 

\begin{theorem}[see {\S}\ref{sec:stability}]\label{thm:general}
Let
$k>d/2+1$ 
and suppose that \textnormal{(Hf)}, \textnormal{(HTw)}, \textnormal{(HgQ)} and  \textnormal{(H$L^1$)} hold.  Pick two sufficiently small constants $\delta_\eta>0$ and $\delta_\sigma>0$. 
There exists a constant $K>0$ so that for any integer $T\geq 2,$ any $0<\eta<\delta_\eta,$ any $0\leq \sigma\leq \delta_\sigma,$ any integer $p\geq 1$, and any initial condition that satisfies       \begin{equation}
         \langle v(\cdot,y,0),\psi_{\rm tw}\rangle_{L_x^2}=0 \label{eq:ortho}
    \end{equation} 
    for all $y\in \mathbb R^{d-1},$ we have the moment estimates
 \begin{equation}
 \begin{aligned}\label{eq:general}
     \mathbb E\left[\sup_{0\leq t\leq t_{\rm st}(\eta;k)\wedge T}|N_{\mu;k}(t)|^p\right]&\leq K^p\Big[ 
\|v(0)\|_{H^k}^{2p} +\|\theta(0)\|_{H^k_y}^{2p} 
 +\sigma^{2p}\mathfrak d_{\rm det}(T)^{2p}(p^p+\log(T)^p)\Big].
     \end{aligned}
 \end{equation}
In particular, there exists a constant $\kappa\in (0,1)$ such that the probability bound
 \begin{equation}
     \mathbb P\left(\sup_{0\leq t\leq t_{\rm st}(\eta;k)\wedge T}|N_{\mu;k}(t)|>\eta\right)\leq 2T\exp\left(-\frac{\kappa\eta}{\sigma^2\mathfrak d_{\rm det}(T)^2}\right)\label{eq:25}
 \end{equation}
 holds whenever $\|v(0)\|_{H^k}^{2} +\|\theta(0)\|_{H^k_y}^{2} < \kappa \eta$.

\end{theorem}

    
We point out that by invoking the standard Markov inequality, the control of the second moment $(p=1)$ 
 alone  allows us to show that the exit probability increases logarithmically or polynomially in time; see also \cite[Sec. 2.3]{bosch2024multidimensional}.
 Returning to our main theorem in {\S}\ref{sec:intro}, we remark that the requirement for the initial condition $u(0)$ to be $\mathcal O(\eta)$-close to $\Phi_0(x)$ means that we should be able to 
 make the decomposition \begin{equation}u(0)=v(0)+T_{\theta(0)}\Phi_0\label{eq:decomposition}\end{equation}
 in such a way that $\|v(0)\|_{H^k}^2+\|\theta(0)\|_{H^k_y}^2< \kappa \eta$ holds. 
 This is similar to the requirement in \cite{kapitula1997}, yet\\[-.085cm] we  do not need the  $L^1_y$-control on  $\theta(0)$.

\begin{remark}\label{remark:ortho}
    Our proof will show that the orthogonality condition \eqref{eq:ortho} is not strictly necessary for \eqref{eq:25} to hold. We nevertheless include it here to ensure the (local) uniqueness of the decomposition 
    \eqref{eq:decomposition}|see \cite[Lem 2.2]{kapitula1997}|and preserve the one-to-one correspondence between the systems \eqref{eq:reaction:diff} and \eqref{eq:pert:system}.
    In the deterministic setting, \eqref{eq:ortho} together with the integrability $\theta(0) \in L^1_y$ are required to ensure that the contribution of the initial condition decays to zero. This also becomes relevant in the stochastic setting when studying the relative decay of $v(t)$ compared to $\theta(t)$ in greater detail or when performing asymptotic expansions in the spirit of \cite[Prop. 2.3 and 2.6]{bosch2025conditionalspeedshapecorrections}. 
    %
\end{remark}

 \begin{corollary} Consider the setting of Theorem \ref{thm:general}. Then for some $0<\kappa_b\leq \kappa$ we have
 \begin{equation}
     \mathbb P\left(\sup_{0\leq t\leq T}\|u(t)-\Phi_0(\cdot-c_0 t)\|_{H^k}^2>\eta\right)\leq 4T\exp\left(-\frac{\kappa_b\eta}{\sigma^2\mathfrak d_{\rm det}(T)^2}\right).
 \end{equation}
 \end{corollary}
 \begin{proof}
As a direct consequence of \eqref{eq:25},  we have the probability tail estimates
\begin{equation}
        \mathbb P\left(\sup_{0\leq t\leq T}\|u(x+c_0 t,y,t)-\Phi_0(x-\theta(y,t))\|_{H^k}^2>\eta\right)\leq 2T\exp\left(-\frac{\kappa\eta}{\sigma^2\mathfrak d_{\rm det}(T)^2}\right)
    \end{equation}
    together with
    \begin{equation}
        \mathbb P\left(\sup_{0\leq t\leq T}\|\theta(y,t)\|_{H^k_y}^2>\eta\right)\leq 2T\exp\left(-\frac{\kappa\eta}{\sigma^2\mathfrak d_{\rm det}(T)^2}\right).
    \end{equation}
 Since  $A\Rightarrow B$ implies $\mathbb P(A)\leq \mathbb P(B)$, we may use the triangle inequality to find
   \begin{align}
       \nonumber&\mathbb P\left(\sup_{0\leq t\leq T}\|u(t)-\Phi_0(\cdot-c_0 t)\|_{H^k}^2>\eta\right)\\&\nonumber\qquad \leq \mathbb P\left(\sup_{0\leq t\leq T}\big[\|u(t)-\Phi_0(\cdot-c_0 t-\theta(t))\|_{H^k}+\|\Phi_0(\cdot-c_0 t-\theta(t))-\Phi_0(\cdot -c_0 t)\|_{H^k}\big]^2>\eta\right)\\
       \nonumber&\qquad\leq \mathbb P\left(\sup_{0\leq t\leq T}\|u(t)-\Phi_0(\cdot-c_0 t-\theta(t))\|_{H^k}^2>\eta/4\right)\\\nonumber&\qquad\qquad\qquad+\mathbb P\left(\sup_{0\leq t\leq T}\|\Phi_0(\cdot-c_0 t-\theta(t))-\Phi_0(\cdot -c_0 t)\|_{H^k}^2>\eta/4\right)\\
       \nonumber &\qquad\leq \mathbb P\left(\sup_{0\leq t\leq T}\|u(t)-\Phi_0(\cdot-c_0 t-\theta(t))\|_{H^k}^2>\eta/4\right)+\mathbb P\left(\sup_{0\leq t\leq T}\|\theta(t)\|_{H^k_y}^2>\eta/4K_0^2\right)\\
        &\qquad\leq 4T\exp\left(-\frac{\kappa_b\eta}{\sigma^2\mathfrak d_{\rm det}(T)^2}\right)
        \label{eq:mr:prf:main:thm:bnd:4:T}
   \end{align}
   for some $0<\kappa_b\leq \kappa.$
   In the penultimate inequality we have exploited the translation invariance of integration and  the fact that
   \begin{equation}
       \|\Phi_0-\Phi_0(\cdot-\theta)\|_{H^k}\leq K_0\|\theta\|_{H^k_y}
   \end{equation}
    holds for some constant $K_0>0$, which depends only the value of $\|\Phi_0'\|_{H^k_x}$; see \eqref{eq:nl:bnd:prod:hky:hkx:lip}.  
 \end{proof}


\begin{proof}[Proof of Theorem \ref{thm:main}] 
    We now consider the timescales \eqref{eq:timescales} and introduce 
     the convenient notation
    \begin{equation}
        p( \eta, \sigma,r) = 4T(\eta/\sigma^2,r)\exp\left(-\frac{\kappa_b\eta}{\sigma^2\mathfrak d_{\rm det}\big(T(\eta/\sigma^2,r)\big)^2}\right).
    \end{equation}
We remark that for $\eta/\sigma^2 \le 1$ the desired bound is automatically satisfied, provided that $\kappa_b > 0$ is sufficiently small to ensure $8 \, \mathrm{exp}(-\kappa_b / 4) \ge 1$. In particular, from now on we assume $\eta / \sigma^2 \ge 1$. Moreover, for $T(\eta/\sigma^2,r) \le 2$, the bound follows directly from \eqref{eq:mr:prf:main:thm:bnd:4:T} by choosing $T = 2$.
We hence also assume $T(\eta/\sigma^2,r) \ge 2$ from now on and split our analysis into three separate cases.

   \textbf{{Case $\mathbf{\textit{d}>5}$}\,\,\,}
   Choosing $\kappa_c=3\kappa_b/4$ and recalling $\mathfrak{d}_{\rm det}(T) =1 $,
   the bound \eqref{eq:int:bnd:p:thm}
   follows directly from
   \eqref{eq:mr:prf:main:thm:bnd:4:T}
   upon noting that $1 \le (\eta/\sigma^2)^r \le (\eta/\sigma^2) $.

       \textbf{Case $\mathbf{\textit{d}=5}$\,\,\,} 
        First assume $r \in [\frac{1}{3}, 1]$.
        This allows us to compute
    \begin{equation}
    \sigma^2 \mathfrak{d}_{\rm det}\big(T(\eta/\sigma^2,r)\big)^2
    = \sigma^2 \log T(\eta/\sigma^2,r)^2
    = \sigma^2 \kappa_c^2 ( \eta/\sigma^2)^{1-r},
\end{equation}
   which implies 
\begin{equation}\begin{aligned}
       p(\eta,\sigma,r) &= 
    \mathrm{exp}\big( \kappa_c ( \eta/\sigma^2)^{(1-r)/2} \big) \mathrm{exp}\big( - \kappa_b \eta^r/( \kappa_c^2   \sigma^{2r} )\big).
\end{aligned}\end{equation}
The key point\footnote{This property explains why $r = 1/3$ is the critical value for $d =5$.
Similarly, for $d > 5$ one can interpret $r=1$ as  the critical value, since smaller values of $r$ do not result into longer timescales.} 
here is that $(1-r)/2 \le r$
holds for $r \ge \frac{1}{3}$. This yields
 $   1 \le (\eta / \sigma^2)^{(1-r)/2} \le (\eta  / \sigma^2)^r$
and hence
\begin{equation}\begin{aligned}
       p(\eta,\sigma,r) &\le 
    \mathrm{exp}\big(
      (\kappa_c -\kappa_b /\kappa_c^2  )  (\eta / \sigma^2)^r\big)
       \le \mathrm{exp}\left(
      -\frac{\kappa_b\eta^r}{4\sigma^{2r}}\right),
\end{aligned}\end{equation}
where the second bound follows by picking
$\kappa_c > 0$ to be sufficiently small.
Now for $0 < r \le \frac{1}{3}$, we note that our defined timescale is invariant 
while the probability bound increases as $r$ decreases. In particular, \eqref{eq:int:bnd:p:thm} remains valid, completing the proof for this case.

    \textbf{Case $\mathbf{2\leq \textit{d}\leq 4}$\,\,\,} 
    We first compute
    \begin{equation}
    \sigma^2 \mathfrak{d}_{\rm det}\big(T(\eta/\sigma^2,r)\big)^2
    = \sigma^2 T(\eta/\sigma^2,r)^{(5-d)/2}
    = \sigma^2 \kappa_c^{(5-d)/2}
    ( r \eta / \sigma^2)^{1-r}
    = \kappa_c^{(5-d)/2} \sigma^{2r} (r\eta)^{1-r},
\end{equation}
   which implies that
   \begin{equation}\begin{aligned}
       p(\eta,\sigma,r) &= \kappa_c 
       (r \eta/\sigma^2)^{\frac{2(1-r)}{5-d}}\exp\left(-\frac{{\kappa_b (r\eta)^r}}{\kappa_c^{(5-d)/2} r{\sigma^{2r}}}\right).\label{eq:first:bound:2d4}\\
\end{aligned}\end{equation}
We now make use of the  identity
   \begin{equation}
       z^pe^{-qz^{2r}}\leq \left(\frac{2p}{e(4-s)q r}\right)^{p/2r} e^{-sqz^{2r}/4},\quad z>0,\label{eq:power-exp}
   \end{equation}
   which holds for any $p,q,r>0$ and $0\leq s<4.$ 
   By picking $p = 4(1-r)/(5-d)$, $q = \kappa_b / (\kappa_c^{(5-d)/2} r)$,
$s= \kappa_c^{(5-d)/2} r^{1-r}$,  and  $z = \sqrt{r\eta}/\sigma$,
we hence find
\begin{equation}
p(\eta,\sigma,r) \le 
\left[ \kappa_c
 \left( \frac{8 (1-r)}{e(4 -\kappa_c^{(5-d)/2}r^{1-r})(5-d) \kappa_b } \right)^{ \frac{2 (1-r)}{(5-d)}} \right]^{1/r}
    \exp\left(-\frac{{\kappa_b \eta^r}}{4\sigma^{2r}}\right).
\end{equation}
Recalling that $0 < r \le 1$ and assuming a-priori that $\kappa_c \le 1$ yields
\begin{equation}
\begin{array}{lcl}
p(\eta,\sigma,r) &\le& 
\left[ \kappa_c
 \left( \dfrac{2 }{e(5-d) \kappa_b } \right)^{ \frac{2 (1-r)}{(5-d)}} \right]^{1/r} \exp\left(-\dfrac{{\kappa_b \eta^r}}{4\sigma^{2r}}\right)
 \\[0.2cm]
 &\le& 
 \left[ \kappa_c
 \max\left\{1, 
  \dfrac{2 }{e(5-d) \kappa_b }  \right\}^{ \frac{2 }{(5-d)}} \right]^{1/r} \exp\left(-\dfrac{{\kappa_b \eta^r}}{4\sigma^{2r}}\right).
\end{array}
\end{equation}
By choosing $\kappa_c > 0$ to be sufficiently small, we can ensure that the bracketed term is less than one for each $d \in \{2, 3, 4\}$ and each $0 < r \le 1$, which yields
\begin{equation}
    p(\eta,\sigma,r) \le \exp\left(-\frac{{\kappa_b \eta^r}}{4\sigma^{2r}}\right),
\end{equation}
as desired. Note that $\kappa_c$ does depend on $\kappa_b,$ just like for the cases above.
\end{proof}

\section{Nonlinearities}\label{sec:phase}
Our goal here is to provide definitions
for the nonlinearities appearing in \eqref{eq:pert:system}
and obtain a number of preliminary estimates for these terms, \pagebreak which we need for proving (local) existence and uniqueness as well as metastability. Starting with the former, 
we introduce a $C^\infty$-smooth  cut-off function
\begin{equation}
    \chi:\mathbb R\to[\tfrac14,\infty)\quad \textnormal{such that}\quad \chi(z)=\tfrac14,\quad z\leq \tfrac14,\quad \chi(z)=z,\quad z\geq \tfrac12,
\end{equation}
and subsequently define the  auxiliary function
\begin{equation}
\label{eq:nl:def:k:2}
\begin{array}{lcl}
 \Upsilon(\vartheta)  =  
-\dfrac{1}{\chi(\langle T_\theta\Phi_0',\psi_{\rm tw}\rangle_{L^2_x})}.
\end{array}
\end{equation}
Observe that $\Upsilon(0) = -1$ holds on account of the normalisation $\langle \Phi_0', \psi_{\rm tw} \rangle = 1$. To ensure that $\Upsilon$ is unaffected by the cut-off function, it suffices to take $\theta$   sufficiently
small in $H^k_y$ for $k > (d-1)/2$, as this implies pointwise control due to the Sobolev embedding $H^k_y \subset L^\infty_y$. This can be deduced from Lemmas \ref{lem:innerproduct:Hky} and \ref{lem:kappa:phi}.

We further define the expressions
\begin{equation}
\label{eq:nl:def:j:jtr}
    \begin{array}{lcl}
    \mathcal J(v,\theta)&=&
    f(v+T_\theta\Phi_0)-f(T_\theta\Phi_0)
    -Df(\Phi_0)v+|\nabla_y\theta|^2T_\theta\Phi_0'' ,
     \\[0.2cm]
   \mathcal J_{\rm tr}(v, \vartheta, \gamma)
    &=& - \frac{1}{2}T_\vartheta \Phi_0'' \Upsilon(\vartheta)^2\sum_{j=0}^\infty 
\langle g(v + T_{\vartheta} \Phi_0)T_{-\gamma}\sqrt Qe_j ,  \psi_{\rm tw} \rangle_{L^2_x}^2 ,
    \end{array}
\end{equation}
together with the notation
\begin{equation}
\textstyle\nabla_y\theta=(\partial_{x_2}\theta,\ldots,\partial_{x_d}\theta),\qquad|\nabla_y\theta|^2=\sum_{j=2}^d(\partial_{x_j}\theta)^2, 
\end{equation}
and remark  that $\mathcal J_{\rm tr}$ does not depend on the choice of the basis $(e_j)_{j\geq 0}$ of $\mathcal W$; see 
Lemma \ref{lem:nl:independence}.
The nonlinearities in \eqref{eq:pert:system}
can now be defined as
\begin{equation}
\label{eq:nl:def:n:m}
    \begin{array}{lcl}
    N_\sigma(v, \vartheta,\gamma) 
    & = & \Upsilon(\vartheta)  \langle \mathcal J(v, \vartheta) + \sigma^2 \mathcal J_{\rm tr}(v, \vartheta, \gamma),
    \psi_{\rm tw}
    \rangle_{L^2_x} ,
    \\[0.2cm]
    M(v, \vartheta)[\xi] & = & 
    \Upsilon(\vartheta) \langle g(v + T_{\vartheta} \Phi_0)[\xi] ,  \psi_{\rm tw} \rangle_{L^2_x} ,
    \end{array}
\end{equation}
for $v\in H^k$, $\theta\in H^k_y$, $\gamma\in\mathbb R$,  and $\xi\in\mathcal W_Q,$ together with
\begin{equation}
\label{eq:nl:def:cal:n:m}
    \begin{array}{lcl}
    \mathcal N_\sigma(v, \vartheta,\gamma)
     & = & 
     \mathcal J(v, \vartheta) + \sigma^2 \mathcal J_{\rm tr}(v, \vartheta, \gamma)
    + N_\sigma(v,\vartheta,\gamma) T_\vartheta \Phi_0' ,
    \\[0.2cm]
    \mathcal M(v, \vartheta)[\xi]
     & = & 
     g(v+T_\theta\Phi_0)[\xi]+M(v,\theta)[\xi]T_\theta\Phi_0' .
    \end{array}
\end{equation}
To appreciate these definitions, we note that
it is clear from construction that
\begin{equation}
    P_{\rm tw} \mathcal{N}_\sigma = 0,
    \qquad
    P_{\rm tw} \mathcal{M} = 0,
\end{equation}
whenever $\theta$ is sufficiently small in $H^k_y$ again.
The identity $P_{\rm tw} \mathcal{L}_{\rm tw} v = 0$ in combination with the commutation
relation\footnote{In this work|as in \cite{kapitula1997}|we exploit the crucial property that  the operators $P_{\rm tw}$ and $\Delta_y$ commute. This is   no longer the case when the diffusion coefficients are unequal, hence a different approach is needed. We claim that one could proceed as in \cite{hoffman2015multi}, provided that no transverse instabililities can occur. 
In particular, the (spectral) projections that we have introduced in \cite{bosch2025conditionalspeedshapecorrections,bosch2024multidimensional} for the cylindrical domain setting are slightly different, yet they also  commute with the corresponding linear operators. The latter is true even in the case of unequal diffusion coefficients. }  \begin{equation}P_{\rm tw}\Delta_yv=\Delta_yP_{\rm tw}v
\end{equation}
shows
that the (desired) orthogonality  $P_{\rm tw} v(t) = 0$ is maintained
for all $t> 0$ if $P_{\rm tw}v(0)=0$, and as long as the solution exists and $\| \vartheta(t) \|_{H^k_y}$ remains sufficiently small.  

To illustrate the connection
with the original problem \eqref{eq:reaction:diff},
we introduce the representation
\begin{equation}
    u(t) = \phi(v(t), \vartheta(t); t),\qquad \phi(v,\theta;t)=  T_{\gamma(t) + \vartheta} \Phi_0 + T_{\gamma(t)} v
\end{equation}
with $\gamma(t)=c_0 t$ and perform a formal computation involving   It{\^o}'s formula. 
Computing the required derivatives
\begin{equation}
\begin{array}{lcl}
    D_v \phi( v, \vartheta; t)[V] &=& T_{\gamma(t)}[V] ,
    \\[0.2cm]
    D_\vartheta \phi(v, \vartheta; t)[\Theta]
     & = & -\Theta T_{\gamma(t) + \vartheta}\Phi_0' ,
    \\[0.2cm]
    D^2_{\vartheta\vartheta} \phi(v, \vartheta; t)[\Theta,\Theta]
    & = & \Theta^2 T_{\gamma(t) + \vartheta}\Phi_0'',
\end{array}
\end{equation}
we arrive at 
\begin{equation}
\begin{aligned}
 \mathrm du  &= \textstyle \bigg[  - c_0 \partial_x u 
+ T_{\gamma(t)}\Big[ \mathcal{L} v + \mathcal{N}_\sigma(v,\theta,\gamma(t))
- T_{\vartheta} \Phi_0'(\Delta_y \vartheta + N_\sigma(v,\theta,\gamma(t)))
\\&\qquad\quad \textstyle + \frac{1}{2} \sigma^2 T_{\vartheta} \Phi_0'' \sum_{j=0}^\infty \big[M(v,\theta)[T_{-\gamma(t)}\sqrt{Q}e_j]\big]^2\Big]
\bigg] \, \mathrm dt 
+ \sigma T_{\gamma(t)} \big[ \mathcal{M}(v,\theta) - T_{\vartheta} \Phi_0'M(v,\theta)\big] \, \mathrm d W^Q_t.\label{eq:nonsimplified}
\end{aligned}
\end{equation}
Exploiting the structure of the nonlinearities in \eqref{eq:nl:def:n:m} and \eqref{eq:nl:def:cal:n:m} 
and the definition of the trace term $\mathcal{J}_{\rm tr}$, which can in fact be rewritten as
\begin{equation}
    \textstyle \mathcal{J}_{\rm tr}(v, \vartheta, \gamma)
    = -\frac{1}{2} T_{\vartheta} \Phi_0''\sum_{j=0}^\infty 
    \big[ M(v, \vartheta)[T_{-\gamma} \sqrt{Q} e_j] \big]^2,
\end{equation}
allows us to simplify \eqref{eq:nonsimplified} into
\begin{equation}
\begin{array}{lcl}
\mathrm du   =   \Big[  - c_0 \partial_x u 
+ T_{\gamma(t)}\big[ \mathcal{L} v + \mathcal{J}(v,\theta)
- T_{\vartheta} \Phi_0' \Delta_y \vartheta 
\big]\Big] \, \mathrm  dt
+ \sigma g(u)\,\mathrm  d W^Q_t.
\end{array}
\end{equation}
Importantly, after exploiting the identity
\begin{equation}
    \Delta_y T_{\theta}\Phi_0+ T_{\theta}\Phi_0'\Delta_y\theta=T_\theta\Phi_0''|\nabla_y\theta|^2,\label{eq:Deltay:Ttheta}
\end{equation}
together with the fact that $\Phi_0$ satisfies the travelling wave ODE in \eqref{eq:mr:trv:wave:ode}, we  recover
\begin{equation}
   \mathrm  du = [ \Delta u + f(u)] \, \mathrm dt + \sigma g(u) \, \mathrm dW^Q_t.
\end{equation}
These It\^o computations are made rigorous in {\S}\ref{sec:evol:pert}.

The main estimates of this section are summarised in the following two results. The fact that the right-hand side of \eqref{eq:nl:prop:nl:bnds:N:etc} does not contain any $\mathcal O(\|\vartheta\|_{H_y^k}^2)$ contributions
is of crucial importance\\[-0.065cm] for the stability argument in dimensions $2 \le d \le 5$. The Lipschitz bounds
in Proposition \ref{prop:nl:bnds:lip} will be used in {\S}\ref{sec:evol:pert} to establish well-posedness of \eqref{eq:pert:system}. Notice that
\eqref{eq:nl:main:prp:bnd:n:caln:lip} deliberately references the lower-order spaces $H^{k-1}_y$
and $H^{k-1}$ in order to align with the requirements
of the variational framework developed in \cite{agresti2024criticalNEW}. Finally, let us  point out that the assumption (H$L^1$) is only needed to obtain
the\linebreak $\Pi_2(\mathcal W_Q;L^1_y)$-estimate in 
\eqref{eq:nl:prp:bnd:m:cal:m:unif}. We exploit the latter in {\S}\ref{sec:stability} to prove stability.

\begin{proposition}\label{prop:nl:bnds}
 Pick $k >d/2$ and suppose that \textnormal{(Hf)}, \textnormal{(HTw)},  \textnormal{(HgQ)}, and \textnormal{(H$L^1$)} hold. For all $N>0$ there exists a   $K^N > 0$  so that
 for all $\gamma \in \mathbb R$, and every $v \in H^k$ and $\vartheta \in H^k_y$   with 
 $\| v \|_{H^k} \le N$ and $\|\theta\|_{H^k_y} \leq N$, we have the bounds
\begin{equation}
\label{eq:nl:prp:bnd:m:cal:m:unif}
    \begin{array}{lcl}
       \left\|M(v,\theta)T_{-\gamma}\right\|_{HS(\mathcal W_Q;H^k_y )}
       + \left\|M(v,\theta)T_{-\gamma}\right\|_{\Pi_2(\mathcal W_Q;L^1_y)} & \le & K^N,
       \\[0.2cm]
       \left\|\mathcal M(v,\theta)T_{-\gamma}\right\|_{HS(\mathcal W_Q;H^k )}
        & \le & K^N.
    \end{array}
\end{equation}
If in addition $\vartheta \in H^{k+1}_y$, then we also have the estimates
\begin{equation}
\label{eq:nl:prop:nl:bnds:N:etc}
    \begin{array}{lcl}
\left\| N_\sigma(v,\theta,\gamma)\right\|_{H^k _y}+\left\|N_\sigma(v,\theta,\gamma)\right\|_{L^1 _y}
     & \leq & K^N\left(\sigma^2 + \|v\|_{H^k}^2+ \|v\|_{H^k} \| \vartheta\|_{H^k_y}  + \|\nabla_y
     \theta\|_{H^{k} _y}^2\right)     , 
       \\[0.2cm]
       \left\| \mathcal N_\sigma(v,\theta,\gamma)\right\|_{H^k}
        & \le & K^N
        \left(\sigma^2 + \|v\|_{H^k}^2+ \|v\|_{H^k} \| \vartheta\|_{H^k_y}  + \|\nabla_y
     \theta\|_{H^{k} _y}^2\right), 
    \end{array}
\end{equation}
for any $\sigma \in \mathbb R$.
\end{proposition}

\begin{proposition}\label{prop:nl:bnds:lip}
Pick $k > d/2 + 1$ and 
suppose that \textnormal{(Hf)}, \textnormal{(HTw)}, and \textnormal{(HgQ)}  hold.
For all $N>0$ there exists a $K^N > 0$
so that for all  $\gamma \in \mathbb R$, and every  $v_A, v_B \in H^k$ and $\vartheta_A, \vartheta_B \in H^{k}_y$ 
 with
\begin{equation}
    \max \{ \| v_A \|_{H^k} ,  \| v_B\|_{H^k} , \| \vartheta_A \|_{H^k_y} , \| \vartheta_B \|_{H^k_y} \}
    \le N, 
\end{equation}
we have the bounds
\begin{equation}
\label{eq:nl:main:prp:bnd:m:calm:lip}
\begin{array}{lcl}
    \| M(v_A, \vartheta_A)T_{-\gamma} - M(v_B, \vartheta_B)T_{-\gamma} \|_{HS(\mathcal W_Q;H^k_y)} & \le & K^N\big[ \|v_A - v_B\|_{H^k} + \| \vartheta_A - \vartheta_B \|_{H^k_y}   \big] ,
\\[0.2cm]
    \| \mathcal M(v_A, \vartheta_A)T_{-\gamma} - \mathcal M(v_B, \vartheta_B)T_{-\gamma} \|_{HS(\mathcal W_Q;H^k)} & \le & K^N\big[ \|v_A - v_B\|_{H^k} + \| \vartheta_A - \vartheta_B \|_{H^k_y}   \big] ,
\end{array}
\end{equation}
together with 
\begin{equation}
\label{eq:nl:main:prp:bnd:n:caln:lip}
\begin{array}{lcl}
\| N_\sigma(v_A, \vartheta_A,\gamma) -  N_\sigma(v_B, \vartheta_B,\gamma) \|_{H^{k-1}_y}  
    &\le& K^N 
    \big[ \| v_A - v_B \|_{H^{k-1}} + \|\vartheta_A - \vartheta_B\|_{H^k_y}
    \big] ,
\\[0.2cm]
    \| \mathcal{N}_\sigma(v_A, \vartheta_A,\gamma) -  \mathcal{N}_\sigma(v_B, \vartheta_B,\gamma) \|_{H^{k-1}}  
    &\le& K^N 
    \big[ \| v_A - v_B \|_{H^{k-1}} + \|\vartheta_A - \vartheta_B\|_{H^k_y}
    \big],
\end{array}
\end{equation}
for any $0\leq \sigma\leq 1$. 
\end{proposition}

\subsection{Preliminaries}\label{sec:nonlin:prelim}

In this section we formulate several  preliminary estimates. We start by considering the behaviour of various $L^2_x$ inner-products and norms, which we view as functions of the transverse coordinate $y$.

\begin{lemma}\label{lem:innerproduct:Hky}
Pick $k \ge 0$ and suppose that \textnormal{(Hf)} and \textnormal{(HTw)} both hold. Then for any $w \in H^k$ we have the bound
\begin{equation}
\label{eq:nl:bnd:l2x:to:hk}
    \| \langle w , \psi_{\rm tw} \rangle_{L^2_x} \|_{H^k_y} \le \| w \|_{H^k} \| \psi_{\rm tw} \|_{L^2_x}.
\end{equation}
\end{lemma}
\begin{proof}
For any multi-index $\alpha \in \mathbb Z^{d-1}_{\ge 0}$ with $|\alpha| \le k$,
we obtain the estimate
\begin{equation}
 |   \partial^\alpha_y \langle  w, \psi_{\rm tw} \rangle_{L^2_x} |^2
    = 
     | \langle  \partial^\alpha_y w, \psi_{\rm tw} \rangle_{L^2_x} |^2
\le 
\| \partial^\alpha_y w \|_{L^2_x}^2 \| \psi_{\rm tw} \|_{L^2_x}^2,
\end{equation}
that is, pointwise with respect to $y$.
Integrating over $y \in \mathbb R^{d-1}$
 yields
\begin{equation}
 \|   \partial^\alpha_y \langle  w, \psi_{\rm tw} \rangle_{L^2_x} \|^2_{L^2_y}
 \le \| \partial^\alpha_y w \|_{L^2}^2 \| \psi_{\rm tw} \|_{L^2_x}^2,
\end{equation}
which implies the stated bound.   
\end{proof}

\begin{lemma}
    For any  $\vartheta_A, \vartheta_B \in \mathbb R$ and any 
    $\zeta \in C(\mathbb R; \mathbb R^n)$ with $\zeta'\in L^2_x$ we have the bound
    \begin{equation}
       \label{eq:nl:lip:bnd:l2:x:diff:shift}
        \|T_{\vartheta_A} \zeta - T_{\vartheta_B} \zeta \|_{L^2_x} \le  | \vartheta_A - \vartheta_B | \| \zeta'\|_{L^2_x}.
    \end{equation}
\end{lemma}
\begin{proof}
Exploiting the invariance of 
integration under translations,
a direct computation yields
    \begin{equation}
         \begin{aligned}
             \|T_{\vartheta_A} \zeta-T_{\vartheta_B} \zeta\|_{L^2_x}^2&=\int_{\mathbb R}|T_{\vartheta_A}\zeta(x)-T_{\vartheta_B}\zeta(x)|^2\,\mathrm dx\\
&=\int_{\mathbb R} \left|\int_0^1 (\vartheta_A-\vartheta_B) T_{\vartheta_B + s(\vartheta_A - \vartheta_B)}\zeta'(x)\,\mathrm ds\right|^2\mathrm dx\\
             &\leq (\theta_A-\theta_B)^2\int_0^1\int_{\mathbb R} \big(T_{\vartheta_B + s(\vartheta_A-\vartheta_B)}\zeta'(x)\big)^2\mathrm dx\,\mathrm ds\\
             &= |\vartheta_A-\vartheta_B|^2 \|\zeta'\|_{L^2_x}^2,
         \end{aligned}
     \end{equation}
as desired.
\end{proof}

\begin{corollary}
For any 
$\vartheta_A, \vartheta_B \in L^2_y$ and any $\zeta \in C(\mathbb R; \mathbb R^n)$ with $ \zeta' \in L^2_x$ we have
the bound
\begin{equation}
\label{eq:nl:bnd:shift:diff:l2:lip}
    \|T_{\vartheta_A} \zeta - T_{\vartheta_B}\zeta \|_{L^2} \le  \|\vartheta_A-\vartheta_B\|_{L^2_y}  \| \zeta'\|_{L^2_x}.
\end{equation}
\end{corollary}
\begin{proof}
A direct computation using \eqref{eq:nl:lip:bnd:l2:x:diff:shift} yields
    \begin{equation}
         \begin{aligned}
             \|T_{\vartheta_A} \zeta-T_{\vartheta_B} \zeta\|_{L^2}^2&=\int_{\mathbb R^{d-1}} \|T_{\vartheta_A(y)}\zeta-T_{\vartheta_B(y)}\zeta\|^2_{L^2_x} \,\mathrm dy
             \\
             &\leq \int_{\mathbb R^{d-1}}(\vartheta_A(y)-\vartheta_B(y))^2\|\zeta'\|_{L^2_x}^2 \, \mathrm d y,
         \end{aligned}
     \end{equation}
     which equals $ \| \vartheta_A - \vartheta_B \|_{L^2_y}^2  \|\zeta'\|_{L^2_x}^2 $ and proves the assertion.
\end{proof}

Our final two results concern products of $H^k_y$ functions
with $H^k$ functions that can be interpreted as $y$-dependent translations of functions that depend only on $x$. This latter condition is important, as it allows us to estimate $x$-integrals independently from their $y$-counterparts. The key ingredient behind our estimates is the inequality
\begin{equation}
\label{eq:nl:bnd:multilinear:general}
    \| \Lambda[ \partial^{\alpha_1} \phi_1, \ldots, \partial^{\alpha_\ell} \phi_{\ell}] \|_{L^2_y}
    \le  K |\Lambda|  \| \phi_1 \|_{H^k_y} \cdots \| \phi_\ell \|_{H^k_y},
\end{equation}
which holds for every $\ell$-linear map $\Lambda: (\mathbb R^{n})^{ \ell} \to \mathbb R^n$ and any tuple $(\phi_1, \ldots, \phi_\ell) \in (H^k)^{\ell}$, provided that $|\alpha_1| + \ldots + |\alpha_j| \le k$
    and $k > (d-1)/2$. Observe that $K>0$ does not depend on $\Lambda.$ This is related to the fact that $H^k_y$ is an algebra under multiplication for these values of $k$, in the sense that $\|\phi\psi\|_{H^k_y} \le K \|\phi\|_{H^k_y}\|\psi\|_{H^k_y}$; see 
\cite[Thm. 4.39]{adams2003sobolev}. Naturally,  analogous inequalities hold for $H^k$ provided that $k> d/2$.

\begin{lemma}\label{lem:kappa:phi}
Pick $k > (d-1)/2$. There exists a constant $K > 0$ so that for every $\kappa \in \mathbb R$,
any
$\phi,\theta_A,\theta_B\in H^k_y$, and any $\zeta \in C(\mathbb R; \mathbb R^n)$ with $ \zeta' \in H^k_x$,
we have the bound
\begin{equation}
\label{eq:nl:bnd:prod:hky:hkx:lip}
    \| ( \kappa + \phi) (T_{\vartheta_A} \zeta - T_{\vartheta_B} \zeta ) \|_{H^k}
    \le K \big( |\kappa| +  \| \phi \|_{H^k_y} \big) (1 + \|\vartheta_A\|^{k-1}_{H^{k}_y}+ \|\vartheta_B\|^k_{H^k_y} ) \|
    \| \vartheta_A - \vartheta_B \|_{H^k_y} \| \zeta' \|_{H^k_x}.
\end{equation}
\end{lemma}
\begin{proof}
Pick a multi-index $\alpha \in \mathbb Z^{d-1}_{\ge 0}$  and an integer $i \ge 0$ with $|\alpha| + i \le k$. Then the derivative \begin{equation}\partial^\alpha_y \partial_x^i [ (\kappa+\phi) (T_{\vartheta_A} \zeta - T_{\vartheta_B} \zeta ) ]\end{equation}
can be expressed as a finite sum of expressions of two types.
Writing $\Pi$ for the pointwise product, the first type is given by
\begin{equation}
 \mathcal{I}_{I} =  \Pi[ \partial_y^\beta (\kappa + \phi)   , \partial_y^{\gamma_1} (\vartheta_A - \vartheta_B), \partial_y^{\gamma_2} \vartheta_{\# 2 }, \ldots , \partial_y^{\gamma_{\ell}} \vartheta_{\# \ell}] T_{\vartheta_A} \zeta^{(i+\ell)},
\end{equation}
with labels $\#_i \in \{A, B\}$
and multi-indices $\beta \in \mathbb Z^{d-1}_{\ge 0}$
and $\{ \gamma_i \}_{i=1}^{\ell} \in \mathbb Z^{d-1}_{\ge 0}$
that satisfy 
$|\gamma_i| \ge 1 $ for $1 \le i \le \ell \le |\alpha|$,
together with
\begin{equation}
 |\beta| + |\gamma_1| + \ldots + |\gamma_\ell| = |\alpha|.
\end{equation}
The second type is given by
\begin{equation}
 \mathcal{I}_{II} =   \Pi[ \partial^\beta_y (\kappa+\phi) , \partial_y^{\gamma_1} \vartheta_B, \ldots , \partial_y^{\gamma_{\ell}} \vartheta_B] [T_{\vartheta_A} \zeta^{(i+\ell)} - T_{\vartheta_B} \zeta^{(i+\ell)} ],
\end{equation}
with the same conditions on $\beta$ and $\{ \gamma_i \}_{i=1}^{\ell}$, but where now $\ell = 0$ is also allowed.

Exploiting the invariance of the $x$-integral with respect to the translation $T_{\vartheta_A}$, we obtain
\begin{equation}
    \| \mathcal{I}_{I} \|_{L^2}
      \le K (|\kappa| + \| \phi\|_{H^k_y} ) \| \vartheta_A - \vartheta_B \|_{H^k_y}\| \big( \| \vartheta_A\|_{H^k_y}^{\ell-1} + \| \vartheta_B\|_{H^k_y}^{\ell-1} \big) \| \zeta^{(i+\ell)} \|_{L^2_x},
\end{equation}
which can be absorbed in \eqref{eq:nl:bnd:prod:hky:hkx:lip}.
Appealing to \eqref{eq:nl:lip:bnd:l2:x:diff:shift}, 
we also obtain
\begin{equation}
    \| \mathcal{I}_{II} \|_{L^2}
          \le K (|\kappa| + \| \phi\|_{H^k_y} )\| \vartheta_A - \vartheta_B \|_{L^2_y}
          \| \vartheta_B\|_{H^k_y}^\ell \| \|\zeta^{(i + \ell + 1)} \|_{L^2_x},
\end{equation}
which completes the proof.
\end{proof}

\begin{lemma}
Pick $k > (d-1)/2$. There exists a constant $K > 0$ so that for every 
$\phi \in H^k_y$ and $ \zeta \in H^k_x$
we have
\begin{equation}
\label{eq:nl:bnd:prod:hky:hkx}
    \| \phi T_{\vartheta} \zeta \|_{H^k}
    \le K  \| \phi \|_{H^k_y} (1 + \|\vartheta\|^k_{H^k_y} ) \| \zeta\|_{H^k_x}.
\end{equation}
\end{lemma}
\begin{proof}
Pick a multi-index $\alpha \in \mathbb Z^{d-1}_{\ge 0}$ and an integer $i \ge 0$ with $|\alpha| + i \le k$. Again writing $\Pi$ for the pointwise product,
we observe that $\partial^\alpha_y \partial_x^i [ \phi T_{\vartheta} \zeta]$
can be expressed as a finite sum of terms of the form
\begin{equation}
\mathcal{I} =     \Pi[ \partial^\beta_y \phi , \partial_y^{\gamma_1} \vartheta, \ldots , \partial_y^{\gamma_{\ell}} \vartheta] T_{\vartheta} \zeta^{(i+\ell)},
\end{equation}
with $0 \le \ell \le |\alpha|$
and
multi-indices $\beta \in \mathbb Z^{d-1}_{\ge 0}$
and $\{ \gamma_i \}_{i=1}^{\ell} \in \mathbb Z^{d-1}_{\ge 0}$
that satisfy 
$|\gamma_i| \ge 1 $ for integers $i \in \{1, \ldots, \ell\}$,
together with
\begin{equation}
 |\beta| + |\gamma_1| + \ldots + |\gamma_\ell| = |\alpha|.
\end{equation}
Exploiting the invariance of the $x$-integral with respect to the translation $T_{\vartheta}$,
we obtain
\begin{equation}
    \| \mathcal I \|_{L^2}
    \le  K \| \phi\|_{H^k_y} \| \vartheta\|_{H^k_y}^{\ell}  \| \zeta^{(i+\ell)} \|_{L^2_x},
\end{equation}
which can be absorbed in \eqref{eq:nl:bnd:prod:hky:hkx}.
\end{proof}

\subsection{Proofs of Propositions \ref{prop:nl:bnds} and \ref{prop:nl:bnds:lip}}

We shall now systematically derive
bounds for the expressions \eqref{eq:nl:def:k:2} and \eqref{eq:nl:def:j:jtr} that appear in the definitions
\eqref{eq:nl:def:n:m} and \eqref{eq:nl:def:cal:n:m},
enabling us to establish the desired main estimates. Let us remark that several of the uniform bounds for the deterministic quantities are similar to those in \cite{kapitula1997}. We start by considering the auxiliary function $\Upsilon$, which has $\Upsilon(0) = -1$. In particular, the bound \eqref{eq:nl:est:k2:hk:lip} below
allows us to conclude that
\begin{equation}
\label{eq:nl:est:k2:hk}
    \| \Upsilon(\vartheta) + 1\|_{H^k_y} \le K^N 
\end{equation}
holds for some $K^N>0$ for all $\vartheta \in H^k_y$ with $\|\vartheta\|_{H^k_y} \le N$.

\begin{lemma}
\label{lem:nl:bnds:K2}
Pick $k > (d-1)/2$ 
and suppose that \textnormal{(Hf)} and \textnormal{(HTw)} both hold. 
There for all $N>0$ there exists a constant $K^N > 0$ so that the bound
\begin{equation}
\label{eq:nl:est:k2:hk:lip}
        \| \Upsilon(\vartheta_A) - \Upsilon(\vartheta_B) \|_{H^k_y}
    \le K^N  \| \vartheta_A - \vartheta_B \|_{H^k_y}
\end{equation}
holds for all $\vartheta_A, \vartheta_B \in H^k_y$ with $\|\vartheta_A\|_{H^k_y} \le N$ and $\|\vartheta_B\|_{H^k_y} \le N$.
\end{lemma}
\begin{proof}
Consider the scalar function \begin{equation}\Upsilon_{I}: \mathbb R \to\mathbb R, \quad \varrho \mapsto \langle T_{\varrho} \Phi_0', \psi_{\rm tw} \rangle_{L^2_x},\end{equation} which 
satisfies $\Upsilon_{I}(0) = 0$ and has
derivatives $
    \Upsilon_{I}^{(j)}(\varrho) = \langle T_{\varrho} \Phi_0^{(1 + j)}, \psi_{\rm tw} \rangle_{L^2_x},$
for $1 \le j \le k+1$. It is  $C^{k+1}$-smooth with uniformly bounded derivatives. In addition, let us introduce the scalar functions
\begin{equation}
\begin{aligned}
    &\Upsilon_{II}: \mathbb R\to [-\tfrac34,\infty), \quad z\mapsto \chi(z)-1,\\
    &\Upsilon_{III}: [-\tfrac{3}{4}, \infty)\to \mathbb [-3,1], \quad z  \mapsto 1 -1/(1 + z),
\end{aligned}\end{equation}
which are $C^\infty$-smooth. 
Observe that for $\theta\in H^k_y$, the identity
\begin{equation}
    \Upsilon(\vartheta) + 1 = \Upsilon_{III}(\Upsilon_{II}(\Upsilon_{I}(\vartheta)))
\end{equation}
holds, pointwise with respect to $y$. The function $\mathbb R\ni \varrho \mapsto \Upsilon(\varrho)+1 \in [-2,2]$ is globally Lipschitz, including its first $k$ derivatives. An appeal to \cite[eq. (4.8)--(4.10)]{bosch2025conditionalspeedshapecorrections}
with $\Phi=0$ and $j=0$
 yields the bound
\begin{equation}
    \| \Upsilon(\vartheta_A) - \Upsilon(\vartheta_B) \|_{H^k_y}\le 
    K \| \vartheta_A - \vartheta_B \|_{H^k} ( 1 +  \|\vartheta_A\|_{H^k_y}^{k-1} + \|\vartheta_B\|_{H^k_y}^k ),
\end{equation}    
 for some $K>0,$ which in turn can be simplified to \eqref{eq:nl:est:k2:hk:lip}
on account of the a priori bounds on $\vartheta_A$ and $\vartheta_B$.
\end{proof}

We now turn to the nonlinearity $g$ that drives the noise in the original system \eqref{eq:reaction:diff}. The main point of the estimates below is to apply 
the assumptions \textnormal{(HgQ)} and \textnormal{(H$L^1$)} to the situation where
the global phase-shift $T_{\gamma} \Phi_0$ is replaced by $y$-dependent shifts.

\begin{lemma}\label{lem:gHS:pert}
     Let $k >d/2$ and suppose that \textnormal{(HTw)} and \textnormal{(HgQ)} hold.
     For all $N>0$ there exists a constant $K^N>0$ so that 
     for all $\gamma \in \mathbb R$, and every $v_A, v_B \in H^k$ and $\vartheta_A, \vartheta_B \in H^k_y$ with 
     \begin{equation}
         \max\{ \| v_A \|_{H^k}, \|v_B\|_{H^k}, \|\vartheta_A\|_{H^k_y},
          \|\vartheta_B\|_{H^k_y} \} \le N,
     \end{equation}
     we have the bound
\begin{equation} 
\label{eq:gHSLip:pert}
\|[g(v_A+T_{\vartheta_A}\Phi_0)-g(v_B+T_{\vartheta_B}\Phi_0)]T_{-\gamma}\|_{HS(\mathcal W_Q;H^k)}
\leq K^N\big[ \|v_A-v_B\|_{H^k}+\|\vartheta_A-\vartheta_B\|_{H^k_y} \big].
\end{equation}
\end{lemma}
\begin{proof}
For any $\xi \in \mathcal W_Q$, we define
\begin{equation}
    \Delta_{AB} g[\xi] =  [g(v_A+T_{\vartheta_A}\Phi_0)-g(v_B+T_{\vartheta_B}\Phi_0)]T_{-\gamma} \xi,
\end{equation}
and use the commutation relation in (HgQ) to compute
\begin{equation}
T_{\gamma} \Delta_{AB} g[\xi] = [g(T_\gamma v_A + T_{\vartheta_A + \gamma } \Phi_0)
- g( T_{\gamma} v_B + T_{\vartheta_B + \gamma} \Phi_0 ) ]\xi
\end{equation}
In particular, writing
\begin{equation}
    \tilde{v}_A = T_{\gamma} v_A + T_{\vartheta_A + \gamma} \Phi_0 - T_{\gamma} \Phi_0, \qquad 
    \tilde{v}_B = T_{\gamma} v_B + T_{\vartheta_B + \gamma} \Phi_0 - T_{\gamma} \Phi_0, 
\end{equation}
we see that
\begin{equation}
T_{\gamma} \Delta_{AB} g[\xi] = [g(T_\gamma \Phi_0 + \tilde{v}_A)
- g( T_{\gamma} \Phi_0 + \tilde{v}_B) ]\xi.
\end{equation}
Applying \eqref{eq:nl:bnd:prod:hky:hkx:lip} we find
constants $C_1^N > 0$ and $C_2^N > 0$ for which the uniform bounds
\begin{equation}
    \| \tilde{v}_A \|_{H^k} \le \|v_A\|_{H^k} + C_1^N \| \vartheta_A\|_{H^k_y}
    \le C_2^N,
    \qquad 
     \| \tilde{v}_B \|_{H^k} \le \|v_B\|_{H^k} + C_1^N \| \vartheta_B\|_{H^k_y}
     \le C_2^N
\end{equation}
and the Lipschitz estimate
\begin{equation}
    \| \tilde{v}_A - \tilde{v}_B \|_{H^k} \le \|v_A - v_B\|_{H^k} + C_1^N \| \vartheta_A - \vartheta_B\|_{H^k_y}
\end{equation}
hold. In here, we utilise the property $\| T_{-\gamma} v\|_{H^k}=\| v\|_{H^k}$, which holds for any $v\in H^k$  as a result of the translation invariance of integration. Applying \eqref{eq:HgLip}, we may hence compute
\begin{equation}
\begin{array}{lcl}
    \| \Delta_{AB} g \|_{HS(\mathcal W_Q; H^k)}^2
    & = & \sum_{j=0}^\infty \| \Delta_{AB} g[ \sqrt{Q} e_j] \|_{H^k}^2
    \\[0.2cm]
    & = & \sum_{j=0}^\infty \| T_{\gamma} \Delta_{AB} g[ \sqrt{Q} e_j] \|_{H^k}^2
    \\[0.2cm]
    & \le & C_3^N \| \tilde{v}_A - \tilde{v}_B \|_{H^k}^2 
    \\[0.2cm]
    & \le & C_4^N [ \|v_A - v_B\|_{H^k} + \| \vartheta_A - \vartheta_B\|_{H^k_y} ]^2
\end{array}
\end{equation}
for some $C_3^N > 0$ and $C_4^N > 0$, completing the proof.
\end{proof}

\begin{corollary}
\label{cor:gHS:pert}
     Let $k >d/2$ and suppose that \textnormal{(HTw)}, \textnormal{(HgQ)}, 
     and \textnormal{(H$L^1$)} hold. For all $N>0$ there exists a constant $K^N>0$ so that 
     for all $\gamma \in \mathbb R$, and every $v\in H^k$ and $\vartheta\in H^k_y$ with 
     $\|v\|_{H^k} \le N$ and $\| \vartheta \|_{H^k_y} \le N$,
     we have the bounds
\begin{equation} 
\label{eq:nl:g:unif:bnd}
\begin{array}{lcl}
\|g(v+T_{\vartheta}\Phi_0)T_{-\gamma}\|_{HS(\mathcal W_Q;H^k)}
& \leq & K^N,
\\[0.2cm]
\|\psi_{\rm tw} ^\top g(v+T_{\vartheta}\Phi_0)T_{-\gamma}\|_{\Pi_2(\mathcal W_Q;L^1(\mathbb R^d;\mathbb R))}
& \leq & K^N.
\end{array}
    \end{equation}
\end{corollary}
\begin{proof}
Picking any $\xi \in \mathcal W_Q$ and writing
\begin{equation}
    \tilde{v} =  T_{\gamma} v + T_{\vartheta + \gamma} \Phi_0 - T_{\gamma} \Phi_0,
\end{equation}
we note that
\begin{equation}
    g(v + T_{\vartheta} \Phi_0) [T_{-\gamma} \xi] = 
    T_{-\gamma} \big[g( T_{\gamma} v + T_{\vartheta + \gamma} \Phi_0) [\xi]\big]
    = T_{-\gamma} \big[g( T_\gamma \Phi_0 + \tilde{v} )[\xi]\big]
\end{equation}
holds, and similarly we have
\begin{equation}
    \psi_{\rm tw}^\top  g(v + T_{\vartheta} \Phi_0) [T_{-\gamma} \xi]
    = T_{-\gamma} \big[   g( T_\gamma \Phi_0 + \tilde{v} )[\xi] ^\top T_{\gamma} \psi_{\rm tw} \big].
\end{equation}
By applying \eqref{eq:nl:bnd:prod:hky:hkx:lip} we find the uniform bound
\begin{equation}
    \| \tilde{v} \|_{H^k} \le \| v\|_{H^k} + C_1^N \| \vartheta \|_{H^k_y} \le C_2^N,
\end{equation}
for some constants $C_1^N > 0$ and $C_2^N > 0$, allowing us to read off
the desired bounds directly from \eqref{eq:Hg} and \eqref{eq:int:cond}, respectively.
\end{proof}

We now turn to the nonlinearities $M$ and $\mathcal M$,
noting that the former can be written as
\begin{equation}
\label{eq:nl:decomp:M}
    M = (\Upsilon(\vartheta) + 1) \langle g(v + T_{\vartheta} \Phi_0), \psi_{\rm tw} \rangle_{L^2_x}
    - \langle g(v + T_{\vartheta} \Phi_0), \psi_{\rm tw} \rangle_{L^2_x}.
\end{equation}
In view of \eqref{eq:nl:est:k2:hk}, this representation allows us to exploit the fact that $H^k_y$ is an algebra.

\begin{lemma}\label{lem:nl:m}
 Let $k >d/2$ and suppose that \textnormal{(Hf)}, \textnormal{(HTw)},  \textnormal{(HgQ)}, 
 and \textnormal{(H$L^1$)}  hold. For all $N>0$ there exists a constant $K^N > 0$
 so that for all $\gamma \in \mathbb R$, and every $v \in H^k$ and  $\vartheta \in H^k_y$ with $\|v\|_{H^k} \le N$ and $\|\vartheta\|_{H^k_y} \le N$,
 we have the bounds
\begin{equation}
\label{eq:nl:bnds:m:cal:m:unif}
    \begin{array}{lcl}
       \left\|M(v,\theta)T_{-\gamma}\right\|_{HS(\mathcal W_Q;H^k_y )}
       + \left\|M(v,\theta)T_{-\gamma}\right\|_{\Pi_2(\mathcal W_Q;L^1_y)} & \le & K^N ,
       \\[0.2cm]
       \left\|\mathcal M(v,\theta)T_{-\gamma}\right\|_{HS(\mathcal W_Q;H^k )}
        & \le & K^N.
    \end{array}
\end{equation}
\end{lemma}
\begin{proof}
For any $\xi \in \mathcal W_Q$, 
we may apply \eqref{eq:nl:bnd:l2x:to:hk}
and \eqref{eq:nl:est:k2:hk}
to the decomposition \eqref{eq:nl:decomp:M} to obtain 
\begin{equation}
\begin{array}{lcl}
    \| M(v, \vartheta)[T_{-\gamma} \xi] \|_{H^k_y}
    & \le & C_1^N
    \big( \| \Upsilon(\vartheta) + 1 \|_{H^k_y} + 1 \big)
    \| g(v  + T_\vartheta \Phi_0)[T_{-\gamma} \xi] \|_{H^k} \|\psi_{\rm tw}\|_{L^2_x}
    \\[0.2cm]
    & \le & C_2^N \| g(v  + T_\vartheta \Phi_0)[T_{-\gamma} \xi] \|_{H^k},
\end{array}
\end{equation}
for some $C_1^N > 0$ and $C_2^N>0$. Furthermore, a direct estimate implies
\begin{equation}
\begin{array}{lcl}
    \| M(v, \vartheta)[T_{-\gamma} \xi] \|_{L^1_y} 
    & \le & \|\Upsilon(\vartheta)\|_{\infty}
    \| \langle g(v + T_{\vartheta} \Phi_0)[T_{-\gamma} \xi] , \psi_{\rm tw} \rangle_{L^2_x} \|_{L^1_y} \\[0.2cm]
    & \le & \|\Upsilon(\vartheta)\|_{\infty}
    \|  g(v + T_{\vartheta} \Phi_0)[T_{-\gamma} \xi]^\top  \psi_{\rm tw} \|_{L^1(\mathbb R^d;\mathbb R)} .
\end{array}
\end{equation}
Similarly, we may apply the product estimate
\eqref{eq:nl:bnd:prod:hky:hkx} to the definition \eqref{eq:nl:def:cal:n:m}
to obtain
\begin{equation}
    \| \mathcal{M}(v, \vartheta) T_{-\gamma} \xi \|_{H^k}
     \le  \| g( v + T_{\vartheta} \Phi_0)T_{-\gamma} \xi \|_{H^k}
    + C_3^N \| M(v, \vartheta) T_{-\gamma} \xi \|_{H^k_y} \| \Phi_0'\|_{H^k_x},
\end{equation} for some $C_3^N>0.$
The desired bounds are then inferred from \eqref{eq:nl:g:unif:bnd} by taking $\xi = \sqrt{Q} e_j$, squaring, and summing over all $j \ge 0$, where $(e_j)_{j\geq 0}$ is any orthonormal basis of $\mathcal W$.
\end{proof}

We are now ready to establish Lipschitz estimates for $M$ and $\mathcal{M}$. 
In order to exploit the algebra structure of $H^k_y$, we will use
the convenient representation
\begin{equation}
\label{eq:nl:lip:id:m}
\begin{aligned}
    M(v_A, \vartheta_A)-M(v_B, \vartheta_B)
    &= \big(\Upsilon(\vartheta_A)- \Upsilon( \vartheta_B) \big)  \langle g(v_A + T_{\vartheta_A} \Phi_0), \psi_{\rm tw} \rangle_{L^2_x}
    \\&  \qquad
    + (\Upsilon(\vartheta_B) + 1)
    \langle g(v_A + T_{\vartheta_A} \Phi_0) - 
    g(v_B + T_{\vartheta_B} \Phi_0) , \psi_{\rm tw} \rangle_{L^2_x}
    \\ & \qquad 
    - \langle g(v_A + T_{\vartheta_A} \Phi_0) - 
    g(v_B + T_{\vartheta_B} \Phi_0) , \psi_{\rm tw} \rangle_{L^2_x},
\end{aligned}
\end{equation}
together with its counterpart
\begin{equation}
\label{eq:nl:lip:id:cal:m}
\begin{aligned}
    \mathcal M(v_A, \vartheta_A)
    - \mathcal M(v_B, \vartheta_B)
& = g(v_A + T_{ \vartheta_A} \Phi_0) - g(v_B + T_{\vartheta_B} \Phi_0)
\\ & \qquad
+ \big( M(v_A, \vartheta_A) - M(v_B, \vartheta_B) \big) T_{\vartheta_A} \Phi_0'
\\& \qquad
 + M(v_B, \vartheta_B) \big( T_{\vartheta_A}  - T_{\vartheta_B} \big) \Phi_0'.
\end{aligned}
\end{equation}

\begin{lemma}
 Let $k >d/2$ and suppose that \textnormal{(Hf)}, \textnormal{(HTw)},  and \textnormal{(HgQ)} hold. There exists a constant $K^N > 0$
so that
     for all $\gamma \in \mathbb R$, and every   $v_A, v_B \in H^k$ and $\vartheta_A, \vartheta_B \in H^k_y$ with 
     \begin{equation}
         \max\{ \| v_A \|_{H^k}, \|v_B\|_{H^k} , \|\vartheta_A\|_{H^k_y},
          \|\vartheta_B\|_{H^k_y} \} \le N,
     \end{equation}
we have the bounds     
\begin{equation}
\label{eq:nl:bnd:m:calm:lip}
\begin{array}{lcl}
    \| [M(v_A, \vartheta_A) - M(v_B, \vartheta_B)]T_{-\gamma} \|_{HS(\mathcal W_Q;H^k_y)} & \le & K^N\big[ \|v_A - v_B\|_{H^k} + \| \vartheta_A - \vartheta_B \|_{H^k_y}   \big] ,
\\[0.2cm]
    \| [\mathcal M(v_A, \vartheta_A) - \mathcal M(v_B, \vartheta_B)] T_{-\gamma} \|_{HS(\mathcal W_Q;H^k)} & \le & K^N\big[ \|v_A - v_B\|_{H^k} + \| \vartheta_A - \vartheta_B \|_{H^k_y}   \big].
\end{array}
\end{equation}
\end{lemma}
\begin{proof}
Upon writing
\begin{equation}
    \Delta_{AB} M = [M(v_A, \vartheta_A) - M(v_B, \vartheta_B)]T_{-\gamma},
\end{equation}
we may use the representation \eqref{eq:nl:lip:id:m} 
together with the bounds \eqref{eq:nl:bnd:l2x:to:hk} and \eqref{eq:nl:est:k2:hk} to obtain
\begin{equation}
    \begin{array}{lcl}
    \| \Delta_{AB} M \|_{HS(\mathcal W_Q; H^k_y)}
    & \le & C_1^N \| \Upsilon(\vartheta_A)- \Upsilon( \vartheta_B) \|_{H^k_y}
\|g(v_A + T_{\vartheta_A} \Phi_0)T_{-\gamma}\|_{HS(\mathcal W_Q;H^k)}
\\[0.2cm]
& &  \qquad
+\,C_1^N 
 \| [g(v_A + T_{\vartheta_A} \Phi_0) - 
    g(v_B + T_{\vartheta_B} \Phi_0)]T_{-\gamma} \|_{HS(\mathcal W_Q;H^k)},
    \end{array}
\end{equation}
for some $C_1^N > 0$. The desired 
estimate now follows 
 from \eqref{eq:nl:est:k2:hk:lip},
 \eqref{eq:gHSLip:pert}, and \eqref{eq:nl:g:unif:bnd}.
Turning to the difference
\begin{equation}
    \Delta_{AB} \mathcal M = [ \mathcal M(v_A, \vartheta_A) - \mathcal M(v_B, \vartheta_B)]T_{-\gamma} ,
\end{equation}
we invoke \eqref{eq:nl:bnd:prod:hky:hkx:lip}
and 
\eqref{eq:nl:bnd:prod:hky:hkx}
together with the a priori bounds on $\vartheta_A$ and $\vartheta_B$
to compute
\begin{equation}
\begin{array}{lcl}
   \| \Delta_{AB} \mathcal{M} \|_{HS(\mathcal W_Q; H^k)}
& \le & \|g(v_A + T_{ \vartheta_A} \Phi_0) - g(v_B + T_{\vartheta_B} \Phi_0) \|_{HS(\mathcal W_Q; H^k)}
\\[0.2cm]
& & \qquad
+ \,C_2^N\, \| M(v_A, \vartheta_A) - M(v_B, \vartheta_B) \|_{HS(\mathcal W_Q; H^k_y)} \| \Phi_0' \|_{H^k_x}
\\[0.2cm]
& & \qquad
 + \,C_2^N \,\| M(v_B, \vartheta_B) \|_{HS(\mathcal W_Q; H^k)}
 \| \vartheta_A - \vartheta_B\|_{H^k_y} \|\Phi_0''\|_{H^k_x},
\end{array}
\end{equation} for some constant $C_2^N>0.$
Applying \eqref{eq:gHSLip:pert}, \eqref{eq:nl:bnds:m:cal:m:unif},
and the first bound in \eqref{eq:nl:bnd:m:calm:lip}
leads to the second bound.
\end{proof}

We proceed with the function $\mathcal{J}$, defined in \eqref{eq:nl:def:j:jtr}, for which we introduce the expressions
\begin{equation}
\begin{array}{lcl}
    \mathcal{J}_I(v, \vartheta) & = &
      f(v+T_\theta\Phi_0)-f(T_\theta\Phi_0)-Df(T_\theta\Phi_0)v ,
    \\[0.2cm]
    \mathcal{J}_{II}(v, \vartheta) & = &
     [Df(T_\theta\Phi_0)-Df(\Phi_0)]v ,
    \\[0.2cm]
    \mathcal{J}_{III}(\vartheta) & = &
    |\nabla_y\theta|^2T_\theta\Phi_0'',
\end{array}
\end{equation}
allowing us to write
\begin{equation}
    \mathcal{J} = \mathcal{J}_{I} + \mathcal{J}_{II} + \mathcal{J}_{III} .
\end{equation}
Setting out to establish Lipschitz bounds for $\mathcal{J}_I$, we introduce the notation
\begin{equation}
\begin{array}{lcl}
    \Delta_{AB;v} \mathcal{J}_I[\vartheta] 
    &=& \mathcal{J}_I(v_A, \vartheta) - \mathcal{J}_I(v_B, \vartheta) ,
\\[0.2cm]
  \Delta_{AB;\vartheta} \mathcal{J}_I[v]
   & = & \mathcal{J}_I(v, \vartheta_A) - \mathcal{J}_I(v, \vartheta_B) ,
\end{array}
\end{equation}
yielding
\begin{equation}
    \mathcal{J}_I(v_A, \vartheta_A) - \mathcal{J}_I(v_B, \vartheta_B)
    = \Delta_{AB;v}\mathcal{J}_I[\vartheta_A]
    + \Delta_{AB;\vartheta}\mathcal{J}_I[v_B].
\end{equation}
Applying Taylor's theorem results into the representation
\begin{equation}
\label{eq:nl:def:d:j1:ab:v}
\begin{aligned} 
 \Delta_{AB;v} \mathcal{J}_I[\vartheta] 
 & = 
  \int_0^1 Df( T_{\vartheta} \Phi_0 + v_B + t (v_A - v_B)) [v_A - v_B] \, \mathrm dt - Df( T_{\vartheta} \Phi_0) [v_A - v_B]
 \\
 & =  
 \int_0^1 \int_0^1 D^2 f\big( T_{\vartheta} \Phi_0 + s(v_B + t(v_A - v_B) \big)[ v_B + t(v_A - v_B),  v_A - v_B] \, \mathrm ds \, \mathrm dt,
\end{aligned}
\end{equation}
together with its counterpart
\begin{align} \label{eq:nl:def:d:j1:ab:theta}
 \Delta_{AB;\vartheta} \mathcal{J}_I[v] 
 & = 
     \int_0^1 \big( Df ( T_{\vartheta_A} \Phi_0 + tv)[v] - 
             Df ( T_{\vartheta_B} \Phi_0 + tv)[v]  \big) \, \mathrm dt
             - Df ( T_{\vartheta_A} \Phi_0 )[v]  
             + Df ( T_{\vartheta_B} \Phi_0 )[v] \nonumber
\\ 
&  =  
    \int_0^1 \int_0^1 D^2f ( T_{\vartheta_B} \Phi_0 + tv + s( T_{\vartheta_A} \Phi_0 - T_{\vartheta_B}\Phi_0))[T_{\vartheta_A} \Phi_0 - T_{\vartheta_B}\Phi_0,v]  \, \mathrm ds \, \mathrm dt\nonumber 
\\& \qquad\qquad\qquad\quad
             - \int_0^1 D^2f ( T_{\vartheta_B} \Phi_0 + s( T_{\vartheta_A} \Phi_0 - T_{\vartheta_B}\Phi_0)[T_{\vartheta_A} \Phi_0 - T_{\vartheta_B}\Phi_0,v]  
             \, \mathrm ds .
\end{align} 
The latter expression could in principle be reduced further to a single triple-integral involving third derivatives of $f$. However, we refrain from doing so, since it would  require additional smoothness on $f$ while it is also not necessary for our purposes in this paper.

\begin{lemma}\label{lem:d2:f:v}
Let $k > d/2$ and assume that \textnormal{(Hf)} and \textnormal{(HTw)} hold.
For all $N>0$ there exists a constant $K^N > 0$ so that for any $\vartheta \in H^k_y$ and $w, v_1, v_2 \in H^k$ with $\| \vartheta \|_{H^k_y} \le N$ and $\|w\|_{H^k} \le N$ we have the bound
\begin{equation}
\label{eq:nl:est:d2:f:v1:v2}
   \| D^2 f( T_{\vartheta} \Phi_0 + w) [v_1, v_2 ] \|_{H^k}
   \le K^N \|v_1\|_{H^k} \|v_2\|_{H^k} .
\end{equation}    
\end{lemma}
\begin{proof}
Choosing a multi-index $\alpha \in \mathbb Z^{d}_{\ge 0}$,
we note that $\partial^\alpha D^2 f( T_{\vartheta} \Phi_0 + w) [v_1, v_2 ]$ can be expressed as a finite sum of terms of the form
\begin{equation}
\mathcal{I} =    D^{2 + \ell} f( T_{\vartheta} \Phi_0 + w )
    [ \partial^{\beta_1} [ T_{\vartheta} \Phi_0 + w] ,
    \ldots ,\partial^{\beta_{\ell}} [ T_{\vartheta} \Phi_0 + w], \partial^{\beta_{\ell +1}}  v_1, \partial^{\beta_{\ell+2}} v_2 ],
\end{equation}
with $0 \le \ell \le |\alpha|$ and multi-indices $\{ \beta_i\}_{i=1}^{\ell+2} \in \mathbb Z^d_{\ge 0}$ that satisfy
$|\beta_i| \ge 1$ for integers $i \in \{1, \ldots, \ell\}$ together with
\begin{equation}
    |\beta_1| + \ldots + |\beta_{\ell + 2}| = |\alpha|.
\end{equation}
Fix any $i \in \{1, \ldots, \ell\}$. If $\partial^{\beta_i} = \partial_x^{|\beta_i|}$,
then we find
\begin{equation}
    \| \partial^{\beta_i} T_{\vartheta} \Phi_0 \|_{\infty}
    =  \|\Phi_0^{(|\beta_i|)} \|_{\infty} < \infty.
\end{equation}
On the other hand, if $\partial^{\beta_i}$ contains derivatives related to $y$, we deduce that
\begin{equation}
    \| \partial^{\beta_i} T_{\vartheta} \Phi_0 \|_{L^2}
    = \| \partial^{\beta_i} [T_{\vartheta} \Phi_0 - \Phi_0 ] \|_{L^2}
    \le K \|\vartheta\|_{H^k_y}(1 +  \|\vartheta\|_{H^k_y}^{k-1} )\|\Phi_0'\|_{H^k_x} \le  C_1^N
\end{equation}
holds for some $C_1^N > 0$,
as a result of \eqref{eq:nl:bnd:prod:hky:hkx:lip}
and the a priori bound on $\vartheta$.
In particular, using the a priori bound on $w$, we obtain
\begin{equation}
    \| \mathcal{I} \|_{L^2} \le C_2^N \| v_1 \|_{H^k} \|v_2 \|_{H^k},
\end{equation}
for some $C_2^N > 0$, from which the claim follows.
\end{proof}

\begin{lemma}
\label{lem:nl:bnds:j:i}
Let $k > d/2$ and assume that \textnormal{(Hf)} and \textnormal{(HTw)} hold. For all $N>0$ there exists a constant $K^N > 0$
so that for any $v \in H^k$ and $\vartheta \in H^k_y$ with $\|v\|_{H^k} \le N$ and $\|\vartheta\|_{H^k_y} \le N$
we have the bounds 
\begin{equation}
\label{eq:nl:bnd:j:i:l1:hk}
\begin{array}{lcl}
\| \mathcal{J}_I(v, \vartheta) \|_{L^1}
& \le & K^N \|v \|_{L^2}^2, \\[0.2cm]
\| \mathcal{J}_I(v, \vartheta) \|_{H^k} 
&\le& K^N \| v \|_{H^k}^2 ,
\end{array}
\end{equation}
while for any $v_A, v_B \in H^k$ and $\vartheta_A, \vartheta_B \in H^k_y$ with \begin{equation}
    \max\{ \|v_A\|_{H^k}, \|v_B\|_{H^k}, \|\vartheta_A\|_{H^k_y}, \|\vartheta_B\|_{H^k_y} \} \le N
\end{equation}  
we have
\begin{equation}
\label{eq:nl:j:i:lip}
    \| \mathcal{J}_I(v_A, \vartheta_A) -  \mathcal{J}_I(v_B, \vartheta_B) \|_{H^k}  
    \le K^N \big[ \|v_A\|_{H^k} + \|v_B\|_{H^k} \big]
    \big[ \| v_A - v_B \|_{H^k} + \|\vartheta_A - \vartheta_B\|_{H^k}
    \big].
\end{equation}
\end{lemma}
\begin{proof}
Starting with the Lipschitz bound \eqref{eq:nl:j:i:lip},
we utilise \eqref{eq:nl:bnd:prod:hky:hkx:lip} and \eqref{eq:nl:est:d2:f:v1:v2} to the representations \eqref{eq:nl:def:d:j1:ab:v} and \eqref{eq:nl:def:d:j1:ab:theta}
to find
\begin{equation}
\begin{array}{lcl}
 \|    \Delta_{AB;v} \mathcal{J}_I[\vartheta_A] \|_{H^k}  &\le & C_1^N \big[ \|v_A\|_{H^k} + \|v_B\|_{H^k} \big] \| v_A - v_B \|_{H^k} ,
\\[0.2cm]
\| \Delta_{AB;\vartheta} \mathcal{J}_I[v_B]  \|_{H^k}
&\le & C_2^N 
    \| T_{\vartheta_A} \Phi_0 - T_{\vartheta_B} \Phi_0 \|_{H^k} \| v_B\|_{H^k}
\\[0.2cm]
& \le & 
     C_3^N \| \vartheta_A - \vartheta_B \|_{H^k} \|v_B\|_{H^k}
\end{array}
\end{equation}
for some $C_1^N > 0$, $C_2^N > 0$, and $C_3^N > 0$, which can be absorbed in the stated estimate.
The $H^k$-bound in \eqref{eq:nl:bnd:j:i:l1:hk}
follows directly from \eqref{eq:nl:j:i:lip} by noting  $\mathcal{J}_I(0, \vartheta) = 0$, while the 
remaining $L^1$-bound follows from the pointwise
estimate $|\mathcal{J}_I(v, \vartheta)| \le K |v|^2$, for some $K>0$.
\end{proof}

\begin{lemma}
Let $k > d/2$ and assume that \textnormal{(Hf)} and \textnormal{(HTw)} hold. For all $N>0$ there exists a constant $K^N > 0$
so that for any $v \in H^k$ and $\vartheta \in H^k_y$ with $\|v\|_{H^k} \le N$ and $\|\vartheta\|_{H^k_y} \le N$
we have the bounds 
\begin{equation}
\label{eq:nl:bnd:j:ii:l1:hk}
\begin{array}{lcl}
\| \mathcal{J}_{II}(v, \vartheta) \|_{L^1}
& \le & K^N \| \vartheta \|_{L^2_y} \| v \|_{L^2}, \\[0.2cm]
\| \mathcal{J}_{II}(v, \vartheta) \|_{H^k} 
&\le& K^N \| \vartheta \|_{H^k_y} \| v \|_{H^k} ,
\end{array}
\end{equation}
while for any $v_A, v_B \in H^k$ and $\vartheta_A, \vartheta_B \in H^k_y$ with \begin{equation}
    \max\{ \|v_A\|_{H^k}, \|v_B\|_{H^k}, \|\vartheta_A\|_{H^k_y}, \|\vartheta_B\|_{H^k_y} \} \le N
\end{equation}  
we have
\begin{equation}
\label{eq:nl:j:ii:lip}
    \| \mathcal{J}_{II}(v_A, \vartheta_A) -  \mathcal{J}_{II}(v_B, \vartheta_B) \|_{H^k}  
    \le K^N \big[ \|v_A\|_{H^k} + \|\vartheta_B\|_{H^k} \big]
    \big[ \| v_A - v_B \|_{H^k} + \|\vartheta_A - \vartheta_B\|_{H^k}
    \big] .
\end{equation}
\end{lemma}
\begin{proof}
Starting with the Lipschitz estimate \eqref{eq:nl:j:ii:lip},
we write
\begin{equation}
\begin{array}{lcl}
\Delta_{AB} \mathcal{J}_{II} 
 &= & \mathcal{J}_{II}(v_A, \vartheta_A) -  \mathcal{J}_{II}(v_B, \vartheta_B)
\\[0.2cm]
& = & [ Df( T_{\vartheta_A} \Phi_0) - Df (T_{\vartheta_B} \Phi_0) ] v_A
    + [Df(T_{\vartheta_B} \Phi_0) - Df(\Phi_0) ] (v_A - v_B).
\end{array}
\end{equation}
Applying \eqref{eq:nl:bnd:prod:hky:hkx:lip} and exploiting the algebra structure of 
$H^k$, we compute
\begin{equation}
    \| \Delta_{AB} \mathcal{J}_{II} \|_{H^k} \le 
    K \| D^2 f(\Phi_0) \Phi_0' \|_{H^k(\mathbb R;\mathbb R^{n\times n})}
    \big[
    \|\vartheta_A - \vartheta_B\|_{H^k_y}
    \|v_A\|_{H^k}
    + \| \vartheta_B\|_{H^k_y} \|v_A-v_B\|_{H^k}\big],
\end{equation}
for some constant $K^N>0,$ which implies the stated bound after updating $K^N>0$. The $H^k$-bound in
\eqref{eq:nl:bnd:j:ii:l1:hk} follows  from
\eqref{eq:nl:j:ii:lip} by taking $v_A = 0$ and $\vartheta_A = \vartheta_B$.
Finally, we use
\eqref{eq:nl:bnd:shift:diff:l2:lip} to compute
\begin{equation}
\begin{array}{lcl}
 \| \mathcal{J}_{II}(v, \vartheta) \|_{L^1} 
 & \le & 
 \| Df(T_{\vartheta} \Phi_0) - Df(\Phi_0) \|_{L^2(\mathbb R^d;\mathbb R^{n\times n})} \|v \|_{L^2}   
\\[0.2cm]
& \le &  \| \vartheta\|_{L^2_y} \| D^2f(\Phi_0) \Phi_0' \|_{L^2(\mathbb R;\mathbb R^{n\times n})} \|v \|_{L^2},
\end{array}
\end{equation}
which provides the $L^1$-bound in \eqref{eq:nl:bnd:j:ii:l1:hk}.
\end{proof}

\begin{lemma}
Let $k > d/2$ and assume that \textnormal{(Hf)} and \textnormal{(HTw)} hold. For all $N>0$ there exists a constant $K^N > 0$
so that for any  $\vartheta \in H^{k+1}_y$ with $\|\vartheta\|_{H^k_y} \le N$
we have the bounds 
\begin{equation}
\label{eq:nl:bnd:j:iii:l1:hk}
\begin{array}{lcl}
\| \mathcal{J}_{III}( \vartheta) \|_{L^1}
& \le & K^N \| \nabla_y \vartheta \|_{L^2_y}^2, \\[0.2cm]
\| \mathcal{J}_{III}( \vartheta) \|_{H^k} 
&\le& K^N \| \nabla_y \vartheta \|_{H^k_y}^2,
\end{array}
\end{equation}
while for any  $\vartheta_A, \vartheta_B \in H^k_y$ with 
$\max\{ \|\vartheta_A\|_{H^k_y}, \|\vartheta_B\|_{H^k_y} \} \le N $  
we have
\begin{equation}
\label{eq:nl:j:iii:lip}
\begin{aligned} 
    \| \mathcal{J}_{III}( \vartheta_A) -  \mathcal{J}_{III}( \vartheta_B) \|_{H^k}  
    & \le  K^N  \| \nabla_y \vartheta_B \|_{H^k_y}^2 \|\vartheta_A - \vartheta_B\|_{H^k}
    \\
    &  \qquad
    +K^N  ( \|\nabla_y \vartheta_A \|_{H^k_y} + \| \nabla_y \vartheta_B\|_{H^k} )\|\nabla_y \vartheta_A - \nabla_y \vartheta_B\|_{H^k_y} .
\end{aligned}
\end{equation}
\end{lemma}
\begin{proof}
Upon writing
\begin{equation}
\begin{array}{lcl}
    \Delta_{AB} \mathcal{J}_{III}
    & = &
      \mathcal{J}_{III}( \vartheta_A) -  \mathcal{J}_{III}( \vartheta_B)
    \\[0.2cm]
    & = & (\nabla_y  \vartheta_A + \nabla_y \vartheta_B)\cdot (\nabla_y \vartheta_A - \nabla_y \vartheta_B)\cdot  T_{\vartheta_A} \Phi_0'
    + | \nabla_y \vartheta_B |^2 (T_{\vartheta_A} \Phi_0'- T_{\vartheta_B} \Phi_0'),
\end{array}
\end{equation}
we may use  \eqref{eq:nl:bnd:prod:hky:hkx:lip} and \eqref{eq:nl:bnd:prod:hky:hkx}
to compute
\begin{equation}
\begin{array}{lcl}
    \| \Delta_{AB} \mathcal{J}_{III} \|_{H^k} 
    & \le & 
K^N \|\nabla_y  \vartheta_A + \nabla_y \vartheta_B\|_{H^k_y} \| \nabla_y \vartheta_A - \nabla_y \vartheta_B \|_{H^k_y}\|\Phi_0'\|_{H^k_x}
\\[0.2cm]
& & \qquad
    + \,K^N \| \nabla_y \vartheta_B \|^2_{H^k_y} 
    \| \vartheta_A - \vartheta_B\|_{H^k_y} \| \Phi_0''\|_{H^k_x},
\end{array}
\end{equation}
for some $K^N>0,$ which establishes \eqref{eq:nl:j:iii:lip}  after updating $K^N>0$ when necessary.
The $H^k$-bound in \eqref{eq:nl:bnd:j:iii:l1:hk} follows directly
from \eqref{eq:nl:j:iii:lip} by taking $\vartheta_B = 0$,
while the identity
\begin{equation}
    \| \mathcal{J}_{III}(\vartheta) \|_{L^1} = 
     \|\Phi_0''\|_{L^1_x} \| \nabla_y \vartheta \|_{L^2_y}^2
\end{equation}
immediately confirms the $L^1$-bound in \eqref{eq:nl:bnd:j:iii:l1:hk}.
\end{proof}

We are now ready to consider the final
ingredient $\mathcal{J}_{\rm tr}$. Before diving into the relevant estimates of this function, let us first   point out that its definition  does not depend on the choice of the basis. This statement can either be concluded from the evolution system derivation in {\S}\ref{sec:evol:pert} or from a direct computation which we  demonstrate below. Lemma \ref{lem:nl:bnds:j:tr} tells us when $\mathcal J_{\rm tr}$ is well-defined in $H^k$.
\begin{lemma}\label{lem:nl:independence}
    Suppose that \textnormal{(Hf)}, \textnormal{(HTw)},  and \textnormal{(HgQ)} hold. The definition of
    \begin{equation}
        \mathcal J_{\rm tr}(v,\theta,\gamma)=- \frac{1}{2}T_\vartheta \Phi_0'' \Upsilon(\vartheta)^2\sum_{j=0}^\infty 
\langle g(v + T_{\vartheta} \Phi_0)T_{-\gamma}\sqrt Qe_j ,  \psi_{\rm tw} \rangle_{L^2_x}^2,
    \end{equation}
    when well-defined, is  independent of the orthonormal basis $(e_j)_{j\geq 0}$ of $\mathcal W$.
\end{lemma}
\begin{proof}
   Fix an orthonormal basis $(e_j)_{j\geq 0}$ of $\mathcal W$. For   any other  basis $(f_j)_{j\geq 0}$,  there exists a unitary transformation $U\in\mathscr L(\mathcal W)$ such that $f_j=Ue_j.$ In other words, we have
    \begin{equation}
        Uw=\sum_{i=0}^\infty\sum_{j=0}^\infty U_{ij}\langle w,e_i\rangle_{\mathcal W}f_j,\quad  U_{ij}=\langle e_i,f_j\rangle_{\mathcal W},\quad f_j=\sum_{i=0}^\infty U_{ij}e_i, \quad w\in\mathcal W.
    \end{equation}
    Introduce the shorthand $Z=g(v+T_\theta\Phi_0)T_{-\gamma}\sqrt{Q}.$ This yields
    \begin{equation}
        \begin{array}{lcl}
            \sum_{j=0}^\infty 
\langle Zf_j ,  \psi_{\rm tw} \rangle_{L^2_x}^2 & =
             & \sum_{j=0}^\infty \sum_{i=0}^\infty \sum_{i'=0}^\infty U_{ij}U_{i'j} 
\langle Ze_i ,  \psi_{\rm tw} \rangle_{L^2_x}\langle Ze_{i'} ,  \psi_{\rm tw} \rangle_{L^2_x}\\[0.2cm]
&=&  \sum_{i=0}^\infty \sum_{i'=0}^\infty  
\langle Ze_i ,  \psi_{\rm tw} \rangle_{L^2_x}\langle Ze_{i'} ,  \psi_{\rm tw} \rangle_{L^2_x}\sum_{j=0}^\infty U_{ij}U_{i'j}\\[0.2cm]
&=&  \sum_{i=0}^\infty  
\langle Ze_i ,  \psi_{\rm tw} \rangle_{L^2_x}^2,\\
        \end{array}
    \end{equation}
    utilising the fact that $\sum_{j=0}^\infty U_{ij}U_{i'j}=\delta_{ii'}$ holds. From this computation  the assertion follows.
\end{proof}

Recall that \eqref{eq:nl:def:j:jtr}
can be recast in the form
\begin{equation}
    \mathcal{J}_{\rm tr}(v, \vartheta, \gamma)
    = -\frac{1}{2}\sum_{j=0}^\infty 
    T_{\vartheta} \Phi_0'' \big[ M(v, \vartheta)[T_{-\gamma} \sqrt{Q} e_j] \big]^2.
\end{equation}
In order to facilitate our Lipschitz computations,
we introduce the expressions
\begin{equation}
\begin{array}{lcl}
    \Delta_{AB;I} \mathcal J_{\rm tr} &=& \frac{1}{2} 
    \sum_{j=0}^\infty 
    \big[ 
    T_{\vartheta_A} \Phi_0''  - T_{\vartheta_B} \Phi_0'' 
    \big] \big[ M(v_A, \vartheta_A) [T_{-\gamma} \sqrt{Q} e_j] \big]^2 ,
\\[0.2cm]
    \Delta_{AB;II} \mathcal J_{\rm tr} &=& \frac{1}{2} 
    \sum_{j=0}^\infty 
    \big[ 
     T_{\vartheta_B} \Phi_0'' 
    \big] \big[ M(v_A, \vartheta_A) [T_{-\gamma} \sqrt{Q} e_j] 
     + M(v_B, \vartheta_B) [T_{-\gamma} \sqrt{Q} e_j]
    \big]
\\[0.2cm]
& & \qquad \qquad \times \big[
       M(v_A, \vartheta_A) [T_{-\gamma} \sqrt{Q} e_j] 
     - M(v_B, \vartheta_B) [T_{-\gamma} \sqrt{Q} e_j]
    \big] ,
\end{array}
\end{equation}
which allow us to write
\begin{equation}
    \mathcal J_{\rm tr} (v_A,\vartheta_A,\gamma)
    - \mathcal J_{\rm tr} (v_B,\vartheta_B,\gamma)
    = 
    \Delta_{AB;I} \mathcal{J}_{\rm tr} + \Delta_{AB;II} \mathcal{J}_{\rm tr} .
\end{equation}

\begin{lemma}
\label{lem:nl:bnds:j:tr}
Let $k > d/2$
and suppose that \textnormal{(Hf)}, \textnormal{(HTw)},  and \textnormal{(HgQ)} hold.
For all $N>0$ there exists a constant   $K^N > 0$
so that for all $\gamma \in \mathbb R$, and every $v \in H^k$ and  $\vartheta \in H^k_y$ with $\|v\|_{H^k} \le N$ and $\|\vartheta\|_{H^k_y} \le N$,
we have the bounds 
\begin{equation}
\label{eq:nl:bnd:j:tr:l1:hk}
\begin{array}{lcl}
\| \mathcal{J}_{\rm tr}(v, \vartheta,\gamma) \|_{L^1}
& \le & K^N , \\[0.2cm]
\| \mathcal{J}_{\rm tr}(v, \vartheta,\gamma) \|_{H^k} 
&\le& K^N ,
\end{array}
\end{equation}
while for any $v_A, v_B \in H^k$ and $\vartheta_A, \vartheta_B \in H^k_y$ with \begin{equation}
    \max\{ \|v_A\|_{H^k}, \|v_B\|_{H^k} ,\|\vartheta_A\|_{H^k_y}, \|\vartheta_B\|_{H^k_y} \} \le N, 
\end{equation}  
we have
\begin{equation}
\label{eq:nl:j:tr:lip}
    \| \mathcal{J}_{\rm tr}(v_A, \vartheta_A) -  \mathcal{J}_{\rm tr}(v_B, \vartheta_B) \|_{H^k}  
    \le K^N 
    \big[ \| v_A - v_B \|_{H^k} + \|\vartheta_A - \vartheta_B\|_{H^k}
    \big] .
\end{equation}
\end{lemma}
\begin{proof}
A direct computation yields
\begin{equation}
\begin{array}{lcl}
    \|\mathcal{J}_{\rm tr}(v, \vartheta, \gamma)\|_{L^1}
    &\le& \frac{1}{2} \| \Phi_0''\|_{L^1_x} \sum_{j=0}^\infty 
     \| M(v, \vartheta)[T_{-\gamma} \sqrt{Q} e_j] \|_{L^2_y}^2
     \\&\le& \frac{1}{2} \| \Phi_0''\|_{L^1_x}  \| M(v, \vartheta) T_{-\gamma} \|_{HS(\mathcal W_Q;H^k_y)}^2,
     \end{array}
\end{equation}
while \eqref{eq:nl:bnd:prod:hky:hkx} together with the algebra structure of $H^k_y$ 
allows us to compute
\begin{equation}
\begin{array}{lcl}
    \|\mathcal{J}_{\rm tr}(v, \vartheta, \gamma)\|_{H^k}
    & \le & C_1^N \|\Phi_0''\|_{H^k_x} \sum_{j=0}^\infty 
     \| M(v, \vartheta)[T_{-\gamma} \sqrt{Q} e_j] \|_{H^k_y}^2
     \\[0.2cm]
     & = & C_1^N
     \|\Phi_0''\|_{H^k_x} \| M(v, \vartheta) T_{-\gamma} \|_{HS(\mathcal W_Q;H^k_y)}^2,
\end{array}
\end{equation}
for some $C_1^N>0.$
In particular, the bounds in \eqref{eq:nl:bnd:j:tr:l1:hk}
follow directly from the estimate \eqref{eq:nl:bnds:m:cal:m:unif}.

Turning to the Lipschitz estimate, we first
use  \eqref{eq:nl:bnd:prod:hky:hkx:lip}
together with the algebra structure of $H^k_y$ to obtain
\begin{equation}
\begin{array}{lcl}
  \|  \Delta_{AB;I} \mathcal{J}_{\rm tr} \|_{H^k} & \le &
C_2^N \| \Phi_0''' \|_{H^k_x}
  \sum_{j =0}^\infty
  \|M(v_A, \vartheta_A)[T_{-\gamma} \sqrt{Q} e_j] \|_{H^k_y}^2
\\[0.2cm]
& = & 
C_2^N \| \Phi_0''' \|_{H^k_x}
  \|M(v_A, \vartheta_A)T_{-\gamma} \|_{HS(\mathcal W_Q;H^k_y)}^2,
\end{array}
\end{equation}
for some $C_2^N>0.$
In addition, applying \eqref{eq:nl:bnd:prod:hky:hkx}
together with Cauchy-Schwarz gives us
\begin{align}
  \|  \Delta_{AB;II} \mathcal{J}_{\rm tr} \|_{H^k} &\nonumber\textstyle \le 
  C_3^N \| \Phi_0''\|_{H^k_x}
  \sum_{j=0}^\infty 
  \big[ \|M(v_A, \vartheta_A)T_{-\gamma}\sqrt{Q} e_j \|_{H^k_y} +  \|M(v_B, \vartheta_B)T_{-\gamma}\sqrt{Q} e_j \|_{H^k_y} \big]
\\
&\nonumber\textstyle \qquad \qquad \times
  \| [M(v_A, \vartheta_A) - M(v_B, \vartheta_B)]T_{-\gamma} \sqrt{Q} e_j \|_{HS(\mathcal W_Q;H^k_y)}
\\ 
& \le \nonumber\textstyle
 C_3^N \| \Phi_0''\|_{H^k_x}
  \big[ \sum_{j=0}^\infty 
  \big[ \|M(v_A, \vartheta_A)T_{-\gamma}\sqrt{Q} e_j \|_{H^k_y}^2 +  \|M(v_B, \vartheta_B)T_{-\gamma}\sqrt{Q} e_j \|_{H^k_y}^2 \big] \big]^{1/2}
\\\nonumber
& \nonumber\textstyle \qquad \qquad \times
  \big[ \sum_{j=0}^\infty \| [M(v_A, \vartheta_A) - M(v_B, \vartheta_B)]T_{-\gamma} \sqrt{Q} e_j \|_{HS(\mathcal W_Q;H^k_y)}^2   \big]^{1/2}
\\ \nonumber
& \le \nonumber
  C_3^N \| \Phi_0''\|_{H^k_x} \big[ \|M(v_A, \vartheta_A)T_{-\gamma} \|_{HS(\mathcal W_Q;H^k_y)}
  + \|M(v_A, \vartheta_A) T_{-\gamma}\|_{HS(\mathcal W_Q;H^k_y)}
  ]
\\ 
&   \qquad \qquad \times \| [M(v_A, \vartheta_A) - M(v_B, \vartheta_B)] T_{-\gamma}\|_{HS(\mathcal W_Q;H^k_y)}, 
\end{align}
for some $C_3^N>0.$
The bound \eqref{eq:nl:j:tr:lip} now follows from \eqref{eq:nl:bnds:m:cal:m:unif} and \eqref{eq:nl:bnd:m:calm:lip}.
\end{proof}

Our final task here is to consider the functions $N_\sigma$ and $\mathcal{N}_\sigma$.
For convenience, we write
\begin{equation}
    \mathcal{J}_{\rm tot}(v, \theta,\gamma, \sigma) = \mathcal{J}(v, \vartheta) + \sigma^2 \mathcal{J}_{\rm tr}(v, \theta, \gamma),
\end{equation}
which allows us to recast \eqref{eq:nl:def:n:m}
in the form
\begin{equation}
    N_\sigma(v, \vartheta,\gamma) = \Upsilon(\vartheta) \langle \mathcal{J}_{\rm tot}(v, \theta,\gamma,\sigma) , \psi_{\rm tw} \rangle_{L^2_x}
\end{equation}
and obtain the representation
\begin{equation}
\label{eq:nl:id:lip:n:sigma}
\begin{aligned}
    N_\sigma(v_A, \vartheta_A,\gamma)-N_\sigma(v_B, \vartheta_B,\gamma)
    &=(\Upsilon(\vartheta_A)- \Upsilon( \vartheta_B))  \langle \mathcal{J}_{\rm tot}(v_A, \vartheta_A,\gamma, \sigma), \psi_{\rm tw} \rangle_{L^2_x}
    \\& 
    \,\,\,\,\,+ \Upsilon(\vartheta_B) 
    \langle \mathcal{J}_{\rm tot}(v_A, \vartheta_A,\gamma, \sigma) - 
    \mathcal{J}_{\rm tot}(v_B, \vartheta_B,\gamma, \sigma) , \psi_{\rm tw} \rangle_{L^2_x} .
\end{aligned}
\end{equation}

\begin{lemma}\label{lem:nl:bnds:n:sigma}
 Let $k >d/2$  and suppose that \textnormal{(Hf)}, \textnormal{(HTw)},  and \textnormal{(HgQ)}  hold. For all $N>0$ there exists a constant  $K^N > 0$ so that for all $\gamma \in \mathbb R$,  and every $v \in H^k$ and
  $\vartheta \in H^{k+1}_y$ with $\|v\|_{H^k} \le N $ and $\| \vartheta\|_{H^k_y} \le N$,
 we have the bounds
 \begin{equation}
 \label{eq:nl:bnds:n:cal:n:unif}
 \begin{array}{lcl}
 \left\| N_\sigma(v,\theta,\gamma)\right\|_{H^k _y}+\left\|N_\sigma(v,\theta,\gamma)\right\|_{L^1 _y}
     & \leq & K^N\left(\sigma^2 + \|v\|_{H^k}^2+ \|v\|_{H^k} \| \vartheta\|_{H^k_y}  + \|\nabla_y
     \theta\|_{H^{k} _y}^2\right) ,
\\[0.2cm]
     \left\|\mathcal N_\sigma(v,\theta,\gamma)\right\|_{H^k}
     & \leq & K^N\left(\sigma^2 + \|v\|_{H^k}^2+ \|v\|_{H^k} \| \vartheta\|_{H^k_y}  + \|\nabla_y
     \theta\|_{H^{k} _y}^2\right),
\end{array}
 \end{equation}
 for any $\sigma\in \mathbb R.$
  \end{lemma}
\begin{proof}
We first note that a direct computation yields
\begin{equation}
    \|N_\sigma(v, \vartheta,\gamma)\|_{L^1_y} \le 
    \| \Upsilon(\vartheta)\|_{\infty} 
    \|\mathcal{J}_{\rm tot}(v, \vartheta, \gamma,\sigma)\|_{L^1} \|\psi_{\rm tw}\|_{\infty},
\end{equation}
while \eqref{eq:nl:bnd:l2x:to:hk} together with the algebra structure of $H^k_y$
imply that
\begin{equation}
    \|N_\sigma(v, \vartheta,\gamma)\|_{H^k_y} \le \big( 1 + \|\Upsilon(\vartheta) + 1 \|_{H^k_y} \big) \|\mathcal{J}_{\rm tot}(v, \vartheta, \gamma,\sigma)\|_{H^k}
     \|\psi_{\rm tw}\|_{L^2_x} .
\end{equation}
In addition, bound
\eqref{eq:nl:bnd:prod:hky:hkx}
yields
\begin{equation}
 \begin{array}{lcl}
 \|\mathcal{N}_\sigma(v, \vartheta,\gamma) \|_{H^k}
 \le \| \mathcal{J}_{\rm tot}(v, \vartheta,\gamma,\sigma) \|_{H^k}
  + C^N \|N_\sigma(v, \vartheta,\gamma) \|_{H^k_y} \|\Phi_0'\|_{H^k_x},
 \end{array}
\end{equation}
for some $C^N>0.$
The stated bounds hence follows from 
\eqref{eq:nl:est:k2:hk},
\eqref{eq:nl:bnd:j:i:l1:hk},
\eqref{eq:nl:bnd:j:ii:l1:hk},
\eqref{eq:nl:bnd:j:iii:l1:hk},
and
\eqref{eq:nl:bnd:j:tr:l1:hk}.
\end{proof}

\begin{lemma}
 Let $k >d/2$ and suppose that \textnormal{(Hf)}, \textnormal{(HTw)},  and \textnormal{(HgQ)} hold. For all $N>0$ there exist a constant $K^N > 0$
so that
     for all $\gamma \in \mathbb R$, and every   $v_A, v_B \in H^k$ and $\vartheta_A, \vartheta_B \in H^{k+1}_y$ with 
     \begin{equation}
         \max\{ \| v_A \|_{H^k}, \|v_B\|_{H^k}  , \|\vartheta_A\|_{H^k_y},
          \|\vartheta_B\|_{H^k_y} \} \le N,
     \end{equation}
we have the bound
\begin{align}
\label{eq:nl:n:sigma:lip} 
    \| N_\sigma(v_A, \vartheta_A,\gamma) -  N_\sigma(v_B, \vartheta_B,\gamma) \|_{H^k_y}  
    &\le K^N 
    \big[ \sigma^2 +  \|v_A\|_{H^k} + \|v_B\|_{H^k} + \|\vartheta_B\|_{H^k_y} + \| \nabla_y \vartheta_B\|^2_{H^k_y}]\nonumber
    \\ & \qquad \qquad \times
    \big[ \| v_A - v_B \|_{H^k} + \|\vartheta_A - \vartheta_B\|_{H^k_y}
    \big]
\\ & \qquad + K^N  \big[\|\nabla_y \vartheta_A \|_{H^k_y} + \| \nabla \vartheta_B\|_{H^k_y} \big]\|\nabla_y \vartheta_A - \nabla_y \vartheta_B\|_{H^k_y},\nonumber
\end{align}
together with
\begin{align}
\label{eq:nl:cal:n:lip} 
    \nonumber \| \mathcal{N}_\sigma(v_A, \vartheta_A,\gamma) -  \mathcal{N}_\sigma(v_B, \vartheta_B,\gamma) \|_{H^k}  
    &\le K^N 
    \big[ \sigma^2 +  \|v_A\|_{H^k} + \|v_B\|_{H^k}
     + \|\vartheta_B\|_{H^k_y}
     + \| \nabla_y \vartheta_B\|^2_{H^k_y
    }]
    \\ 
     &  \qquad \qquad \times
    \big[ \| v_A - v_B \|_{H^k} + \|\vartheta_A - \vartheta_B\|_{H^k_y}
    \big]
\\
& \qquad + K^N  \big[ \|\nabla_y \vartheta_A \|_{H^k_y} + \| \nabla_y \vartheta_B\|_{H^k_y} \big]\|\nabla_y \vartheta_A - \nabla_y \vartheta_B\|_{H^k_y},\nonumber 
\end{align}
for any $\sigma\in\mathbb R.$
\end{lemma}
\begin{proof}
Upon writing
\begin{equation}
    \Delta_{AB} N = [N_\sigma(v_A, \vartheta_A,\gamma) - N_\sigma(v_B, \vartheta_B,\gamma)],
\end{equation}
we may use the representation \eqref{eq:nl:id:lip:n:sigma} 
together with the bounds \eqref{eq:nl:bnd:l2x:to:hk} and \eqref{eq:nl:est:k2:hk} to obtain
\begin{equation}
    \begin{array}{lcl}
    \| \Delta_{AB} N \|_{ H^k_y}
    & \le & C_1^N \| \Upsilon(\vartheta_A)- \Upsilon( \vartheta_B) \|_{H^k_y}
\|\mathcal{J}_{\rm tot}(v_A,\vartheta_A, \gamma, \sigma)\|_{H^k}
\\[0.2cm]
& &  \qquad
+ \,C_1^N 
 \| \mathcal{J}_{\rm tot}(v_A, \vartheta_A, \gamma,\sigma)
 - \mathcal{J}_{\rm tot}(v_B, \vartheta_B, \gamma,\sigma)
 \|_{H^k}
    \end{array}
\end{equation}
for some $C_1^N > 0$. The desired 
estimate now follows 
 from the bounds in 
 \eqref{eq:nl:est:k2:hk} and \eqref{eq:nl:est:k2:hk:lip}  together with the estimates
 in Lemmas \ref{lem:nl:bnds:j:i}--\ref{lem:nl:bnds:j:tr}.
 Turning to the difference
\begin{equation}
    \Delta_{AB} \mathcal N =  \mathcal N_\sigma(v_A, \vartheta_A,\gamma) - \mathcal N_\sigma(v_B, \vartheta_B,\gamma),
\end{equation}
we invoke \eqref{eq:nl:bnd:prod:hky:hkx:lip} and \eqref{eq:nl:bnd:prod:hky:hkx}
together with the a priori bounds on $\vartheta_A$ and $\vartheta_B$
to compute
\begin{equation}
\begin{array}{lcl}
   \| \Delta_{AB} \mathcal{N} \|_{ H^k}
& \le & 
\| \mathcal{J}_{\rm tot}(v_A, \vartheta_A,\gamma,\sigma) 
- \mathcal{J}_{\rm tot}(v_B, \vartheta_B,\gamma,\sigma) \|_{ H^k}
\\[0.2cm]
& & \qquad
+\, C_2^N\| N_\sigma(v_A, \vartheta_A,\gamma) - N_\sigma(v_B, \vartheta_B,\gamma) \|_{H^k_y} \| \Phi_0' \|_{H^k_x}
\\[0.2cm]
& & \qquad
 + \,C_2^N \| N_\sigma(v_B, \vartheta_B,\gamma) \|_{ H^k_y}
 \| \vartheta_A - \vartheta_B\|_{H^k_y} \|\Phi_0''\|_{H^k_x},
\end{array}
\end{equation}
for some $C_2^N>0.$
Applying 
\eqref{eq:nl:j:i:lip}, \eqref{eq:nl:j:ii:lip},
\eqref{eq:nl:j:iii:lip}, \eqref{eq:nl:j:tr:lip},
\eqref{eq:nl:bnds:n:cal:n:unif},
and \eqref{eq:nl:n:sigma:lip}
leads to the desired estimate.
\end{proof}

\begin{proof}[Proof of Proposition \ref{prop:nl:bnds}]
These bounds follow from
\eqref{eq:nl:bnds:m:cal:m:unif} and
\eqref{eq:nl:bnds:n:cal:n:unif}.
\end{proof}

\begin{proof}[Proof of Proposition \ref{prop:nl:bnds:lip}]
These bounds follow from
\eqref{eq:nl:bnd:m:calm:lip},
\eqref{eq:nl:n:sigma:lip},
and
\eqref{eq:nl:cal:n:lip}.
\end{proof}

\section{Evolution of the perturbation}\label{sec:evol:pert}
In this section, we first establish  the existence and uniqueness of solutions to \eqref{eq:pert:system} in the variational sense. Subsequently, we  justify that the obtained solution pair $(v,\theta)$ describes the evolution of the wave in \eqref{eq:reaction:diff} in an appropriate way.  In particular, we prove  Propositions \ref{prop:var:higher:short} and \ref{prp:mr:an:sol:org:sys}.



With regard to the first result, we proceed by searching for solutions  $(v, \theta)$  measured with respect to the Gelfand triple $(\mathcal V,\mathcal{H},\mathcal V^*)$ given by
    \begin{equation}
    \label{eq:var:weak:setting}
        \mathcal V^*=H^{k-1}\times H^{k-1}_y,\quad \mathcal{H}=H^k\times H^k_y,\quad \mathcal V= H^{k+1}\times H^{k+1}_y. 
    \end{equation}
    The associated inner products are 
    \begin{equation}
    \begin{aligned}\langle (v,\theta),(w,\eta)\rangle_{\mathcal V}&=\langle v,w\rangle_{H^{k+1}}+\langle \theta,\eta\rangle_{H^{k+1}_y},\\\langle (v,\theta),(w,\eta)\rangle_{\mathcal{H}}&=\langle v,w\rangle_{H^k}+\langle \theta,\eta\rangle_{H^k_y},\end{aligned}
    \end{equation}
    while the duality pairing 
\begin{equation}\langle (v,\theta),(w,\eta)\rangle_{\mathcal V^*;\mathcal V}=\langle v,w\rangle_{H^{k-1};H^{k+1}}+\langle \theta,\eta\rangle_{H^{k-1}_y;H^{k+1}_y},
    \end{equation} 
acts on the components as
\begin{equation}
\label{eq:var:dual:pairing:h:km1:kp1}
    \langle v, w \rangle_{H^{k-1};
    H^{k+1}}
    =  \langle v, w \rangle_{H^{k-1}}
     - \sum_{|\alpha| = k-1 } \langle \partial^\alpha v , \partial^\alpha \Delta w \rangle_{L^2},
\end{equation}
with the analogous definition for $\langle \theta,\eta\rangle_{H^{k-1}_y;H^{k+1}_y}$.

Indeed, one can follow 
\cite{agresti2024criticalNEW,bosch2024multidimensional,brezis2011functional,krylov1981stochastic} to show
that for $k \ge 1$ this bilinear map 
allows $H^{k-1}$
to be interpreted as the dual of $H^{k+1}$. 
Furthermore, the diffusion operator $\Delta$ can|as usual|be seen as an element of $\mathscr L(H^{k+1};H^{k-1})$, while definition \eqref{eq:var:dual:pairing:h:km1:kp1}
allows us to write
\begin{equation}
\label{eq:var:delta:v:w:hkm1:hkp1}
    \langle \Delta v,w\rangle_{H^{k-1};H^{k+1}} = -\langle \nabla v, \nabla w \rangle_{H^k}=-\langle v,w\rangle_{H^{k+1}}+\langle v,w\rangle_{L^2},
\end{equation}
for any pair $v,w\in H^{k+1}$.
This can be seen as an alternative yet equivalent definition for $\Delta$ from  $H^{k+1}$ into its dual.  Similar remarks hold for $H^{k+1}_y$  with   $\Delta_y$. Finally, to make sense of \eqref{eq:Hk*}, we emphasise the relation
\begin{equation}
    \langle \mathcal Lv,w\rangle_{H^{k-1};H^{k+1}}=-\langle \nabla v,\nabla w\rangle_{H^k}+c_0\langle \partial_x v,w\rangle_{H^k}+\langle Df(\Phi_0)v,w\rangle_{H^k},
\end{equation}
for $v,w\in H^{k+1}.$ 

In order to acquire a (local) existence and uniqueness result for 
system 
\eqref{eq:pert:system}, we invoke the critical variational framework in \cite{agresti2024criticalNEW}. An important feature of this framework is that it only requires local Lipschitz estimates for the deterministic nonlinearity in the $\mathcal V^*$-norm rather than the $\mathcal H$-norm \cite[Assumption 3.1 part (3)]{agresti2024criticalNEW}. This enables us to cope with the   nonlinearities $N_\sigma$ and $\mathcal N_\sigma$ in \eqref{eq:nl:def:n:m} and \eqref{eq:nl:def:cal:n:m}, respectively;  see Proposition \ref{prop:nl:bnds:lip}. 

\begin{proposition}
\label{prop:var:higher:full}
Let $k>d/2+1$
and suppose  that   \textnormal{(Hf)}, \textnormal{(HTw)}, and \textnormal{(HgQ)} hold.
Further, fix $T>0$ and $0\leq \sigma\leq 1$. Then for any initial condition 
\begin{equation}
(v_0,\theta_0)\in H^k\times H^k_y,
\end{equation}
there exists an increasing sequence of stopping times $(\tau_\ell)_{\ell\geq 0}$ and a stopping time $\tau_{\infty}$, with  $\tau_\ell\to \tau_{\infty}$ and $0<\tau_{\infty}\leq T$ $\mathbb P$-a.s., together with progressively measurable maps
\begin{equation}
    v:[0,T]\times \Omega\to H^k,\quad \theta:[0,T]\times \Omega\to H^k_y,
\end{equation}
that satisfy the following properties:
\begin{enumerate}[\rm (i)]
    \item For almost every $\omega\in\Omega$, the map
    $ 
        t\mapsto (v(t,\omega),\theta(t,\omega))
    $
    is of class $C([0,\tau_{\infty}(\omega));H^{k}\times H^k_y);$ 
   \item We have the integrability condition $(v,\theta)\in L^2(\Omega;L^2([0,\tau_\ell];H^{k+1}\times H^{k+1}_y))$,  for any $\ell\geq 0;$
   \item The $H^k$-valued and $H^k_y$-valued identities\footnote{At first, the identities in \eqref{eq:Hk*} and \eqref{eq:Hk*:theta} should be understood as an equality in $[H^{k+1}]^*\equiv  H^{k-1}$ and $[H^{k+1}_y]^*\equiv  H^{k-1}_y$, respectively, but by (ii) we can conclude that in fact we have equality in $H^k$ and $H^{k}_y$, respectively.} 
   \begin{equation}
       \begin{aligned}
           v(t)=v(0)+\int_0^t [\mathcal L v(s)+\mathcal N_\sigma(v(s),\theta(s),\gamma(s))]\mathrm ds +\sigma \int_0^t\mathcal M(v(s),\theta(s))T_{-\gamma(s)}\mathrm 
 dW_s^Q,
       \end{aligned}\label{eq:Hk*}
   \end{equation}
and
   \begin{equation}\label{eq:Hk*:theta}
\theta(t)=\theta(0)+\int_0^t[\Delta_y\theta(s)+N_\sigma(v(s),\theta(s),\gamma(s))]\mathrm ds+\sigma\int_0^t M(v(s),\theta(s))T_{-\gamma(s)}\mathrm dW_s^Q,
   \end{equation} 
  hold $\mathbb P$-a.s. for all $0\leq t<\tau_{\infty};$
   \item Suppose there are  other progressively measurable maps $\tilde v$ and $\tilde \theta$ that satisfy \textnormal{(i)--(iii)} with another stopping time $\tilde \tau_{\infty}$ and localising sequence $(\tilde \tau_\ell)_{\ell \ge 0}$. Then for almost every $\omega\in\Omega,$ we have  $\tilde \tau_{\infty}(\omega)\leq \tau_{\infty}(\omega)$ together with 
   \begin{equation}
      \quad \tilde v(t,\omega)=v(t,\omega)\quad\text{and}\quad \tilde \theta(t,\omega)=\theta(t,\omega),\quad \text{for all}\quad 0\leq t< \tilde \tau_{\infty}(\omega);
   \end{equation}
  \item Writing 
   \begin{equation}
  \mathcal Z_\ell=\sup_{0\leq t<\tau_\ell}\big[\|v(t)\|_{H^k}^2+\|\theta(t)\|_{H^k_y}^2\big]+\int_0^{\tau_\ell}\big[\|v(t)\|_{H^{k+1}}^2 +\|\theta(t)\|_{H^{k+1}_y}^2\big]\,\mathrm dt,
   \end{equation}
   for all $\ell\geq 0,$
   we have $\mathcal Z_{\ell}^2\leq \ell$ together with
   \begin{equation}
     \textstyle   \mathbb P\left(\tau_{\infty}<T,\quad  \sup_{\ell\geq 0}\mathcal Z^2_\ell <\infty\right)=0.
   \end{equation}
\end{enumerate}
\end{proposition}
\begin{proof}
We  can appeal to
\cite[Thm. 3.3]{agresti2024criticalNEW} with the Gelfand triple 
 \eqref{eq:var:weak:setting}.
This yields the existence and uniqueness of a maximal solution $((v,\theta),\sigma_\infty)$ with a corresponding localising sequence $(\sigma_\ell)_{\ell\geq 0}$.
Let us define the stopping times $
    \tau_\ell=\sigma_\ell \wedge \sigma_\ell,$
    where
    \begin{equation}
       \textstyle  \sigma_\ell^*=\inf\{t\geq 0:\sup_{0\leq s<t}\big[\|v(s)\|_{H^k}^2+\|\theta(s)\|_{H^k_y}^2\big]+\int_0^{t}\big[\|v(s)\|_{H^{k+1}}^2 +\|\theta(s)\|_{H^{k+1}_y}^2\big]\,\mathrm ds\}.
    \end{equation}
    The blow-up criterion of $\sigma_\infty$ in \cite[Thm. 3.3]{agresti2024criticalNEW} tells us that $\sigma_\ell^* \to \sigma_\infty $ holds. We set $\tau_{\infty}=\sigma_\infty$, 
    from which (i)--(v) follows.
\end{proof}
\begin{remark}\label{remark:statistically}
    Several examples in Appendix \ref{sec:about:noise} consider the covariance operator $Q$ to be translation invariant  in the $x$-direction. The terms $T_{-\gamma(s)}\mathrm dW_s^Q$ and $\mathrm dW_s^Q$ are then statistically speaking interchangeable.  Indeed, by defining the  stochastic process 
\begin{equation}
    \widetilde W_t=\sum_{j=0}^\infty\int_0^tT_{-\gamma(s)}\sqrt{Q}{e_j}\mathrm d\beta_j(s),
\end{equation}
as done in \cite{bosch2024multidimensional,hamster2020}, we observe that $T_{-\gamma(s)}\mathrm dW_s^Q$ can be replaced by $\mathrm d\widetilde W_s^Q$ because $T_{-\gamma(s)}$ are isometries and that   $\widetilde W_t^Q$
is indistinguishable from  $W_t^Q$.
\end{remark}


We shall now demonstrate how to relate the pair $(v,\theta)$ in Proposition  \ref{prop:var:higher:full} to a solution $u$ of the original system \eqref{eq:reaction:diff}. We achieve this by showing that the stochastic process
\begin{equation}
z(t)=T_{\gamma(t)}v(t)+T_{\gamma(t)+\theta(t)}\Phi_0-\Phi_{\rm ref},\label{eq:z(t)}
\end{equation}
with $\gamma(t)=c_0t$ and $\Phi_{\rm ref}$ as in {\S}\ref{sec:assumptions}, solves the equation
\begin{equation}
    \mathrm dz=[\Delta z+\Phi_{\rm ref}''+f(\Phi_{\rm ref}+z)]\,\mathrm dt+g(\Phi_{\rm ref}+z)\,\mathrm dW_t^Q
\end{equation}
in the analytically weak sense. This  implies that $u(t)=z(t)+\Phi_{\rm ref}$ solves \eqref{eq:reaction:diff} in the analytically weak sense. We achieve this by means of the
 rigorous It\^o computation below.

\begin{proposition}\label{prop:system:z}
   Consider the setting  of Proposition \ref{prop:var:higher:full}.   For any $\zeta\in C_c^\infty(\mathbb R^d;\mathbb R^n)$, the identity
\begin{equation}\begin{aligned}
\langle z(t),\zeta\rangle_{H^k}=\langle z(0),&\,\zeta\rangle_{H^k}+\int_0^t  \big[\langle z(s),\Delta \zeta\rangle_{H^k} +\langle \Phi_{\rm ref}''+f(\Phi_{\rm ref}+z(s)),\zeta\rangle_{H^k}\big]\mathrm ds\\&\qquad\qquad\qquad+\sigma\int_0^t \mathcal \langle g(\Phi_{\rm ref}+z(s))\,\mathrm d  W_s^Q,\zeta\rangle_{H^k}\label{eq:weak:original}
\end{aligned}
\end{equation}
holds $\mathbb P$-a.s. for all $0\leq t< \tau_{\infty}. $
\end{proposition}
\begin{proof}
    Define the  functionals
    \begin{equation}
    \psi_{1;\zeta}:H^{k-1}\times \mathbb R\to \mathbb R,\quad  \psi_{2;\zeta}:\mathbb R\to \mathbb R,\quad  \psi_{3;\zeta}:H^{k-1}_y\times \mathbb R\to \mathbb R,
\end{equation}
by
\begin{equation}
\begin{aligned}
    \psi_{1;\zeta}(v,\gamma)&=\langle v,T_{-\gamma}\zeta\rangle_{H^{k-1};H^{k+1}},\\
    \psi_{2;\zeta}(\gamma)&=\langle T_\gamma\Phi_0-\Phi_{\rm ref},\zeta\rangle_{H^k},\\
   \psi_{3;\zeta}(\theta,\gamma)&=\langle T_\theta\Phi_0-\Phi_0,T_{-\gamma}\zeta\rangle_{H^{k-1};H^{k+1}}.
\end{aligned}
\end{equation}
Upon introducing the functional $\psi_\zeta:H^{k-1}\times H^{k-1}_y\times \mathbb R\to \mathbb R$, given by
\begin{equation}
\psi_\zeta(v,\theta,\gamma)=\psi_{1,\zeta}(v,\gamma)+\psi_{2,\zeta}(\gamma)+\psi_{3,\zeta}(\theta,\gamma),
\end{equation}
 and moving all $\gamma$-translations to the left of the inner product, we 
 obtain the identity
\begin{equation}
    \langle z(t),\zeta\rangle_{H^k}=\psi_{\zeta}\big(v(t),\theta(t),\gamma(t)\big).
\end{equation}
We note that $\phi_{1;\zeta}, \phi_{2;\zeta}$ and $\phi_{3;\zeta}$ are $C^2$-smooth, with first-order Fr\'echet derivatives given by
\begin{equation}
\begin{aligned}
    D\psi_{1;\zeta}(v,\gamma)[V,\Gamma]&=\langle V,T_{-\gamma}\zeta\rangle_{H^{k-1};H^{k+1}}+\Gamma\langle v,T_{-\gamma}\zeta'\rangle_{H^{k-1};H^{k+1}},\\
    D\psi_{2;\zeta}(\gamma)[\Gamma]&=\Gamma\langle \Phi_{0},T_{-\gamma}\zeta'\rangle_{H^k},\\
D\psi_{3;\zeta}(\theta,\gamma)[\Theta,\Gamma]&=-\langle \Theta T_\theta\Phi_{0}',T_{-\gamma}\zeta\rangle_{H^{k-1};H^{k+1}}+\Gamma\langle T_\theta\Phi_0-\Phi_0,T_{-\gamma}\zeta'\rangle_{H^{k-1};H^{k+1}},
\end{aligned}
\end{equation}
together with the second-order expressions
\begin{equation}
\begin{aligned}
    D^2\psi_{1;\zeta}(v,\gamma)[V,\Gamma][V,\Gamma]&=2\Gamma\langle V,T_{-\gamma}\zeta'\rangle_{H^{k-1};H^{k+1}}+\Gamma^2\langle v,T_{-\gamma}\zeta''\rangle_{H^{k-1};H^{k+1}},\\
D^2\psi_{2;\zeta}(\gamma)[\Gamma][\Gamma]&=\Gamma^2\langle \Phi_{0},T_{-\gamma}\zeta''\rangle_{H^k},\\
    D^2\psi_{3;\zeta}(\theta,\gamma)[\Theta,\Gamma][\Theta,\Gamma]&=\langle \Theta^2 T_\theta\Phi_{0}'',T_{-\gamma}\zeta\rangle_{H^{k-1};H^{k+1}}+2\Gamma\langle \Theta T_\theta\Phi_0',T_{-\gamma}\zeta'\rangle_{H^{k-1};H^{k+1}}\\&\qquad +\Gamma^2\langle T_\theta\Phi_0.T_{-\gamma}\zeta''\rangle_{H^{k-1};H^{k+1}}.
\end{aligned}
\end{equation}
In a similar fashion as \cite[Lem. 5.3]{hamster2020}, we 
may use It\^o's formula to obtain
\begin{equation}
    \begin{aligned}
        \langle z(t),\zeta\rangle_{H^k}&=\langle z(0),\zeta\rangle_{H^k}+\int_0^t \big[\mathcal D_1(v,\theta,\gamma) +\tfrac{1}2\sigma^2 \mathcal T(v,\theta,\gamma)\big]\mathrm ds+\sigma \int_0^t \mathcal D_2(v,\theta,\gamma)T_{-\gamma} \mathrm dW_s^Q,
    \end{aligned}
\end{equation}
in which the deterministic contribution  given by
\begin{equation}
\begin{aligned}
    \mathcal D_1(v,\theta,\gamma)&=D\psi_{\zeta}(v,\theta,\gamma)[\mathcal Lv+\mathcal N_\sigma(v,\theta),\Delta_y\theta+N_\sigma(v,\theta),c_0]\\
    & =\langle \mathcal Lv+\mathcal N_\sigma(v,\theta)-[\Delta_y\theta+N_\sigma(v,\theta)]T_\theta\Phi_0',T_{-\gamma}\zeta\rangle_{H^{k-1};H^{k+1}}\\
&\qquad+c_0\langle v+T_\theta \Phi_0,T_{-\gamma }\zeta'\rangle_{H^{k-1};H^{k+1}}\\
    & =\langle \Delta(v+T_\theta\Phi_0),T_{-\gamma} \zeta\rangle _{H^{k-1};H^{k+1}}+\langle f(v+T_\theta\Phi_0),T_{-\gamma}\zeta\rangle_{H^k}\\
    &\qquad+ \sigma^2\langle \mathcal J_{\rm tr}(v,\theta,\gamma)T_\theta\Phi_0'',T_{-\gamma}\zeta\rangle_{H^k},
    \end{aligned}
\end{equation}
on account of the identities \eqref{eq:Deltay:Ttheta} and $T_\theta\Phi_0''+c_0T_\theta\Phi_0'+f(T_\theta\Phi_0)=0$, and with $\mathcal J_{\rm tr}$ defined in \eqref{eq:nl:def:j:jtr}.
The stochastic contribution reads
\begin{equation}
\begin{aligned}
    \mathcal D_2(v,\theta,\gamma)[\xi]&=D\psi_{\zeta}(v,\theta,\gamma)[\mathcal M(v,\theta)\xi,M(v,\theta)\xi,0]\\
    & =\langle \mathcal   M(v,\theta)[\xi]-M(v,\theta)[\xi] T_\theta\Phi_0',T_{-\gamma}\zeta\rangle_{H^{k-1};H^{k+1}}\\
    & =\langle g(v+T_\theta\Phi_0)\xi,T_{-\gamma}\zeta\rangle_{H^k},
    \end{aligned}
\end{equation}
and   the trace component is given by 
\begin{equation}
\label{eq:ev:trace}
\begin{aligned}
    \mathcal T(v,\theta,\gamma)&=\textstyle  \sum_{j=0}^\infty D^2\psi_\zeta(v,\theta,\gamma)[Z_j,Z_j]\\
&=\textstyle \sum_{j=0}^\infty\langle M(v,\theta)[T_{-\gamma}\sqrt{Q}e_j]^2T_\theta\Phi_0'',T_{-\gamma}\zeta\rangle _{H^k}\\
&=-2\langle \mathcal J_{\rm tr}(v,\theta,\gamma),T_{-\gamma}\zeta\rangle _{H^k},
    \end{aligned}
\end{equation}
where we introduced the shorthand notation $Z_j=(\mathcal M(v,\theta)T_{-\gamma}\sqrt Qe_j,M(v,\theta)T_{-\gamma}\sqrt Qe_j,0).$ Moving all $\gamma$-translations to the left of the inner products proves \eqref{eq:weak:original}.
\end{proof}
\begin{remark}
    The Fr\'echet derivatives above allow actually for even more general phase tracking mechanisms, e.g., when one would like to incorporate both a  global phase and a local phase.
\end{remark}

\begin{proof}[Proof of Proposition \ref{prop:var:higher:short}]
This follows directly from Proposition \ref{prop:var:higher:full}.
\end{proof}
\begin{proof}[Proof of Proposition \ref{prp:mr:an:sol:org:sys}]
  This follows directly from Proposition \ref{prop:system:z}.
\end{proof}

\section{Supremum bounds for stochastic integrals}
\label{sec:supremum}

In this section we obtain several growth bounds for the stochastic solutions that we will encounter during our stability analysis in {\S}\ref{sec:stability}. The main new feature compared to previous works is that the semigroup $S_H(t)$|
generated by $\Delta_y$|features algebraic decay and that we use algebraic weights in our maximal regularity bounds. 

\paragraph{Convolutions with \texorpdfstring{$S_H(t)$}{SH(t)}}
In this section, we  obtain growth bounds for  the  supremum over $[0,T]$ of the stochastic convolution
\begin{equation}
\mathcal E_B(t)=\int_0^t S_H(t-s)  B(s)\mathrm dW_s, \label{eq:EB(t)}
\end{equation}
where $S_H(t)$ denotes the heat semigroup on $\mathbb R^{d-1}$. 
We further determine that the running supremum  of the integral\footnote{We  claim that   $\int_0^t(1+t-s)^{-\mu}\| \mathcal E_B(s)\|_{H^{k+1}_y}^2\mathrm ds$ grows too fast in dimensions 2 and 3 to obtain \eqref{eq:E_B}.}
\begin{equation}\mathcal J_B(t)=\int_0^t(1+t-s)^{-\mu}\|\nabla_y \mathcal E_B(s)\|_{H^{k}_y}^2\mathrm ds, \label{eq:JB}
\end{equation}
grows with the same rate as  \eqref{eq:EB(t)} for any $0<\mu\leq (d-1)/2$. Observe that $(d-1)/2$ is the algebraic decay rate of the heat semigroup squared. The supremum
bounds mentioned above are  obtained by 
imposing the following a priori pathwise estimates on $B$, in line with the other works \cite{bosch2025conditionalspeedshapecorrections,bosch2024multidimensional,hamster2020expstability}. Let us point out that    \eqref{eq:theta:star} below implies  $B\in \mathcal{N}^p\big([0, T] ; \mathbb F  ; H S(\mathcal W_Q;H^k_y  ) \big)$ for all $p\geq 2.$



\begin{itemize}
    \item[(HB)] For the process $B \in \mathcal{N}^2\big([0, T] ; \mathbb F  ; H S(\mathcal W_Q;H^k_y  ) \big)$ 
there exists a constant $\Theta_*>0$ 
such that 
the pathwise bound 
\begin{equation}
\int_0^t (1+t-s)^{-\mu}\big[\|B(s)\|_{H S(\mathcal W_Q;H^k_y  )}^2+\|B(s)\|_{\Pi_2(\mathcal W_Q;L^1_y)}^2\big] \mathrm d s \leq \Theta_*^2,
\label{eq:theta:star}
\end{equation}
holds for all $0 \leq t \leq T$.
\end{itemize}




\begin{proposition}[{see {\S}\ref{sec:estimates:max:reg}}] \label{prop:E_B}
Assume that  $T \ge 2$ is an integer  and let $\mu\in (0,\frac{d-1}{2}]$ for $d\geq 2$.
There exists a constant $K_{\rm gr} > 0$ that does not depend on $T$
so that for any process $B$ that satisfies \textnormal{(HB)}, and every integer $p \ge 1$,
we have the growth estimate
\begin{equation}
\mathbb E \sup _{0 \leq t \leq T}\left\|\mathcal{E}_B(t)\right\|_{H^k_y }^{2p}
+ \mathbb E\sup_{0\leq t\leq T}\mathcal J_B(t)^{p}
\leq K^{2p}_{\rm gr} \Theta_*^{2p} (p^p+\log (T)^p).
\label{eq:E_B}
\end{equation}
\end{proposition}
We  primarily follow the approach of 
\cite[Sec. 3]{bosch2024multidimensional}  to prove Proposition \ref{prop:E_B} and implement several improvements that were introduced in \cite{bosch2025conditionalspeedshapecorrections}. The algebraic decay requires subtle modifications to the latter procedure at several points, which we highlight below. 

It turns out that in {\S}\ref{sec:stability} we will need to satisfy the  
stricter  condition
\begin{equation}
    \mu=\tfrac{d-1}{4},\quad 2\leq d\leq 5,\quad\text{and}\quad  1<\mu\leq \tfrac{d-1}{4},\quad d> 5,\label{eq:mu:regime}
\end{equation}
  in order to be able to deal with the  dynamical flow.
When $\|B(s)\|_{HS(\mathcal W_Q;H^k_y)}$ and $\|B(s)\|_{\Pi_2(\mathcal W_Q;L^1_y)}$ are uniformly bounded in $[0,T] $ by some constant that does not depend on $T$, which is the case in {\S}\ref{sec:stability}, then  condition (HB) is  satisfied with a $T$-independent $\Theta_*$  only if  $\mu>1$. Recall that the terms $\mathfrak d_{\rm stc}(T)$ and $\mathfrak d_{\rm det}(T)$ are defined in \eqref{eq:int:def:d:stc}
and \eqref{eq:int:def:d:det}
and can also be found in Table \ref{table:main:hfdst2}.

\begin{corollary}\label{cor:E_B}
    Assume that $T \ge 2$ is an integer and consider $\Lambda_*>0$. 
There exists a constant $K_{\rm gr} > 0$ that does not depend on $T$
so that for any process $B$ that satisfies
\begin{equation}
     \|B(t)\|_{HS(\mathcal W_Q;H^k_y)}^2\leq \Lambda_*^2,\qquad \|B(t)\|_{\Pi_2(\mathcal W_Q;L^1_y)}^2\leq \Lambda_*^2,\quad 0\leq t\leq T,
\end{equation}and every integer $p \ge 1$,
we have the growth bound
\begin{equation}
\mathbb E \sup _{0 \leq t \leq T}\left\|\mathcal{E}_B(t)\right\|_{H^k_y }^{2p}
+ \mathbb E\sup_{0\leq t\leq T}\mathcal J_B(t)^{p}
\leq K^{2p}_{\rm gr} \Lambda_*^{2p}\mathfrak d_{\rm stc}(T)^p(p^p+\log(T)^p)\label{eq:RHS57}
\end{equation}
for $\mu=\frac{d-1}2$, while for $\mu$ in the parameter regime \eqref{eq:mu:regime} we have
\begin{equation}
\label{eq:RHS5}\mathbb E\sup_{0\leq t\leq T}\mathcal J_B(t)^{p}
\leq K^{2p}_{\rm gr} \Lambda_*^{2p}\mathfrak d_{\rm det}(T)^p(p^p+\log(T)^p).
\end{equation}
\end{corollary}

     
    Note that  $\mathcal E_B(t)$ in the setting of Corollary $\ref{cor:E_B}$ is independent of $\mu.$  It is also worthwhile to point out that in the absence of the supremum, we find that 
    $\mathbb E[\mathcal J_B(t)]$ is bounded by
    \begin{equation}\mathfrak d_{\rm stc}(T)^p(p^p+\log(T)^p)+\mathfrak d_{\rm det}(T)^pp^p;\end{equation}see {\S}\ref{sec:estimates:max:reg}. But once we consider the running supremum of $\mathcal J_B(t)$, we run into the  additional growth term $\mathfrak d_{\rm det}(T)^p\log(T)^p$. 

\paragraph{Convolutions with \texorpdfstring{$S_L(t)$}{SL(t)}}

Let $S_{\rm tw}(t)$ and $S_L(t)$
denote the analytic semigroups generated by the linear operators $\mathcal L_{\rm tw}$  and $\mathcal L=\mathcal L_{\rm tw}+\Delta_y$, respectively. Note that $S_L(t)=S_{\rm tw}(t)S_H(t)=S_H(t)S_{\rm tw}(t).$ In addition to the analysis above, we also obtain
estimates for the stochastic convolution
\begin{equation}
\label{eq:cnv:def:z:x}
\mathcal Z_X(t)=\int_0^t S_L(t-s) P^\perp_{\rm tw} X(s)\mathrm dW_s,
\end{equation}
together with the associated weighted integral
\begin{equation}
\label{eq:cnv:def:i:x}
   \mathcal I_X(t)=\int_0^t(1+t-s)^{-\mu}\|\mathcal Z_X(s)\|_{H^{k+1}}^2\mathrm ds,  
\end{equation}
which is defined for every $\mu>0.$
Here $P_{\rm tw}^\perp = I-P_{\rm tw}:H^k\to H^k$ denotes the complement of the spectral projection of $\mathcal L_{\rm tw}$ which has been trivially extended to the transverse coordinate, i.e.,
\begin{equation}
    [P_{\rm tw}v](x,y)=\langle v(\cdot,y),\psi_{\rm tw}\rangle_{L^2_x}\Phi_0'(x).
\end{equation}
We impose the following condition on $X$.
\begin{itemize}
    \item[(HX)] For the process $X \in \mathcal{N}^2\big([0, T] ; \mathbb F  ; H S(\mathcal W_Q;H^k  ) \big)$ 
there exists a constant $\Xi_{*}>0$ such that 
the pathwise bound
\begin{equation}
\int_0^t (1+t-s)^{-\mu}\|X(s)\|_{H S(\mathcal W_Q;H^k  )}^2 \mathrm d s \leq \Xi_{*}^2
\label{eq:HB}
\end{equation}
holds for all $0 \leq t \leq T$.
\end{itemize}

This stochastic convolution \eqref{eq:cnv:def:z:x} is very similar to the ones we  have encountered before \cite{bosch2025conditionalspeedshapecorrections,bosch2024multidimensional,hamster2019stability,hamster2020diag,hamster2020expstability,hamster2020},
but we woud like to emphasise that (HX) does not contain a smoothening assumption in the spirit of \cite[condition (HB)]{bosch2024multidimensional}. In addition,
the algebraic weight in \eqref{eq:cnv:def:i:x} is novel.

\begin{proposition}[{see {\S}\ref{sec:spectral:gap}}] 
\label{prop:IX}
Suppose that \textnormal{(Hf)} and \textnormal{(HTw)} hold.   Assume that $T \ge 2$ is an integer and let $\mu>0$.
There exists a constant $K_{\rm gr} > 0$ that does not depend on $T$
so that for any process $X$  that satisfies \textnormal{(HX)}, and every integer $p \ge 1$,
we have the supremum bound
\begin{equation}
    \mathbb E \sup _{0 \leq t \leq T}\left\|\mathcal{Z}_X(t)\right\|_{H^k }^{2p}
    + \mathbb E \sup _{0 \leq t \leq T}\left\|\mathcal{I}_X(t)\right\|_{H^k }^{2p}
\leq K^{2p}_{\rm gr} \Xi_{*}^{2p} (p^p+\log (T)^p).\label{eq:IX(t)}
\end{equation}
\end{proposition}

We further formulate a result analogous to Corollary \ref{cor:E_B} for 
the situation that $X$ is uniformly bounded. 
The bound \eqref{eq:cnv:growth:bnd:exp:bnd:with:sup}
is sharper than one would obtain from a direct application of 
\eqref{eq:IX(t)}. This is related to the fact that 
the exponential decay allows the $(1-t-s)^{-\mu}$ weight to be replaced by $e^{-\epsilon(t-s)}$.
On the other hand, the estimate
\eqref{eq:cnv:l:cor:weighted:RHS5} does follow directly 
from \eqref{eq:IX(t)}, but involves
the term $\mathfrak d_{\rm det}(T)$ again, as in \eqref{eq:RHS5}.
This difference plays an important role in the estimates in {\S}\ref{sec:est:mixed:type}. 

\begin{proposition}[{see {\S}\ref{sec:spectral:gap}}] \label{prp:cnv:bnd:x:unif}
    Assume that $T \ge 2$ is an integer and consider $\Lambda_*>0$. 
There exists a constant $K_{\rm gr} > 0$ that does not depend on $T$
so that for any process $X$ that satisfies
\begin{equation}
     \|X(t)\|_{HS(\mathcal W_Q;H^k)}^2\leq \Lambda_*^2,\quad 0\leq t\leq T,
\end{equation}and every integer $p \ge 1$,
we have the growth bound
\begin{equation}
\mathbb E \sup _{0 \leq t \leq T}\left\|\mathcal{Z}_X(t)\right\|_{H^k }^{2p}
\leq 
K^{2p}_{\rm gr} \Lambda_*^{2p}(p^p+\log(T)^p) ,
\label{eq:cnv:growth:bnd:exp:bnd:with:sup}
\end{equation}
while for $\mu$ in the parameter regime \eqref{eq:mu:regime} we have
\begin{equation}
    \label{eq:cnv:l:cor:weighted:RHS5}
 \mathbb E\sup_{0\leq t\leq T}\mathcal I_X(t)^{p}
\leq K^{2p}_{\rm gr} \Lambda_*^{2p}\mathfrak d_{\rm det}(T)^p(p^p+\log(T)^p).
\end{equation}
\end{proposition}

\subsection{Convolution estimates for \texorpdfstring{$S_H(t)$}{SH(t)}}\label{sec:est:stoch:conv}
Proving the growth estimate \eqref{eq:E_B:only} can be achieved completely in line with \cite{bosch2024multidimensional,hamster2020expstability}  by following a chaining principle \cite{dirksen2015tail,talagrand2005generic} to control the long-time behaviour. Nevertheless, we  show this estimate by means of a more direct approach in which we split $[0,T]$ into intervals of length 1; see also \cite{bosch2025conditionalspeedshapecorrections}.
In particular, we consider the splitting
\begin{equation}
\mathcal{E}_B(t)=\mathcal{E}_B^{\mathrm{lt}}(t)+\mathcal{E}_B^{\mathrm{sh}}(t),
\end{equation}
where the long-time (lt) and short-time (sh) contributions are given by
\begin{equation}
\mathcal{E}_B^{\mathrm{lt}}(t)=\int_0^{\lfloor t\rfloor} S_H(t-s) B(s)\, \mathrm d W_s^Q\quad\text{and} \quad \mathcal{E}_B^{\mathrm{sh}}(t)=\int_{\lfloor t\rfloor}^t S_H(t-s)B(s)\, \mathrm d W_s^Q,
\end{equation}
respectively, where $\lfloor t\rfloor =i$ for $i\leq t<i+1$ and  any $i\in\mathbb Z_{\geq 0}.$ 
For the short-term integral, we rely on  the maximal inequality \eqref{eq:maximal:ineq:SH}.

\begin{lemma}[short-term bound]
     \label{lem:sh}  
     Assume that  $T \ge 2$ is an integer  and let $\mu>0$. There exists a constant $K_{\rm sh}>0$ that does not depend on $T$,
     so that for any process $B$ that satisfies
     \textnormal{(HB)}, and every integer $p\geq 1$, we have the supremum  bound
    \begin{equation}
        \mathbb E\sup_{0\leq t\leq T}\|\mathcal E_B^{\rm sh}(t)\|_{H^k_y}^{2p}\leq (p^p+\log(T)^p)K_{\rm sh}^{2p}\Theta_*^{2p}.
    \end{equation}
\end{lemma}
\begin{proof}
    Applying the maximal inequality \eqref{eq:maximal:ineq:SH} yields 
    \begin{equation}
    \begin{aligned}
        \mathbb E\sup_{i\leq t\leq i+1}\|\mathcal E_B^{\rm sh}(t)\|_{H^k_y}^{2p}&\leq K_H^{2p}p^p\mathbb E\left[\int_i^{i+1}\|B(s)\|_{HS(\mathcal W_Q;H^k_y)}^2\,\mathrm ds\right]^p\\
&\leq 2^{\mu p}K_H^{2p}p^p \mathbb E\left[\int_i^{i+1}(1+(i+1)-s)^{-\mu}\|B(s)\|_{HS(\mathcal W_Q;H^k_y)}^2\,\mathrm ds\right]^p\\
&\leq 2^{\mu p}K_H^{2p}p^p\Theta_*^{2p},
\end{aligned}
    \end{equation}
    on account of  \eqref{eq:theta:star}. Observing that
    \begin{equation}
        \sup_{0\leq t\leq T}\|\mathcal E_B^{\rm sh}(t)\|_{H^k_y}^{2p}\leq \max_{i\in\{0,\ldots,T-1\}}\sup_{i\leq t\leq i+1}\|\mathcal E_B^{\rm sh}(t)\|_{H^k_y}^{2p}
    \end{equation}
    and
     exploiting the estimate on maximum expectations in either \cite[Cor. B.3]{bosch2024multidimensional} or \cite[Cor. 2.4]{hamster2020expstability},
     we arrive at the assertion with $K_{\rm sh}^2=(24e)2^\mu K_H^2$.
\end{proof}
For the  long-term integral,  inequality \eqref{eq:maximal:ineq:SH} is unsuitable to retrieve a suitable bound, since it disregards all decay properties of the semigroup. We proceed by using  \eqref{eq:decay:H} to exploit the decay of the semigroup.

\begin{lemma}[long-term bound]\label{lem:Talagrand}
Assume that  $T \ge 2$ is an integer  and let $\mu\in (0,\frac{d-1}{2}]$ with $d\geq 2$. There exists a constant $K_{\mathrm{lt}} > 0$ that does  not depend on $T$, so that
     for any process $B$ that satisfies \textnormal{(HB)}, and any integer $p \geq 1$,  we have the supremum bound
\begin{equation}
\mathbb E \sup _{0 \leq t \leq T}\left\|\mathcal{E}_B^{\mathrm{lt}}(t)\right\|_{H^k_y}^{2 p} \leq\left(p^p+\log (T)^p\right) K_{\mathrm{lt}}^{2 p} \Theta_*^{2 p} .
\end{equation}

\end{lemma}
\begin{proof} For $t\geq 0$ we have
\begin{equation}
    \mathcal E_{B}^{\rm lt}(t)=S_H(t-\lfloor t\rfloor)\int_0^{\lfloor t\rfloor}S_H(\lfloor t\rfloor -s)B(s)\,\mathrm dW_s^Q.
\end{equation} 
We invoke estimate \eqref{eq:decay:H} to find
\begin{equation}
    \begin{aligned}
        \mathbb E\sup_{i\leq t\leq i+1}\left\|\mathcal{E}_B^{\mathrm{lt}}(t)\right\|_{H^k_y}^{2 p}&\leq \sup_{i\leq t\leq i+1}\|S_H(t-i)\|_{\mathscr L(H^k_y)}^{2p}\mathbb E\left\|\int_0^iS_H(i-s)B(s)\,\mathrm dW_s^Q\right\|_{H^k_y}^{2p}\\
        &\leq K_H^2p^p\mathbb E\left[\int_0^i\left\|S_H(i-s)B(s)\right\|_{HS(\mathcal W_Q;H^k_y)}^{2}\mathrm ds\right]^p.
    \end{aligned}
\end{equation}
Here we implicitly used the fact that $S_H(t)$ is a contractive semigroup; see Lemma \ref{lem:theta}. Observe that Corollary \ref{cor:algebraic:decay} together with condition \eqref{eq:theta:star} implies
\begin{equation}
\begin{aligned}
        &\mathbb E\sup_{i\leq t\leq i+1}\left\|\mathcal{E}_B^{\mathrm{lt}}(t)\right\|_{H^k_y}^{2 p}\\&\qquad\leq K_H^{2p}M_{\rm alg}^{2p}p^p\mathbb E\Bigg[\int_0^i(1+i-s)^{-\frac{d-1}2}\big[\left\|B(s)\right\|_{HS(\mathcal W_Q;H^k_y)}+\left\|B(s)\right\|_{\Pi_2(\mathcal W_Q;L^1_y)}\big]^2\mathrm ds\Bigg]^p\\&\qquad\leq 2^pK_H^{2p}M_{\rm alg}^{2p}p^p\Theta_*^2,
        \end{aligned}
\end{equation}
due to the  elementary inequality $(a+b)^2\leq 2a^2+2b^2.$ Similar reasoning as in the proof of Lemma \ref{lem:sh} allows us to conclude that the assertion holds with $K_{\rm lt}^2=48eK_H^2M_{\rm alg}^2.$
\end{proof}
\begin{corollary}\label{cor:E_B:only}
Assume that  $T \ge 2$ is an integer  and let $\mu\in (0,\frac{d-1}{2}]$ with $d\geq 2$.
There exists a constant $K_{\rm gr} > 0$ that does not depend on $T$
so that for any process $B$ that satisfies \textnormal{(HB)}, and every integer $p \ge 1$,
we have the supremum bound
\begin{equation}
\mathbb E \sup _{0 \leq t \leq T}\left\|\mathcal{E}_B(t)\right\|_{H^k_y }^{2p}
\leq K^{2p}_{\rm gr} \Theta_*^{2p} (p^p+\log (T)^p).\label{eq:E_B:only}
\end{equation}
\end{corollary}
\begin{proof}
   Combining the short-time result in Lemma \ref{lem:sh} with the long-time result in Lemma \ref{lem:Talagrand} yields the supremum bound.
\end{proof}



\subsection{Maximal regularity estimates for \texorpdfstring{$S_H(t)$}{SH(t)}}\label{sec:estimates:max:reg}
Our main result here  states that in a certain sense the integrated control
over the $H^k_y$-norm of first-order derivatives
of the convolution is assured
when the supremum of the stochastic convolution is well-behaved. The proposition below  can be seen as an adaptation of \cite[Prop 3.11]{bosch2024multidimensional} to the present setting involving algebraic decay.

\begin{proposition}[maximal regularity estimate]
\label{prop:max:reg:main:bnd}
Suppose $T>0$  and let $\mu\in (0,\frac{d-1}{2}]$ for $d\geq 2$.
There exists a constant $K_{\rm mr} > 0$ that does not depend on $T$
so that for any process
\begin{equation}
    B \in \mathcal N^{2p}\big( [0,T]; \mathbb F ; HS(\mathcal{W}_Q; H^k_y) \big),
\end{equation}
and every integer $p \ge 1$, we have the bound
\begin{equation}
\label{eq:fw:max:reg:bnd:j:perp:b}
\begin{aligned}
\mathbb E  [\mathcal J_B(t)]^{p}
 &\le  
    K_{\rm mr}^{p} \, \mathbb E   \sup_{0 \le r \le t} 
      \|\mathcal{E}_B(r)\|_{H^k_y}^{2p} 
    + p^{p} K_{\rm mr}^{p} \mathbb E
    \left[\int_0^t (1+t-r)^{-\mu} \| B(r) \|_{HS(\mathcal W_Q;H^k_y)}^2 \, \mathrm dr \right]^{p} \\
    & \qquad\qquad\,\,+ p^{p} K_{\rm mr}^{p} \mathbb E
    \left[\int_0^t (1+t-r)^{-\mu} \| B(r) \|_{\Pi_2(\mathcal{W}_Q;L^1_y )}^2 \, \mathrm dr \right]^{p},
\end{aligned}
\end{equation}
for any $0\leq t\leq T.$
\end{proposition}

The key towards coping with the first-order derivatives of the convolution is the identity
\begin{equation}
    \begin{aligned}
    \frac{\mathrm d}{\mathrm dt}\langle S_H(t-s)v,S_H(t-s')w\rangle_{H^k_y}&=-2\langle \nabla_yS_H(t-s)v,\nabla_yS_H(t-s')w\rangle_{H^k_y},  
    \end{aligned}
    \end{equation}
for any $t,s,s'\geq 0$ that satisfy $t>s$ and $t>s'.$
Let us introduce
the bilinear forms
\begin{equation}
\label{eq:fw:max:reg:def:j}
	\mathcal{B}(s_2',s_2, s_1', s_1)[v, w]
	= \int_{s_1'}^{s_2} (1+s_2'-s)^{-\mu}   \langle \nabla_y S_H(s-s_1) v,
	\nabla_y S_H(s-s_1')  w \rangle_{H^{k}_y} \mathrm ds,
\end{equation}
for any pair $v,w\in H^{k+1}_y$ and  any $0 \leq s_1\leq s_1'\leq s_2\leq s_2'\leq T$ fixed. Performing an integration by parts,
we obtain the decomposition
\begin{equation}
    \mathcal{B} (s_2',s_2,s_1', s_1)[v,w]
    =  \mathcal{B}_1(s_2',s_2,s_1',s_1)[v,w]
    + \mathcal{B}_2(s_2',s_2,s_1',s_1)[v,w]\label{eq:Jtotal}
\end{equation}
in terms of the integral form
\begin{align}
\mathcal{B}_1(s_2',s_2,s_1', s_1)[v,w]
    & =  \textstyle \frac{\mu}{2} \int_{s_1'}^{s_2}  (1+s_2'-s)^{-(\mu+1)}
          \langle S_H(s-s_1) v,   S_H(s-s_1') w \rangle_{H^k_y}
     \, \mathrm ds,
    \end{align}
     together with its boundary counterpart
     \begin{equation}
     \begin{aligned}
\textstyle
     \mathcal{B}_2(s_2',s_2,s_1',s_1)[v,w] &  =  \textstyle  
-\frac{1}{2}(1+s_2'-s_2)^{-\mu} \langle S_H(s_2-s_1) v, S_H(s_2- s_1') w \rangle_{H^k_y} \\&\textstyle\qquad\qquad+ \frac{1}{2} (1+ s_2'-s_1')^{-\mu}\langle  S_H(s_1'-s_1)v,  w \rangle_{H^k_y}.
\end{aligned} 
\end{equation}
The main point is that these expressions can be bounded using $H^k_y$-norms and $L^1_y$-norms only.  In this section, it suffices to determine estimates where $s_1'=s_1$ and  $s_2'=s_2$;  in {\S}\ref{sec:det:bounds:critical} we encounter this bilinear form in its full form.
We shall see that for certain future estimates we do not want to end up with an $L^1_y$-norm  in both $v$ and $w$, and therefore we provide the alternative estimate \eqref{eq:J-slower}. 

\begin{lemma}
\label{lem:max:reg:bnds:j:delta} 
Suppose $T>0$  and let $\mu\in(0,\frac{d-1}{2}]$ for $d\geq 2.$
There exists a constant $K_{\mathcal B}>0$  that
does not depend on $T$ so that  we have the estimate
\begin{equation}
    \begin{array}{lcl}
    | \mathcal B(s_2,s_2,s_1,s_1)[v, w] | \label{eq:J-faster}
     & \le & K_{\mathcal B} (1+s_2-s_1)^{-\mu} \big[\|v\|_{L^1_y}+\|v\|_{H^k_y}\big]\big[|w\|_{L^1_y}+\|w\|_{H^k_y}\big],
    \end{array}
\end{equation}
 for any $0 \le s_1 \le s_2 \le T$ and any  $v,w \in H^{k+1}$.
As an alternative estimate, we have
 \begin{equation}
    \begin{array}{lcl}
    | \mathcal B(s_2,s_2,s_1,s_1)[v, w] |
     & \le & K_{\mathcal B} (1+s_2-s_1)^{-\nu} \|v\|_{H^k_y}\big[\|w\|_{L^1_y}+\|w\|_{H^k_y}\big],\label{eq:J-slower}
    \end{array}
\end{equation}
where $\nu\leq \min\{\mu,\frac{d-1}4\}.$ 
\end{lemma}
\begin{proof}
Let us first focus on proving the    estimates for $\mathcal B_1(s_2,s_2,s_1,s_1)[v,w]$.
Appealing to Corollary \ref{cor:algebraic:decay} in combination with Lemma \ref{lem:z}  shows that
    \begin{equation}
    \begin{array}{lcl}
      \mathcal{B}_1(s_2, s_2, s_1, s_1)[v, w]
      & \le & 
        \tfrac{\mu}2M_{\rm alg}^2\int_{s_1}^{s_2}(1+s_2-s)^{-\color{black}(\mu+1)}(1+s-s_1)^{-(d-1)/2}\,\mathrm ds
        \\[0.2cm]
        & \leq &  C_1 (1+s_2-s_1)^{-\mu},\label{eq:poly:bound}
    \end{array}
    \end{equation}
  for some $C_1>0,$ since $\color{black}(\mu+1)+({d-1})/{2}-1>\mu$ for $d\geq 2.$
For the alternative estimate, we  can trivially bound the $v$-component (using Lemma \ref{lem:theta}; $M=1$) and appeal to the same corollary and lemma as before  
to  arrive at
    \begin{equation}
    \begin{array}{lcl}
        |\mathcal B_1(s_2,s_2,s_1,s_1)[v,w]|
        &\leq & \tfrac{\mu}2\|v\|_{H^k_y}\int_{s_1}^{s_2}(1+s_2-s)^{\color{black}-(\mu+1)}\|S_H(s-s_1)w\|_{H^k_y}\,\mathrm ds
        \\[0.2cm]
        & \le & 
        \tfrac{\mu}2M_{\rm alg}\int_{s_1}^{s_2}(1+s_2-s)^{-\color{black}(\mu+1)}(1+s-s_1)^{-(d-1)/4}\,\mathrm ds
        \\[0.2cm]
        & \leq & C_2 (1+s_2-s_1)^{-\lambda},
    \end{array}
    \end{equation}
    for some $C_2>0,$    where $\lambda\leq \min\{\mu+1,\frac{d-1}4,\mu+\frac{d-1}4\}$ for $d\neq 5$ 
while for  $d=5$ we require in addition $\lambda< \mu+1.$ 
    Note that estimating $\mathcal B_2(s_2,s_2,s_1,s_1)[v,w]$ is rather straightforward, resulting into \eqref{eq:J-faster} as well as \eqref{eq:J-slower} upon observing that $\nu\leq \min\{\mu,\lambda\}=\min\{\mu,\frac{d-1}4\}$ needs to be satisfied.
\end{proof}

An application of the mild It\^o formula \cite{da2019mild} yields
\begin{equation}
\label{eq:fw:max:req:mild:ito}
	\begin{aligned}
	\| \nabla_y \mathcal E_B(s)\|^2_{H^k_y}
	&=    
	\sum_{j=0}^\infty \int_{0}^s \langle\nabla_yS_H(s- r)B(r) \sqrt{Q} e_j,
	\nabla_yS_H(s- r) B(r) \sqrt{Q} e_j \rangle_{H^{k}_y}  \, \mathrm dr
	\\&\qquad\qquad\qquad+ 2 \int_{0}^s \langle \nabla_y S_H(s- r) \mathcal E_B(r), \nabla_yS_H(s- r) B(r)[\,\cdot\,]   \rangle_{H^{k}_y} \mathrm dW_r^Q,
\end{aligned}
\end{equation}
for every $0\leq s\leq T.$
After  reversing the order of integration, this results into the bound
\begin{equation}
    \mathcal J_B(t)\leq  \mathcal I_I(t)+2\mathcal I_{II}(t),\quad 0\leq t\leq T,\label{eq:comp:begin}
\end{equation}
where  
\begin{align}
    \mathcal I_I(t)&=\int_0^t\sum_{j=0}^\infty \mathcal K_I^{(j)}(t,r)\,\mathrm dr,\\
 \mathcal I_{II}(t)&=\int_0^t\int_{r}^t (1+t-s)^{-\mu} \langle \nabla_yS_H(s- r) \mathcal E_B(r), \nabla_y S_H(s- r) B(r)[\,\cdot\,]   \rangle_{H^{k}_y} \mathrm ds \,\mathrm dW_r^Q,
\end{align}
with
\begin{equation}
\begin{aligned}
    \mathcal K_I^{(j)}(t,r)&= \int_{r}^t (1+t-s)^{-\mu} \langle\nabla_y S_H(s- r)B(r) \sqrt{Q} e_j,
	\nabla_yS_H(s- r) B(r) \sqrt{Q} e_j \rangle_{H^{k}_y}  \,\mathrm ds\\
    &=\mathcal B(t,t,r,r)[B(r)\sqrt{Q}e_j,B(r)\sqrt{Q}e_j].
    \end{aligned}
\end{equation}
Note that \eqref{eq:ito} yields
\begin{equation}
    \mathbb  E[\mathcal I_{II}(t)]^{p}\leq p^{p/2}K_{\rm cnv}^p \mathbb E\left[\int_0^t\sum_{j=0}^\infty \mathcal K_{II}^{(j)}(t,r)^2\,\mathrm dr\right]^{p/2},\end{equation}
in which
\begin{equation}\begin{aligned}
    \mathcal K_{II}^{(j)}(t,r)&=\int_r^t (1+t-s)^{-\mu} \langle \nabla_yS_H(s- r) \mathcal E_B(r), \nabla_yS_H(s- r) B(r)\sqrt Qe_j    \rangle_{H^{k}_y} \mathrm ds\\&=\mathcal B(t,t,r,r)[\mathcal E_B(r),B(r)\sqrt{Q}e_j].\label{eq:comp:end}\end{aligned}
\end{equation}
In the same spirit as \cite[Lem. 3.10 and Lem. 4.4]{hamster2020expstability}, we shall prove
the pathwise estimates below, consequently enabling us to prove Proposition \ref{prop:max:reg:main:bnd} and hence Proposition \ref{prop:E_B}.

\begin{lemma}
Suppose $T>0$  and let $\mu\in(0,\frac{d-1}{2}]$ for $d\geq 2.$ 
There exists a constant $K_0>0$  that does not depend on $T$ so that  we have the estimate
    \begin{equation}
        \sum_{j=0}^\infty \mathcal K_{I}^{(j)}(t,r)\leq K_0(1+t-r)^{-\mu}[\|B(r)\|_{HS(\mathcal W_Q;H^k_y)}^2+\|B(r)\|_{\Pi_2(\mathcal W_Q;L^1_y)}^2\big]\label{eq:KIbound},
    \end{equation}
for any $0\leq r\leq t\leq T,$  together with
    \begin{equation}
        \sum_{j=0}^\infty \mathcal K_{II}^{(j)}(t,r)^2\leq K_0(1+t-r)^{-2\nu}\|\mathcal E_B(r)\|_{H^k_y}^2\big[\|B(r)\|_{HS(\mathcal W_Q;H^k_y)}^2+\|B(r)\|_{\Pi_2(\mathcal W_Q;L^1_y)}^2\big],\label{eq:KIIbound}
    \end{equation}
where $\nu\leq \min\{\mu,\frac{d-1}4\}.$
\end{lemma}
\begin{proof}
These bounds 
readily follow from \eqref{eq:J-faster} and \eqref{eq:J-slower}, respectively.
\end{proof}
We are now ready to prove the maximal regularity result. In order to do so, let us introduce  the terms
\begin{equation}
    \begin{aligned}\mathcal S_{\delta}[B]&=\sup_{0\leq r\leq t} (1+t-r)^{-\delta}\|\mathcal E_B(r)\|_{H^k_y}^2,\\
\mathcal H_{\delta';p}[B]&=\left[\int_0^t (1+t-r)^{-\delta'}\|B(r)\|_{HS(\mathcal W_Q;H^k_y)}^2\big]\,\mathrm dr\right]^p\\&\qquad\qquad+\left[\int_0^t (1+t-r)^{-\delta'}\|B(r)\|_{\Pi_2(\mathcal W_Q;L^1_y)}^2\big]\,\mathrm dr\right]^p,\label{eq:Sdelta:Hdeltap}
\end{aligned}
\end{equation}
for $\delta,\delta'\geq 0$ and $p\geq 1.$  
\begin{proof}[Proof of Proposition \ref{prop:max:reg:main:bnd}]  As a result of the computations \eqref{eq:comp:begin}--\eqref{eq:comp:end}, in combination with the bounds \eqref{eq:KIbound} and \eqref{eq:KIIbound}, we obtain 
\eqref{eq:fw:max:reg:bnd:j:perp:b} after noting that 
\begin{equation}
\begin{aligned}   \label{eq:splitting}p^{p/2}\left[\int_0^t\sum_{k=0}^\infty \mathcal K_{II}^{(k)}(t,r)^2\,\mathrm dr\right]^{p/2}&\leq 2^{2p}K_0^p\big[\mathcal S_0[B]^p+p^p\mathcal H_{2\nu;p}[B]\big],\end{aligned}
\end{equation}
with $\nu = \min\{\mu,\frac{d-1}{4}\}$, which follows from the  inequality $p^{p/2}(ab)^{p/2}
=(apb)^{p/2}
\leq 2^{p/2}[a^p+p^pb^p].$
Estimate \eqref{eq:fw:max:reg:bnd:j:perp:b} readily follows from the fact that $\mu\leq  2\nu$ 
holds  in the parameter regime  $\mu\in (0,\frac{d-1}{2}]$ for any dimension $d\geq 2$.
\end{proof}
Note that for $d>5$ and $\mu\in (1,\frac{d-1}4],$ we  have $\nu=\mu$. This particularly allows us to improve \eqref{eq:splitting} by letting the right-hand side depend on the terms $\mathcal S_\mu[B]$ and $\mathcal H_{\mu;p}[B]$, as in \eqref{eq:Sdelta:Hdeltap}, and thus
\begin{align}
    \nonumber\mathbb E[\mathcal I_{II}(t)]^p&\leq K^p\sup_{0\leq s\leq t}(1+t-s)^{-\mu}\|\mathcal E_B(s)\|_{H^k_y}^2+p^pK^p\left[\int_0^t(1+t-s)^{-\mu}\|B(s)\|_{HS(\mathcal W_Q;H^k_y)}\mathrm ds\right]^p\\
    &\qquad\qquad+p^pK^p\left[\int_0^t(1+t-s)^{-\mu}\|B(s)\|_{\Pi_2(\mathcal W_Q;L^1_y)}\mathrm ds\right]^p,
    \end{align} 
for some $K>0.$ In the spirit of the weighted decay estimate in  \cite[Prop. 3.10]{bosch2024multidimensional}, see also the results in \cite[Sec. 3.4]{bosch2024multidimensional}, we expect
\begin{equation}
\begin{aligned}
    \mathbb E[\mathcal S_{\mu}[B]^p]&\leq p^pK^p\mathbb E[\mathcal H_{\mu;p}[B]],
    \end{aligned}
\end{equation}
for some $K>0,$
causing the supremum of the stochastic convolution term in \eqref{eq:fw:max:reg:bnd:j:perp:b} to be superfluous; see also {\S}\ref{sec:spectral:gap}. Recall that in \cite[Prop 3.11] {bosch2024multidimensional} this supremum pops up in the growth estimate due to the presence of ``crosstalk'' as a result of having to deal with forward integrals.



\begin{proof}[Proof of Proposition \ref{prop:E_B}] The proof follows the arguments in \cite[Prop. 3.18]{bosch2024multidimensional}. In view of Corollary \ref{cor:E_B:only} together with Proposition \ref{prop:max:reg:main:bnd}, we find that the  moment bounds for $\mathcal J_B(t)$, 
for any $0\leq t\leq T,$ are given by
\begin{equation}
    \mathbb E [\mathcal J_B(t)^{p}] \le K^{p}   \Theta_*^{2p}(p^{p}+\log(T)^{p})+p^{p}K^{p}\Theta_*^{2p}\leq p^p\Theta_1^{p}+\Theta_2^{p},\label{eq:I_B^p}
\end{equation}
for some  $K>0$, with $\Theta_1=2 K\Theta_*^2$ and $\Theta_2=\frac{1}{2}\Theta_1\log(T).$ 
Now \cite[Cor. B.3]{bosch2024multidimensional}
implies
\begin{equation}
    \mathbb E\max_{i\in\{1,\ldots,T\}}\mathcal J_B(i)^{p}\leq  K^{p}\Theta_*^{2p}(p^{p}+\log(T)^{p}),
\end{equation}
after updating $K>0$.
To conclude, it  suffices to observe
that 
\begin{align}
     \sup_{0\leq t\leq T} \mathcal J_B(t)&= \max_{i\in\{1,\ldots,T\}}\sup_{i-1\leq t\leq i} \mathcal J_B(t)\leq 2^\mu\max_{i\in\{1,\ldots,T\}} \mathcal J_B(i)
\end{align} 
holds for any integer $T\geq 2$. 
In more detail, the ultimate inequality holds true because
\begin{equation}
\begin{aligned}
    \sup_{i-1\leq t\leq i}\mathcal J_B(t)&=\sup_{0\leq u\leq 1} \int_0^{u+i-1}(u+i-s)^{-\mu}\|\mathcal E_B(s)\|_{H^{k+1}_y}^2\mathrm ds\\
    &\leq 2^\mu\sup_{0\leq u\leq 1} \int_0^{u+i-1} (1+i-s)^{-\mu}\|\mathcal E_B(s)\|_{H^{k+1}_y}^2\mathrm ds\\
    &\leq 2^\mu \mathcal J_B(i),
    \end{aligned}
\end{equation}
since $
    (u+x)^{-\mu}\leq 2^\mu(1+x)^{-\mu}$ for $ x\geq 1-u$.
\end{proof}
\subsection{Estimates for \texorpdfstring{$S_L(t)$}{SL(t)}} 
\label{sec:spectral:gap}
In this section, we  discuss the estimates related to the  stochastic convolution \eqref{eq:cnv:def:z:x}.
For our first supremum bound it is possible to relax the condition (HX) to only
require pathwise bounds on $X$ with respect to an exponential weight. In particular, we impose the following condition.
\begin{itemize}
    \item[(HExp)] For the process $X \in \mathcal{N}^2\big([0, T] ; \mathbb F  ; H S(\mathcal W_Q;H^k  ) \big)$ 
there exists a constant $\Xi_{\exp}>0$ such that 
the pathwise bound
\begin{equation}
\int_0^t e^{-\varepsilon(t-s)}\|X(s)\|_{H S(\mathcal W_Q;H^k  )}^2 \mathrm d s \leq \Xi_{\exp}^2
\label{eq:ZXeps}
\end{equation}
holds for all $0 \leq t \leq T$.
\end{itemize}

 For the growth bounds below, we  need $0<\epsilon\leq 2\beta$. Recall that the constant $\beta>0$ is defined in (HTw) and is  associated to the spectral gap of the  operator  $\mathcal L_{\rm tw}.$  
\begin{proposition}\label{prop:ZX}
Suppose that \textnormal{(Hf)} and \textnormal{(HTw)} holds.   Assume that $T \ge 2$ is an integer and let $\epsilon \in (0, 2 \beta]$.
There exists a constant $K_{\rm gr} > 0$ that does not depend on $T$
so that for any process $X$  that satisfies \textnormal{(HExp)}, and every integer $p \ge 1$,
we have the supremum bound 
\begin{equation}
\mathbb E \sup _{0 \leq t \leq T}\left\|\mathcal{Z}_X(t)\right\|_{H^k }^{2p}
\leq K^{2p}_{\rm gr} \Xi_{\rm exp}^{2p} (p^p+\log (T)^p).
\label{eq:supZX}
\end{equation}

\end{proposition}
\begin{proof}
This can be deduced from  computations similar to those in {\S}\ref{sec:est:stoch:conv}. For the short-term integrals we  invoke the  maximal inequality \eqref{eq:maximal:ineq:SL}, whereas for the long-term integrals we   exploit the decay of the semigroup by using \eqref{eq:decay}.
\end{proof}

Observe that (HX) implies (HExp) on account of \eqref{eq:absorb:decay}. We can actually control the weighted integral $\mathcal I_X(t)$ with $(1-t-s)^{-\mu}$ replaced by $e^{-\epsilon(t-s)}$, and (HX) replaced by (HExp), but we do not need it in this paper.  The latter  explains why the  growth bound in Proposition \ref{prop:IX} can be achieved without any restriction on the decay rate $\mu.$ 
 To establish this bound, we follow
 the observations in 
\cite{bosch2024multidimensional,hamster2019stability,hamster2020expstability}, 
 and recall the identity
\begin{equation}
\begin{aligned}
    \frac{\mathrm d}{\mathrm dt}\langle S_{\rm tw}(t-s)v,S_{\rm tw}(t-s')w\rangle_{H^k_x}&=
    \langle S_{\rm tw}(t-s)\mathcal Av,S_{\rm tw}(t-s')w \rangle_{H^{k}_x}\\
    &\qquad-2\langle\partial_xS_{\rm tw}(t-s)v,\partial_xS_{\rm tw}(t-s')w\rangle_{H^{k}_x}, 
\end{aligned}
\end{equation}
which holds for any $t,s,s'\geq 0$ that satisfy $t>s$ and $t>s'$.  Here $\mathcal A=\mathcal L_{\rm tw}+\mathcal L_{\rm tw}^{\rm adj}-2\partial_x^2$ is a bounded operator on $H^k_x$. We can trivially extend $\mathcal A$ to an operator defined on $H^k$ and shall denote it by $A.$ 
In addition, we introduce a new inner product on  $H^{k+1}$, given by
\begin{equation}
    \langle\langle v,w\rangle\rangle_{H^{k+1}}=\langle v,w\rangle_{H^k}+\langle \partial_xv,\partial_x w\rangle_{H^k},
\end{equation}
which is again equivalent to the original inner product in the sense that we have the inequalities
\begin{equation}
    \langle v,v\rangle_{H^{k+1}}\leq \langle \langle v,v\rangle\rangle_{H^{k+1}}\leq 2\langle v,v\rangle_{H^{k+1}}. 
\end{equation}
This allows us to write
\begin{equation}
\begin{aligned}
    \frac{\mathrm d}{\mathrm dt}\langle S_{L}(t-s)v,S_{L}(t-s')w\rangle_{H^k}&=
    \langle S_L(t-s)Av,S_L(t-s') w \rangle_{H^{k}}\\
    &\qquad +2\langle S_L(t-s)v,S_L(t-s)w \rangle_{H^{k}}\\[.2cm]&\qquad\qquad-2\langle\langle  S_L(t-s)v,S_L(t-s)w\rangle \rangle_{H^{k+1}}.
    \label{eq:identity2}
\end{aligned}
\end{equation}

On account of the equivalence above, we  may replace $\|\cdot\|_{H^{k+1}}^2$ in \eqref{eq:cnv:def:i:x} with $\langle\langle\cdot,\cdot\rangle\rangle_{H^{k+1}}$. In a similar fashion as {\S}\ref{sec:estimates:max:reg}, we  define the bilinear form
\begin{equation}
\label{eq:bil:form:V}
	\mathcal{V}(s_2',s_2, s_1',s_1)[v, w]
	= \int_{s_1'}^{s_2} (1+s_2'-s)^{-\mu} \langle  \langle S_L(s-s_1) P^\perp_{\rm tw}  v,
	S_L(s-s_1')  P^\perp_{\rm tw}  
w \rangle\rangle_{H^{k+1}} \mathrm ds,
\end{equation}
for $0\leq s_1\leq s_1'\leq s_2\leq s_2'\leq T$ and any pair $v,w\in H^{k+1}.$ One can show that
\begin{equation}
    |\mathcal V(s_2,s_2,s_1,s_1)[v,w]|\leq K_{\mathcal V}(1+s_2-s_1)^{-\mu}\|v\|_{H^k}\|w\|_{H^k}
\end{equation}
    holds for some $K_{\mathcal V}>0.$ 
\begin{proposition}[maximal regularity estimate]
Suppose $T>0$  and let $\mu>0.$
There exists a constant $K_{\rm mr} > 0$ that does not depend on $T$
so that for any process
\begin{equation}
    X \in \mathcal N^{2p}\big( [0,T]; \mathbb F ; HS(\mathcal{W}_Q; H^k) \big),
\end{equation}
and every integer $p \ge 1$, we have the bound
\begin{equation}
\mathbb E  [\mathcal I_X(t)]^{p}
 \le  K_{\rm mr}^{p} \, \mathbb E   \sup_{0 \le r \le t} 
      \|\mathcal{Z}_X(r)\|_{H^k_y}^{2p}+
     p^{p} K_{\rm mr}^{p} \mathbb E
    \left[\int_0^t (1+t-r)^{-\mu} \| X(r) \|_{HS(\mathcal W_Q;H^k)}^2 \, \mathrm dr \right]^{p},
\end{equation}
for any $0\leq t\leq T.$
\end{proposition}
\begin{proof}
The computations are completely analogous to those  in {\S}\ref{sec:estimates:max:reg}, but this time we need to exploit the bilinear form \eqref{eq:bil:form:V}.
\end{proof}
\begin{proof}[Proof of Proposition \ref{prop:IX}]
    The proof is identical to the proof of Proposition \ref{prop:E_B}.
\end{proof}
\begin{proof}[Proof of Proposition \ref{prp:cnv:bnd:x:unif}]
    The bound \eqref{eq:cnv:growth:bnd:exp:bnd:with:sup} is a direct corollary of Proposition \ref{prop:ZX}, while \eqref{eq:cnv:l:cor:weighted:RHS5} follows immediately from Proposition \ref{prop:IX}.
\end{proof}



\section{Weighted integrals of deterministic convolutions}\label{sec:det:bounds:critical}
In this section, we provide estimates regarding weighted integrals of deterministic convolutions.
As in {\S}\ref{sec:supremum}, we
consider the semigroups $S_H(t)$ and $S_L(t)$,
but a crucial difference is
that the critical value for $\mu$ is now given by $(d-1)/4$ instead of $(d-1)/2.$
Indeed, we also briefly discuss what happens when one goes beyond this critical value.

\paragraph{Convolutions with \texorpdfstring{$S_H(t)$}{SH(t)}}
Consider the convolution
\begin{equation}
    \mathcal E_G(t)=\int_0^tS_H(t-s)G(s)\mathrm ds
\end{equation}
together with the weighted integral
\begin{equation}
   \mathcal J_G(t)= \int_0^t(1+t-s)^{-\mu}\|\nabla_y \mathcal E_G(s)\|_{H^k_y}^2\mathrm ds,\label{eq:JG(t)}
\end{equation}
where $G$ satisfies the following a priori pathwise condition. 
\begin{itemize}
    \item[(HG)] The process  $G:[0,T]\times \Omega\to H^k_y$ has paths  $\mathbb P$-almost surely in $L^1([0,T];H^k_y)\cap L^1([0,T];L^1_y)$ and there exists a constant $\Theta_*>0$ such that
\begin{equation}\int_0^t (1+t-s)^{-\mu}\big[\|G(s)\|_{H^k_y}+\|G(s)\|_{L^1_y}\big]\mathrm ds\leq \Theta_*\end{equation}
holds for all $0\leq t\leq T.$
\end{itemize}


\begin{proposition}\label{prop:deter:conv}
Suppose $T>0$   and let $\mu\in (0,\frac{d-1}{4}]$ for $d\geq 2$.
There exists a constant $K_{\rm dc} > 0$ that does not depend on $T$
so that, for any process $G$ that satisfies \textnormal{(HG)},   the bound   
\begin{equation}
    \sup_{0\leq t\leq T} \mathcal J_G(t)\leq K_{\rm dc}\Theta_*^{2},
\end{equation}
holds $\mathbb P$-almost surely.
\end{proposition}

The  estimate above turns out to be critical, in the sense that for any $\mu > \frac{d-1}{4}$ we obtain
\begin{equation}
\begin{aligned}
    \sup_{0\leq t\leq T}| \mathcal J_G(t)|&\leq K_{\rm dc}\left[\sup_{0\leq t\leq T}\int_0^t (1+t-s)^{- \frac{d-1}{4}}\big[\|G(s)\|_{H^k_y}+\|G(s)\|_{L^1_y}\big]\mathrm ds\right]\\
    &\qquad\qquad  \times \left[\sup_{0\leq t\leq T}\int_0^t(1+t-s)^{-\min\{\mu, \frac{d+3}{4}\}}\big[\|G(s)\|_{H^k_y}+\|G(s)\|_{L^1_y}\big]\mathrm ds\right].\label{eq:JG*}
    \end{aligned}
\end{equation}
Even if we replace the algebraic rate  in \eqref{eq:JG(t)} with an exponential weight by writing
\begin{equation}
    \mathcal J_G^\epsilon(t)=\int_0^te^{-\epsilon(t-s)}\|\nabla_y \mathcal E_G(s)\|_{H^k_y}^2\mathrm ds,
\end{equation}
for some $\epsilon >0$, we end up with the same upper bound as in \eqref{eq:JG*}, where $\min\{\mu, \frac{d+3}{4}\}$ reads as $\frac{d+3}{4}$.


\paragraph{Convolutions with \texorpdfstring{$S_L(t)$}{SL(t)}} We now consider the deterministic convolution
\begin{equation}
    \mathcal Z_F(t)=\int_0^tS_L(t-s)P_{\rm tw}^\perp F(s)\mathrm ds,
\end{equation}
together with the associated weighted integral
\begin{equation}
   \mathcal I_F(t)= \int_0^t(1+t-s)^{-\mu}\|\mathcal Z_F(s)\|_{H^{k+1}}^2\mathrm ds,\label{eq:IF(t)}
\end{equation}
where $F$ satisfies the following a priori pathwise condition. 
\begin{itemize}
    \item[(HF)] The process  $F:[0,T]\times \Omega\to H^k$ has paths  $\mathbb P$-almost surely in $L^1([0,T];H^k)$ and there exists a constant $\Xi_*>0$ such that \begin{equation}
    \label{eq:cnvd:bnd:a:prp:F:Xi}
    \int_0^t (1+t-s)^{-\mu}\|F(s)\|_{H^k}\,\mathrm ds\leq \Xi_*\end{equation}
    holds for all $0\leq t\leq T.$
\end{itemize}

\begin{proposition}\label{prop:deter:conv:gap}
Assume that \textnormal{(Hf)} and \textnormal{(HTw)} holds. Suppose $T>0$   and $\mu> 0$.
There exists a  $K_{\rm dc} > 0$ that does not depend on $T$
so that, for any process $F$ that satisfies \textnormal{(HF)},  the bound   
\begin{equation}
    \sup_{0\leq t\leq T} \mathcal I_F(t)\leq K_{\rm dc}\Xi_*^{2},\label{eq:IF:bound}
\end{equation}
holds $\mathbb P$-almost surely.
\end{proposition}

Let us emphasise that in the result above one can take an arbitrary $\mu > 0$ on account of the exponential decay of $S_L(t) P^\perp_{\rm tw}$.  In fact, we could have considered the weight $e^{-\epsilon(t-s)}$ instead of $(1+t-s)^{-\mu}$ in \eqref{eq:IF(t)} and \eqref{eq:cnvd:bnd:a:prp:F:Xi}, once restricted to  $0<\epsilon\leq 2\beta$ with $\beta>0$ as in (HTw).


\subsection{Proofs of Propositions \ref{prop:deter:conv} and \ref{prop:deter:conv:gap}}
Let us
focus on the convolution with the heat semigroup first. By linearity and continuity of the inner product, we observe that
\begin{equation}
\begin{aligned}
    \|\nabla_y \mathcal E_{G}(s)\|_{H^{k}_y}^2&=\langle \nabla_y \mathcal E_{G}(s),\nabla_y \mathcal E_{G}(s)\rangle _{H^{k}_y}\\&=
    \int_0^s\int_0^s\langle \nabla_y S_H(s-r)G(r),\nabla_y S_H(s-r')G(r')\rangle_{H^{k}_y}\mathrm dr'\,\mathrm dr.\end{aligned}
\end{equation}
  We proceed by splitting $\mathcal J_G(t)$ into a long-term and short-term component, given by 
\begin{equation}
\label{eq:cnvd:split:j:g}
    \begin{aligned}
        \mathcal J_G^{\rm lt}(t)&= \int_0^t(1+t-s)^{-\mu}  \int_{0}^{s-1}\int_{0}^{s-1}\langle \nabla_y S_H(s-r)G(r),\nabla_y S_H(s-r')G(r')\rangle_{H^{k}_y}\mathrm dr'\,\mathrm dr\,\mathrm ds, \\
        \mathcal J_G^{\rm sh}(t)& = \int_0^t(1+t-s)^{-\mu}  \int_{s-1}^{s}\int_{s-1}^{s}\langle \nabla_y S_H(s-r)G(r),\nabla_y S_H(s-r')G(r')\rangle_{H^{k}_y}\mathrm dr'\,\mathrm dr\,\mathrm ds.
    \end{aligned}
\end{equation}
For the long-term integral, we are able to exploit the fact that $\nabla_yS_H(t)$ decays faster than $S_H(t)$; see Lemma \ref{lem:theta} 
and Corollary \ref{cor:algebraic:decay}.
 
\begin{lemma}[long-term bound]\label{lem:deter:conv:long}
   Suppose $T>0$   and let $\mu\in (0,\frac{d-1}{2}]$ for $d\geq 2$.
There exists a constant $K_{\rm dc} > 0$ that does not depend on $T$
so that  the bound
\begin{equation}
    \sup_{0\leq t\leq T}\mathcal J^{\rm lt}_G(t)\leq K_{\rm dc}^p\left[\sup_{0\leq t\leq T}\int_0^t (1+t-s)^{-\mu}\big[\|G(s)\|_{H^k_y}+\|G(s)\|_{L^1_y}\big]\mathrm ds\right]^{2} \label{eq:JGlt}
\end{equation}
holds $\mathbb P$-almost surely.
\end{lemma}
\begin{proof}
Applying Cauchy--Schwartz and introducing the shorthand
$Z(t)=\|G(t)\|_{L^1_y}+\|G(t)\|_{H^k_y}$, we obtain
    \begin{align}
      \nonumber  \mathcal J_G^{\rm lt}(t)&\leq \int_0^t(1+t-s)^{-\mu}  \int_{0}^{s-1}\int_{0}^{s-1}\| \nabla_y S_H(s-r)G(r)\|_{H^{k}_y}\|\nabla_y S_H(s-r')G(r')\|_{H^{k}_y}\mathrm dr'\,\mathrm dr\,\mathrm ds \\
 &\leq \int_0^t(1+t-s)^{-\mu}  \left(\int_{0}^{s-1}\|\nabla_y S_H(s-r)G(r)\|_{H^{k}_y}\mathrm dr\right)^2\mathrm ds\\
    \nonumber&\leq  M_{\rm alg}^2\int_0^t(1+t-s)^{-\mu}  \left(\int_{0}^{s}(1+s-r)^{-\frac{d+1}{4}}Z(r)\,\mathrm dr\right)^2\mathrm ds,
    \end{align}
which follows from Corollary 
\ref{cor:algebraic:decay}. 
A further application of Cauchy--Schwarz
yields
    \begin{align}
        \nonumber\left(\int_{0}^{s}(1+s-r)^{-\frac{d+1}{4}}Z(r)\,\mathrm dr\right)^2&= \left(\int_{0}^{s}(1+s-r)^{-\frac{d+3}{8}}\sqrt{Z(r)}\, (1+s-r)^{-\frac{d-1}{8}}\sqrt{Z(r)}\,\mathrm dr\right)^2\\&\leq \int_0^s(1+s-r)^{-\frac{d+3}{4}}Z(r)\,\mathrm dr\int_0^s(1+s-r)^{-\frac{d-1}{4}}Z(r)\,\mathrm dr\\
        &\leq \mathcal R(T)\int_0^s(1+s-r)^{-\frac{d+3}{4}}Z(r)\,\mathrm dr,\nonumber
    \end{align}
    with 
    $
        \mathcal R(T)=\sup_{0\leq t\leq T}\int_0^t(1+t-s)^{-\mu}Z(s)\,\mathrm ds.$
Upon reversing the order of integration, we conclude
\begin{equation}
    \begin{aligned}
        \mathcal J_G^{\rm lt}(t)&\leq M^2_{\rm alg}\mathcal R(T)\int_0^t(1+t-s)^{-\mu} 
 \int_{0}^{s}(1+s-r)^{-\frac{d+3}{4}}Z(r)\,\mathrm dr \,\mathrm ds\\
 &=M^2_{\rm alg}\mathcal R(T) \int_0^tZ(r)
 \int_{r}^{t}(1+t-s)^{-\mu} (1+s-r)^{-\frac{d+3}{4}}\,\mathrm ds \,\mathrm dr\\
 &\leq M^2_{\rm alg}C\mathcal R(T) \int_0^t(1+t-r)^{-\mu}  Z(r)\,\mathrm dr\\
 &\leq M^2_{\rm alg}C\mathcal R(T)^2
    \end{aligned}
\end{equation}
for some $C>0.$
The penultimate inequality follows from  Lemma  \ref{lem:z} and the fact $({d+3})/{4}>1$ for $d\geq 2$ (which would have failed without the additional decay associated to $\nabla_y S_H$). This implies the stated bound  \eqref{eq:JGlt}.
\end{proof}

Let us now analyse $\mathcal J_G^{\rm sh}(t)$. We would like to point out that the computations below have a high resemblance to 
those in \cite[Sec. 9.2]{hamster2019stability}.
\begin{lemma}[short-term bound]\label{lem:deter:conv:short}
     Suppose $T>0$   and let $\mu\in (0,\frac{d-1}{2}]$ for $d\geq 2$.
There exists a constant $K_{\rm dc} > 0$ that does not depend on $T$
so that the bound
\begin{equation}
    \sup_{0\leq t\leq T}| \mathcal J^{\rm sh}_G(t)|\leq K_{\rm dc}\left[\sup_{0\leq t\leq T}\int_0^t (1+t-s)^{-\mu} \|G(s)\|_{H^k_y}\,\mathrm ds\right]^{2}
\end{equation}
holds $\mathbb P$-almost surely.
\end{lemma}
\begin{proof}
Recalling the bilinear form \eqref{eq:fw:max:reg:def:j}, 
we reverse the order of integration
and exploit the symmetry of the inner product
 to find
\begin{align}
     \mathcal J_G^{\rm sh}(t)&\nonumber =\int_0^t  \int_{\max\{0,r-1\}}^{\min\{t,r+1\}}\int_{\max\{r,r'\}} 
        ^{\min\{t,r+1,r'+1\}}(1+t-s)^{-\mu}\\&\hspace{3.5cm}\times\langle \nabla_y S_H(s-r)G(r),\nabla_y S_H(s-r')G(r')\rangle_{H^{k}_y}\,\mathrm ds\,\mathrm dr'\,\mathrm dr\\
        &\nonumber=\int_0^t \int_{\max\{0,r-1\}}^{\min\{t,r+1\}} \mathcal B\big(t,\min\{t,r+1,r'+1\},\max\{r,r'\},\min\{r,r'\}\big)\,\mathrm dr'\,\mathrm dr.
\end{align}
On account of the decomposition \eqref{eq:Jtotal} and  crude bounds on the semigroup terms, we obtain
\begin{equation}
     \mathcal J_G^{\rm sh}(t)\leq C\mathcal H(t),\quad \mathcal H(t)=|\mathcal H_A(t)|+|\mathcal H_{B}(t)|+|\mathcal H_{C}(t)|,
\end{equation}
for some constant $C>0$, where 
\begin{align}
    \mathcal H_A(t)&=\int_0^t Z(r)\int_{\max\{0,r-1\}}^{\min\{t,r+1\}}(1+t-\min\{t,r+1,r'+1\})^{-\mu} Z(r')\,\mathrm dr'\,\mathrm dr,\\
     \mathcal H_B(t)&=\int_0^t Z(r)\int_{\max\{0,r-1\}}^{\min\{t,r+1\}}(1+t-\max\{r,r'\})^{-\mu} Z(r')\,\mathrm dr'\,\mathrm dr,\\
      \mathcal H_C(t)&=\int_0^t Z(r)\int_{\max\{0,r-1\}}^{\min\{t,r+1\}}Z(r')\int_{\max\{r,r'\}} 
        ^{\min\{t,r+1,r'+1\}}(1+t-s)^{-(\mu+1)} \,\mathrm ds\,\mathrm dr'\,\mathrm dr,
\end{align}
this time with $Z(t)=\|G(t)\|_{H^k_y}$.  Direct inspection of these expressions
yields the bound
\begin{equation}
\begin{aligned}
\mathcal H(t)&\leq 4\cdot 2^\mu\int_0^t(1+t-r)^{-\mu}Z(r)\,\int_{\max\{0,r-1\}}^{\min\{t,r+1\}}Z(r')\,\mathrm dr'\,\mathrm dr\\
   &\leq 4\cdot 4^\mu\int_0^t(1+t-r)^{-\mu}Z(r)\,\int_{\max\{0,r-1\}}^{\min\{t,r+1\}}(1+\min\{t,r+1\}-r')^{-\mu}Z(r')\,\mathrm dr'\,\mathrm dr\\
   &\leq 4^{\mu+1}\left[\sup_{0\leq s\leq t}\int_0^s(1+s-r)^{-\mu}Z(r)\,\mathrm dr\right]^2.
   \end{aligned}
\end{equation}
This proves the short-term estimate.
\end{proof}

\begin{proof}[Proof of Proposition \ref{prop:deter:conv}]
    The assertion follows by combining the results in Lemmas \ref{lem:deter:conv:long} and  \ref{lem:deter:conv:short}.
\end{proof}


\begin{proof}[Proof of Proposition \ref{prop:deter:conv:gap}]
We proceed by mimicking \eqref{eq:cnvd:split:j:g} and splitting
$\mathcal{I}_F$ into
a short-term and long-term component. The proof of Lemma \ref{lem:deter:conv:long}
can be followed for the long-term component, but now by invoking the second inequality in \eqref{eq:SL:bounds}, applying Lemma \ref{lem:z}, and noting that exponential decay can be bounded by algebraic decay of any rate; see \eqref{eq:absorb:decay}. The short-term component can be bounded using  the bilinear form  \eqref{eq:bil:form:V}, by means of computations that are  analogous to those in the proof of Lemma \ref{lem:deter:conv:short}.
\end{proof}

 \section{Nonlinear stability}\label{sec:stability}

As a result of our findings 
in {\S}\ref{sec:phase}--{\S}\ref{sec:evol:pert}, we  focus on the following generic framework. Throughout this section, we assume  there exist continuous processes $
    v:[0,\infty)\times\Omega\to H^k$ and $ \theta:[0,\infty)\times\Omega\to H^k_y$, for some $k\geq 0,$ for which we have the well-posed  mild representations 
\begin{equation}\begin{aligned}
v(t)&=S_L(t)v(0)+\int_0^tS_L(t-s)[\sigma^2F_{\rm bnd}(v(s),\theta(s),s)+F_{\rm nl}(v(s),\theta(s),s)]\,\mathrm ds\\&\quad\quad\quad+\sigma \int_0^t S_L(t-s)X
(v(s),\theta(s),s)\,\mathrm dW_s^Q,\label{eq:generic_v}
\end{aligned}
\end{equation}
together with
\begin{equation}\begin{aligned}
    \theta(t)&=S_H(t)\theta(0)+\int_0^tS_H(t-s)[\sigma^2G_{\rm bnd}(v(s),\theta(s),s)+G_{\rm nl}(v(s),\theta(s),s)]\,\mathrm ds
    \\&\quad\quad\quad+\sigma \int_0^t S_H(t-s)B(v(s),\theta(s),s)\,\mathrm dW_s^Q, \label{eq:generic_theta}
    \end{aligned}
\end{equation}
where both are $\mathbb P$-a.s satisfied for all time $0\leq t<\tau_\infty$. Whenever we have $\tau_\infty < \infty$,
we have the blow-up behaviour
\begin{equation}
\label{eq:st:blowup:crit}
\|v\|_{C([0,t];H^k)}+\|v\|_{L^2([0,t];H^{k+1})}+\|\theta\|_{C([0,t];H^k_y)}+\|\theta\|_{L^2([0,t];H^{k+1}_y)}\to\infty \hbox{ as } t\to\tau_\infty.
\end{equation}

The maps\footnote{Although now stated  for full generality, the explicit dependence on $\Omega$ in \eqref{63} and \eqref{64} is unnecessary and can be ignored for our current application. In case of the one dimensional \cite{hamster2019stability,hamster2020diag,hamster2020expstability,hamster2020} and  multidimensional setting on a cylindrical domain \cite{bosch2025conditionalspeedshapecorrections, bosch2024multidimensional}, if one rather works with noise that is not translation invariant, one can simply adjust the proof by making the nonlinearities explicitly depend on $[0,\infty)\times \Omega$ too|making the equations nonautonomous and random|as this allows us to incorporate the effect of the global phase $\gamma(t).$}
\begin{equation}
\begin{aligned}
    F_{\rm bnd}, F_{\rm nl}&:H^k\times H^{k+1}_y\times [0,\infty)\times \Omega\to H^k,\\ X&:H^k\times H^{k+1}_y\times [0,\infty)\times \Omega\to HS(\mathcal W_Q;H^k),\label{63}
    \end{aligned}
\end{equation}
and
\begin{equation}
\begin{aligned}
    G_{\rm bnd},G_{\rm nl}&:H^k\times H^{k+1}_y\times [0,\infty)\times \Omega\to H^k_y\cap L^1_y,\\ B&:H^k\times H^{k+1}_y\times [0,\infty)\times \Omega\to HS(\mathcal W_Q;H^k_y)\cap \Pi_2(\mathcal W_Q;L^1_y),\label{64}
    \end{aligned}
\end{equation}
are assumed to satisfy the bounds 
\begin{align}
    \|F_{\rm bnd}(v,\theta,\cdot)\|_{H^k}+ \|G_{\rm bnd}(v,\theta,\cdot)\|_{H^k_y}+\|G_{\rm bnd}(v,\theta,\cdot)\|_{L^1_y}&\leq K_{\rm bnd}
    \label{5.5},\\
    \|F_{\rm nl}(v,\theta,\cdot)\|_{H^k}+\|G_{\rm nl}(v,\theta,\cdot)\|_{H^k_y}+\|G_{\rm nl}(v,\theta,\cdot)\|_{L^1_y}&\leq K_{\rm nl}\textrm{RHS}(v,\theta)\label{5.6},\end{align}
    where
    \begin{equation}
        \textrm {RHS}(v,\theta)=\|v\|_{H^{k+1}}^2+\|v\|_{H^k}\|\theta\|_{H^k_y}+\|\nabla_y 
    \theta\|_{H^{k}_y}^2,\label{eq:RHS}
    \end{equation}
    together with
    \begin{align}
    \|X(v,\theta,\cdot)\|_{HS(\mathcal W_Q;H^k)}+ \|B(v,\theta,\cdot)\|_{HS(\mathcal W_Q;H^k_y)}+ \|B(v,\theta,\cdot)\|_{\Pi_2(\mathcal W_Q;L^1_y)}&\leq K_{\rm stc}
    ,\label{5.7}\end{align}
    whenever $\|v\|_{H^k},\|\theta\|_{H^{k}_y}\leq \epsilon_0$ for some $\epsilon_0>0$ fixed.
In addition, if $\|v\|_{H^k},\|\theta\|_{H^k_y}
\leq \eta_0^{1/2}\leq \epsilon_0$  for some sufficiently small $\eta_0>0$, then the orthogonal identities 
\begin{equation}
    \langle \sigma F_{\rm bnd}(v,\theta,\cdot)+F_{\rm nl}(v,\theta,\cdot),\psi_{\rm tw}\rangle_{L^2_x}=0\quad\text{and}\quad\langle B(v,\theta,\cdot)[\xi],\psi_{\rm tw}\rangle_{L^2_x}=0\label{eq:ort}
\end{equation}
are also satisfied, in a pointwise fashion on $\mathbb R^{d-1}$. 
Our main objective  is to control the size of
%
\begin{equation}
\begin{aligned}
  N_{\mu;k}(t)&=\|v(t)\|_{H^k}^2+\|\theta(t)\|_{H^k_y}^2+\int_0^t(1+t-s)^{-\mu}\mathrm{RHS}(v(s),\theta(s))\,\mathrm ds,\label{eq:size:final}
    \end{aligned}
\end{equation}
where the choice of $\mu>0$ depends on the dimension; see Table \ref{table:main:hfdst2}.
For any $0 < \eta < \eta_0$,
we define the stopping time 
\begin{equation}
    t_{\rm st}(\eta;k)=\inf\{t\geq 0:N_{\mu;k}(t)>\eta\},\label{eq:st}
\end{equation}
 associated to our exit problem, and observe that the  conditions in \eqref{eq:ort} are automatically satisfied 
 when 
 $t\leq t_{\rm st}(\eta;k)$.  In particular, the 
 blow-up properties \eqref{eq:st:blowup:crit} and our choice in \eqref{eq:size:final} imply that  $t_{\mathrm{st}}(\eta;k) < \tau_\infty$ whenever $\tau_\infty < \infty$; we refer to  \cite[Sec. 7]{bosch2024multidimensional} for a more detailed discussion on this matter. Finally, we observe that $\langle v(t,\cdot,y),\psi_{\rm tw}\rangle_{L^2_x}=0$ holds for all $y\in \mathbb R^{d-1}$ and $t\leq t_{\rm st}(\eta;k)$ if and only if $\langle v(0,\cdot,y),\psi_{\rm tw}\rangle_{L^2_x}=0$ for all $y\in\mathbb R^{d-1}$. It is worth mentioning that we are  not required to impose the latter orthogonality condition on the initial condition, since we do not exploit the decay of the initial perturbation in the proof of Proposition \ref{prop:E_stab}; see also Remark \ref{remark:ortho}.


\begin{proposition}\label{prop:E_stab}
 Consider the generic setting above. Pick two sufficiently small constants $\delta_\eta>0$ and $\delta_\sigma>0$. Then there exists a constant $K>0$ so that for any integer $T\geq 2,$ any $0<\eta<\delta_\eta,$ any $0\leq \sigma\leq \delta_\sigma,$ and any integer $p\geq 1$, we have the bound
 \begin{equation}\label{eq:better}
     \mathbb E\left[\sup_{0\leq t\leq t_{\rm st}(\eta;k)\wedge T}|N_{\mu;k}(t)|^p\right]\leq K^p\Big[\|v(0)\|_{H^{k}}^{2p}+\|\theta(0)\|_{H^{k}_y}^{2p}
     +\sigma^{2p}\mathfrak d_{\rm det}(T)^{2p}(p^p+\log(T)^p)\Big].
 \end{equation}
\end{proposition}

For our intents and purposes, it  suffices to  have $\|v\|_{H^k}^2$ on the right-hand side of \eqref{5.6}; see {\S}\ref{sec:intro}--{\S}\ref{sec:phase}.  Assuming the latter, it seems that one could simply replace the $H^{k+1}$-norm with the $H^{k}$-norm in \eqref{eq:size:final}. Completely neglecting this integral term makes the nonlinear stability analysis  work as well for $d>5$. On the other hand, one needs to keep the blow-up criterion in mind, which requires us to establish control of the maximal regularity setting.

In line with the discussion in {\S}\ref{sec:intro}, the term $\mathfrak d_{\rm det}(T)^{2p}$ in \eqref{eq:better} is caused by the integral term 
    \begin{equation}
        \int_0^t(1+t-s)^{-\mu}\|v(s)\|_{H^k}\|\theta(s)\|_{H^k_y}\,\mathrm ds
    \end{equation}
    in \eqref{eq:size:final}. Without the presence of the mixed term $\|v\|_{H^k}\|\theta\|_{H^k}$ in \eqref{5.6}|in particular, with regard to evolution \eqref{eq:generic_theta}|the growth term would have been $\mathfrak d_{\rm mix}(T)^{2p},$ where $\mathfrak d_{\rm mix}(T)=\mathfrak d_{\rm det}(T)\sqrt{\mathfrak d_{\rm stc}(T)}$. This would have led to significantly longer timescales.

\subsection{Estimates for the continuous norms}
Following the  earlier works \cite{bosch2024multidimensional,hamster2019stability, hamster2020expstability}, 
we proceed by introducing the notation
\begin{equation}
\mathcal{Z}_0(t)=S_L(t) P^\perp_{\rm tw} v(0),
\end{equation}
together with the integrals
\begin{equation}
\begin{aligned}
 \mathcal{Z}_{F ; \rm bnd}(t)&=\int_0^{t} S_L(t-s) P^\perp_{\rm tw} F_{\rm bnd}(v(s),\theta(s),s) 1_{s<t_{\rm st}(\eta;k)} \mathrm ds, \\
 \mathcal{Z}_{F ; \mathrm{nl}}(t)&=\int_0^{t} S_L(t-s) P^\perp_{\rm tw} F_{\mathrm{nl}}(v(s),\theta(s),s) 1_{s<t_{\rm st}(\eta;k)} \mathrm ds, \\
 \mathcal{Z}_{X}(t)&=\int_0^{t} S_L(t-s) P^\perp_{\rm tw} X(v(s),\theta(s),s) 1_{s<t_{\rm st}(\eta;k)} \mathrm d W_s^Q.
\end{aligned}
\end{equation}
The presence of $P^\perp_{\rm tw}=I-P_{\rm tw}$ in the above is simply to lay emphasis on  \eqref{eq:ort}. Similarly, we define
\begin{equation}
\mathcal{E}_0(t)=S_H(t) \theta(0),
\end{equation}
together with the integrals
\begin{equation}
\begin{aligned}
 \mathcal{E}_{G ; \rm bnd}(t)&=\int_0^{t} S_H(t-s)   G_{\rm bnd}(v(s),\theta(s),s) 1_{s<t_{\rm st}(\eta;k)} \mathrm ds, \\
 \mathcal{E}_{G ; \mathrm{nl}}(t)&=\int_0^{t} S_H(t-s)  G_{\mathrm{nl}}(v(s),\theta(s),s) 1_{s<t_{\rm st}(\eta;k)} \mathrm ds, \\
 \mathcal{E}_{B}(t)&=\int_{0}^{t} S_H(t-s)  B(v(s),\theta(s),s) 1_{s<t_{\rm st}(\eta;k)} \mathrm d W_s^Q.
\end{aligned}
\end{equation}
Using these expressions, we obtain the estimates
\begin{equation}\begin{aligned}
    \label{eq:st:a:priori:bnd:e}
        &\mathbb E\sup_{0\leq t\leq t_{\rm st}(\eta;k)\wedge T} \|v(t)\|_{H^k}^{2p}\\&\quad\leq 
        4^{2p}\mathbb E\sup_{0\leq t\leq T}\left[\|\mathcal Z_0(t)\|_{H^k}^{2p}+\sigma^{4p}\|\mathcal Z_{F;\rm bnd}(t)\|_{H^k}^{2p}+\|\mathcal Z_{F; \rm nl}(t)\|_{H^k}^{2p}+\sigma^{2p}\|\mathcal Z_X(t)\|_{H^k}^{2p}\right]
        \end{aligned}
    \end{equation}
    and
\begin{equation}\begin{aligned}
    \label{eq:st:a:priori:bnd:h}
        &\mathbb E\sup_{0\leq t\leq t_{\rm st}(\eta;k)\wedge T}\|\theta(t)\|_{H^k_y}^{2p}\\&\quad\leq 4^{2p}\mathbb E\sup_{0\leq t\leq T}\left[\|\mathcal{E}_0(t)\|_{H^k_y}^{2p}+\sigma^{4p}\|\mathcal{E}_{G;\rm bnd}(t)\|_{H^k_y}^{2p}+\|\mathcal{E}_{G; \rm nl}(t)\|_{H^k_y}^{2p}+\sigma^{2p}\|\mathcal{E}_B(t)\|_{H^k_y}^{2p}\right],
        \end{aligned}
    \end{equation}
    where all individual term can be dealt with as follows.

\begin{lemma}\label{lem:E0ElinEnl}
    Consider the generic setting above. Then there exists a constant $K>0$ such that for any $0<\eta<\eta_0$, any $T>0$, and any $p \ge 1$, we have the pathwise bounds
\begin{equation}\begin{aligned}
 \sup_{0\leq t\leq T}\big[\|\mathcal Z_0(t)\|_{H^k}^{2p}+\|\mathcal{E}_0(t)\|_{H^k_y}^{2p}\big]&\leq K^{p}\big[\|v(0)\|_{H^k}^{2p}+\|\theta(0)\|_{H^k_y}^{2p}\big],\\
         \sup_{0\leq t\leq T}\big[\|\mathcal Z_{F;\rm nl}(t)\|_{H^k}^{2p}+\|\mathcal{E}_{G;\rm nl}(t)\|_{H^k_y}^{2p}\big]&\leq \eta^p K^{p}\sup_{0\leq t\leq t_{\rm st}(\eta;k)\wedge T} N_{\mu;k}(t)^p,
    \label{eq:sup}\end{aligned}\end{equation}
    together with
    \begin{equation}
        \sup_{0\leq t\leq T} \|\mathcal Z_{F;\rm bnd}(t)\|_{H^k}^{2p}\leq K^{p}\quad\text{and}\quad\sup_{0\leq t\leq T}\|\mathcal{E}_{G;\rm bnd}(t)\|_{H^k_y}^{2p} \leq K^{p}\mathfrak d_{\rm det}(T)^{2p}.\label{eq:bnd:FG}
    \end{equation}
\end{lemma}
\begin{proof}
    The first inequality in \eqref{eq:sup} is obvious. The results for $\mathcal Z_{F;\#}(t)$, where $\#\in\{\rm bnd,nl\}$,  
    follow directly from 
    straightforward norm estimates; see also  \cite[Lem. 5.3]{hamster2020expstability}. 
The algebraic decay estimate in Corollary \ref{cor:algebraic:decay} yields
\begin{equation}
\begin{aligned}
\|\mathcal{E}_{G;\rm bnd}(t)\|_{H^k_y}
&\leq M_{\rm alg} K_{\rm bnd} \int_{0}^t(1+t-s)^{-\frac{d-1}4}\mathrm ds
\leq M_{\rm alg}K_{\rm bnd}\mathfrak d_{\rm det}(T),\label{eq:expres1}
\end{aligned}
\end{equation}
for any $0\leq t\leq T.$ This proves the second inequality in \eqref{eq:bnd:FG}. In a similar fashion, we compute
\begin{equation}
\begin{aligned}
\|\mathcal{E}_{G;\rm nl}(t)\|_{H^k_y}^2
&\leq M_{\rm alg}^2K_{\rm nl}^2\left(\int_{0}^t(1+t-s)^{-\frac{d-1}4}\mathrm {RHS}(v(s),\theta(s)) 1_{s < t_{\rm st}(\eta;k)}\,\mathrm ds \right)^2\\
&
\leq  M_{\rm alg}^2 K_{\rm nl}^2\eta\sup_{0\leq t\leq t_{\rm st}(\eta;k)\wedge T}N_{\mu;k}(t),\label{eq:expres2}
\end{aligned}
\end{equation}
for any $0\leq t\leq T,$
proving the final inequality in \eqref{eq:sup}.
\end{proof}
\begin{lemma}\label{lem:EJ}
Consider the generic setting above. There exist constants $K_X, K_B > 0$ so that for any $0<\eta<\eta_0$, any integer $T \ge 2$, and any integer $p \ge 1$, we have the bound 
%
    \begin{equation}
        \mathbb E\sup_{0\leq t\leq T}   \|\mathcal Z_X(t)\|_{H^k}^{2p} 
        \leq 
        K_X^p (p^p+\log(T)^p),\label{eq:ZX}
    \end{equation}
 together with 
    \begin{equation}
        \mathbb E\sup_{0\leq t\leq T}   \|\mathcal E_B(t)\|_{H^k_y}^{2p}
        \leq 
        K_B^p
      \mathfrak d_{\rm stc}(T)^p(p^p+\log(T)^p).\label{eq:EB}
    \end{equation}
\end{lemma}
\begin{proof} The bound \eqref{eq:EB}  follows from 
Corollary \ref{cor:E_B}, while  \eqref{eq:ZX} follows from Proposition \ref{prp:cnv:bnd:x:unif}.
\end{proof}
\begin{corollary}\label{cor:v:theta}
Consider the generic setting above. There exists a constant $K > 0$ so that for any $0<\eta<\eta_0$, any integer $T\geq 2$, and any integer $p \ge 1$, we have the  bound
\begin{equation}\begin{aligned}
        &\mathbb E\sup_{0\leq t\leq t_{\rm st}(\eta;k)\wedge T}\big[\|v(t)\|_{H^k}^{2p}+\|\theta(t)\|_{H^k_y}^{2p}\big]\leq K^p\Bigg(\|v(0)\|_{H^k}^{2p}+\|\theta(0)\|_{H^k_y}^{2p}\\&\qquad\qquad+\sigma^{4p}\mathfrak d_{\rm det}(T)^{2p}+\sigma^{2p}\mathfrak d_{\rm stc}(T)^{p}\mathfrak (p^p+\log(T)^p)
        +\eta^p\mathbb E\left[\sup_{0\leq t\leq t_{\rm st}(\eta;k)\wedge T}N_{\mu;k}(t)^p\right]\Bigg).\label{eq:cor:v:theta}
        \end{aligned}
    \end{equation}
\end{corollary}
    \subsection{Estimates for  the weighted integrals with higher regularity}\label{sec:est:weighted:Sobolev}
    We  proceed with a method that was initially proposed by \cite{hamster2019stability}. This method turned out to be fruitful\linebreak in \cite{bosch2024multidimensional}, showing long-time persistence of planar waves evolving over a cylindrical domain, while it also gives more optimal results compared to the mild It\^o-formula approach in  \cite[Sec. 5.2]{hamster2020expstability}.
    
In order  to deal with the integrated $H^{k+1}$-norm of  $v(t)$ in \eqref{eq:size:final}, we introduce the integrals
\begin{equation}
\begin{aligned}
\label{eq:st:a:priori:bnd:i}
\mathcal{I}_{  0}(t)&=\int_0^t (1+t-s)^{-\mu}\| \mathcal{Z}_0(s)\|_{H^{k+1}}^2 \mathrm ds,\\
 \mathcal{I}_{ F ; \mathrm{bnd}}(t)&=\int_0^t (1+t-s)^{-\mu}\| \mathcal{Z}_{F ; \mathrm{bnd}}(s)\|_{H^{k+1}}^2 \mathrm ds, \\
 \mathcal{I}_{ F; \mathrm{nl}}(t)&=\int_0^t (1+t-s)^{-\mu}\|\mathcal{Z}_{F ; \mathrm{nl}}(s)\|_{H^{k+1}}^2 \mathrm ds, \\
   \mathcal{I}_{ X}(t)&=\int_0^t (1+t-s)^{-\mu}\| \mathcal{Z}_{X}(s)\|_{H^{k+1}}^2 \mathrm ds.
\end{aligned}
\end{equation}
Turning to the integrated $H^{k}_y$-norm of  $\nabla_y\theta(t)$, we define 
\begin{equation}
\begin{aligned}
\label{eq:st:a:priori:bnd:j}
\mathcal{J}_{  0}(t)&=\int_0^t (1+t- s)^{-\mu}\| \nabla_y\mathcal{E}_0(s)\|_{H^{k}_y}^2 \mathrm ds,\\
 \mathcal{J}_{ G ; \mathrm{bnd}}(t)&=\int_0^t (1+t- s)^{-\mu}\| \nabla_y\mathcal{E}_{G ; \mathrm{bnd}}(s)\|_{H^{k}_y}^2 \mathrm ds, \\
 \mathcal{J} _{ G ; \mathrm{nl}}(t)&=\int_0^t (1+t- s)^{-\mu}\|\nabla_y\mathcal{E}_{G ; \mathrm{nl}}(s)\|_{H^{k}_y}^2 \mathrm ds, \\
   \mathcal{J}_{ B}(t)&=\int_0^t (1+t- s)^{-\mu}\| \nabla_y\mathcal{E}_{B}(s)\|_{H^{k}_y}^2 \mathrm ds.
\end{aligned}
\end{equation}
This leads directly to the estimates
  \begin{equation}
        \begin{aligned}
            &\mathbb E \sup_{0\leq t\leq t_{\rm st}(\eta;k)\wedge T}\left[\int_0^t (1+t-s)^{-\mu}\|v(s)\|_{H^{k+1}}^2\mathrm ds\right]^p \\&\qquad\quad\leq 4^{2p}\mathbb E \sup_{0\leq t\leq T}\Big[\mathcal I_{0}(t)^p+\sigma^{4p}\mathcal I_{F;\rm bnd}(t)^p
+ \mathcal I_{F;\rm nl}(t)^p+\sigma^{2p}\mathcal I_{X}(t)^p\Big]\label{eq:sup:v}
    \end{aligned}
    \end{equation}
and 
  \begin{equation}
        \begin{aligned}\label{eq:sup:theta}
            &\mathbb E \sup_{0\leq t\leq t_{\rm st}(\eta;k)\wedge T}\left[\int_0^t (1+t-s)^{-\mu}\|\nabla_y\theta(s)\|_{H^{k}_y}^2\mathrm ds\right]^p \\&\qquad\quad\leq 4^{2p}\mathbb E \sup_{0\leq t\leq T}\Big[\mathcal J_{0}(t)^p+\sigma^{4p}\mathcal J_{F;\rm bnd}(t)^p
+ \mathcal J_{F;\rm nl}(t)^p+\sigma^{2p}\mathcal J_{B}(t)^p\Big].\\
        \end{aligned}
        \end{equation}

\begin{lemma}\label{lem:maximalregH1}
Consider the generic setting above. Then there exists a constant $K>0$ such that for any $0<\eta<\eta_0$, any $T>0$, and any $p \ge 1$, we have the pathwise bounds
    \begin{equation}\begin{aligned}
        \sup_{0\leq t\leq T}\big[\mathcal I_{0}(t)^p +\mathcal J_{0}(t)^p \big] &\leq 
        K^{p} \big[\|v(0)\|_{H^k}^{2p}+\|\theta(0)\|_{H^k_y}^{2p}\big],\\
        \sup_{0\leq t\leq T}  \big[\mathcal  I_{F;\rm bnd}(t)^p+\mathcal J_{G;\rm bnd}(t)^p\big]&\leq
    K^{p}\mathfrak d_{\rm det}(T)^{2p},\\
          \sup_{0\leq t\leq T}   \big[\mathcal I_{F;\rm nl}(t)^p+\mathcal J_{G;\rm nl}(t)^p\big]&\leq \eta^{p}
  K^{p}\sup_{0\leq t\leq t_{\rm st}(\eta;k)\wedge T} N_{\mu;k}(t)^p\label{eq:IJ:estimates}.\end{aligned}\end{equation}
\end{lemma}
\begin{proof}
Recalling the bilinear form in \eqref{eq:fw:max:reg:def:j}, we  observe that
$\mathcal J_0(t)\leq \mathcal B(t,t,0,0)[\theta(0),\theta(0)]$ holds. The first inequality in \eqref{eq:IJ:estimates} is inferred from exploiting the decomposition  \eqref{eq:Jtotal}, where we bound both expressions in a crude way. 
The estimates for $\mathcal J_{G;\rm bnd}$
and $\mathcal J_{G;\rm nl}$ follow from Proposition \ref{prop:deter:conv}.
In a similar fashion,
the bilinear form \eqref{eq:bil:form:V} can be exploited to derive the first estimate in \eqref{eq:IJ:estimates} for $\mathcal I_0(t)$, while we appeal to Proposition \ref{prop:deter:conv:gap} for the remaining two estimates.
\end{proof}
 \begin{lemma} Consider the generic setting above. There exists a constant $K > 0$ so that for any $0<\eta<\eta_0$, any integer $T\geq 2$, and any integer $p \ge 1$, we have the  bound
    \begin{equation}
        \mathbb E\sup_{0\leq t\leq T}\big[\mathcal I_X(t)^p+\mathcal J_B(t)^p\big]
        \leq 
        K^p
      \mathfrak d_{\rm det}(T)^{p}(p^p+\log(T)^p).\label{eq:EBJB}
    \end{equation}
\end{lemma}
\begin{proof} The bound follows from Corollary \ref{cor:E_B} in combination with Proposition \ref{prp:cnv:bnd:x:unif}.
\end{proof}
\begin{corollary}\label{cor:int:v:theta}
Consider the generic setting above. There exists a constant $K > 0$ so that for any $0<\eta<\eta_0$, any integer $T\geq 2$, and any integer $p \ge 1$, we have the  bound
\begin{align}
\nonumber&\mathbb E\sup_{0\leq t\leq t_{\rm st}(\eta;k)\wedge T}\left[\int_0^t(1+t-s)^{-\mu}\big[\|v(s)\|_{H^{k+1}}^{2}+\|\nabla_y\theta(s)\|_{H^k_y}^{2}\big]\mathrm ds\right]^p\leq K^p\Bigg(\|v(0)\|_{H^k}^{2p}+\|\theta(0)\|_{H^k_y}^{2p}\\&\qquad\qquad+\sigma^{4p}\mathfrak d_{\rm det}(T)^{2p}+\sigma^{2p}\mathfrak d_{\rm det}(T)^{p}\mathfrak (p^p+\log(T)^p)
        +\eta^p\mathbb E\left[\sup_{0\leq t\leq t_{\rm st}(\eta;k)\wedge T}N_{\mu;k}(t)^p\right]\Bigg).\label{eq:cor:int:theta}
        \end{align}
\end{corollary}

\subsection{Estimates for  the weighted integrals of mixed type}\label{sec:est:mixed:type}

Finally, we need to address the mixed term $\|v\|_{H^k}\|\theta\|_{H^k}$ of the nonlinearities in \eqref{eq:RHS}.  
Our strategy will be to pull the supremum of
$\|\vartheta\|_{H^k}$ outside of the integral \eqref{eq:size:final}.
To this end, we introduce the notation
\begin{equation}
    \mathcal{E}_{\rm max}(T;\sigma) = \sigma^2 \mathfrak{d}_{\rm det}(T) + \| \vartheta(0) \|_{H^k_y} + \sup_{0 \le t \le t_{\rm st}(\eta;k)\wedge T} N_{\mu;k}(t) + \sigma \sup_{0 \le t \le T} \| \mathcal{E}_B(t) \|_{H^k_y}.
\end{equation}
Note that the representation
\eqref{eq:generic_theta},
together with the estimates
\eqref{eq:expres1}--\eqref{eq:expres2}
and \eqref{eq:inequalities:heatsemigroup},
implies the pathwise bound
\begin{equation}
\sup_{0 \le t \le t_{\rm st}(\eta;k) \wedge T}
\|\vartheta(t)\|_{H^k_y} \le K \mathcal{E}_{\rm max}(T;\sigma),
\end{equation}
for some $K > 0$.
We now proceed by introducing the integrals
\begin{equation}
\begin{aligned}
\label{eq:st:M:simpler}
\mathcal{M}_{\#}(t)&=\int_0^t (1+t-s)^{-\mu}\| \mathcal{Z}_{\#}(s)\|_{H^{k}} \mathrm ds,
\end{aligned}
\end{equation}
for $\#\in \{0,F;\textrm{bnd},F;\textrm{nl},X\} $.
Upon recalling the representation
\eqref{eq:generic_v}
and applying the elementary inequality $(ab)^p \le a^{2p}+b^{2p}$,
we readily arrive at the bound
   \begin{align}
            \nonumber &\mathbb E \sup_{0\leq t\leq t_{\rm st}(\eta;k)\wedge T}\left[\int_0^t (1+t-s)^{-\mu}\|v(s)\|_{H^{k}}\|\theta(s)\|_{H^k_y}\mathrm ds\right]^p  \\  &\qquad  \leq K^p  \mathbb E \,\mathcal{E}_{\rm max}(T;\sigma)^p \sup_{0\leq t\leq T}\Big[\mathcal M_{0}(t)^p+\sigma^{2p}\mathcal M_{F;\rm bnd}(t)^p
+ \mathcal M_{F;\rm nl}(t)^p+\sigma^{p}\mathcal M_{X}(t)^p
\Big]\label{eq:sup:mixed:short}\\
& \qquad  \leq  4K^{2p} \mathbb E\, \mathcal{E}_{\rm max}(T;\sigma)^{2p}+\mathbb E \sup_{0\leq t\leq T}\Big[\mathcal M_{0}(t)^{2p}+\sigma^{4p}\mathcal M_{F;\rm bnd}(t)^{2p}
+ \mathcal M_{F;\rm nl}(t)^{2p}+\sigma^{2p}\mathcal M_{X}(t)^{2p}
\Big].\nonumber
\end{align}
The relevant supremum estimates on the integrals
\eqref{eq:st:M:simpler} are provided in our next result.

\begin{lemma}\label{lem:mix}
Consider the generic setting above. Then there exists a constant $K>0$ such that for any $0<\eta<\eta_0$, any $T>0$, and any $p \ge 1$, we have the pathwise bounds
\begin{equation}\begin{aligned}
        \sup_{0\leq t\leq T} \mathcal M_{0}(t)^{2p}   &\leq 
        K^p\|v(0)\|_{H^k}^{2p}, \\
        \sup_{0\leq t\leq T}   \mathcal  M_{F;\rm bnd}(t)^{2p}&\leq
    K^p \mathfrak d_{\rm det}(T)^{2p},\\
          \sup_{0\leq t\leq T}    \mathcal M_{F;\rm nl}(t)^{2p}&\leq
          K^p \eta^p \sup_{0\leq t\leq t_{\rm st}(\eta;k)\wedge T} N_{\mu;k}(t)^p,\\
  \sup_{0\leq t\leq T}   \mathcal M_{X}(t)^{2p}&\leq K^p \mathfrak d_{\rm det}(T)^{2p} \sup_{0\leq s\leq T}\|\mathcal Z_X(s)\|_{H^k}^{2p}.
\end{aligned}\end{equation}
\end{lemma}

\begin{proof}
    Exploiting the exponential decay of $S_L(t)P^{\perp}_{\rm tw}$, see Lemma \ref{lem:decay:SL}, yields
\begin{equation}
        \mathcal M_{0}(t)\leq M\|v(0)\|_{H^k} \int_0^t(1+t-s)^{-\mu}e^{-\beta s}\mathrm ds\leq MK_{\mu,\beta}\|v(0)\|_{H^k},
    \end{equation}
    for some $K_{\mu,\beta}>0.$ For $\# \in \{F; \mathrm{bnd}, X\}$, we obtain
    \begin{equation}
    \begin{aligned}
        \mathcal M_{\#}(t)&\leq \sup_{0\leq t\leq T}\|\mathcal Z_{\#}(t)\|_{H^k}\int_0^t(1+t-s)^{-\mu}\mathrm ds\leq K' \mathfrak d_{\rm det}(T)\sup_{0\leq t\leq T}\|\mathcal Z_{\#}(t)\|_{H^k},
        \end{aligned}
    \end{equation}
    for some other  constant $K'>0$. Finally, reversing the order of integration yields
    \begin{equation}
        \begin{aligned}
            \mathcal M_{F;\rm nl}(t)&\leq K_{\rm nl}\int_0^t(1+t-s)^{-\mu}\int_0^s e^{-\beta (s-r)}
        \mathrm {RHS}\big(v(r),\theta(r)\big) 1_{r < t_{\rm st}(\eta;k)}\,\mathrm dr\,\mathrm ds\\
        &\leq K_{\rm nl}\int_0^t\mathrm {RHS}\big(v(r),\theta(r)\big) 1_{r < t_{\rm st}(\eta;k)} \int_r^t (1+t-s)^{-\mu} e^{-\beta (s-r)}\,\mathrm ds\,\mathrm dr \\
        &\leq K_{\rm nl}K_{\mu,\beta}\int_0^t(1+t-r)^{-\mu} \mathrm {RHS}\big(v(r),\theta(r)\big) 1_{r < t_{\rm st}(\eta;k)}\,\mathrm dr\\
        &\leq K_{\rm nl}K_{\mu,\beta}\sup_{0\leq t\leq t_{\rm st}(\eta;k)\wedge T} N_{\mu;k}(t).
        \end{aligned}
    \end{equation}
    The stated bounds above now all follow 
    after recalling the uniform bound \eqref{eq:bnd:FG} for the supremum of $\mathcal{Z}_{F;\rm bnd}$ and upon noting that $N_{\mu;k}(t)\leq \eta$ holds for $0\leq t\leq  t_{\rm st}(\eta;k)\wedge T. $ 
\end{proof}
\begin{corollary}\label{cor:int:mix}
Consider the generic setting above. There exists a constant $K > 0$ so that for any $0<\eta<\eta_0$, any integer $T\geq 2$, and any integer $p \ge 1$, we have the  bound
\begin{equation}
\begin{aligned}
&\mathbb E\sup_{0\leq t\leq t_{\rm st}(\eta;k)\wedge T}\left[\int_0^t(1+t-s)^{-\mu} \|v(s)\|_{H^{k}}\|\theta(s)\|_{H^k_y}\mathrm ds\right]^p\leq K^p\Bigg(\|v(0)\|_{H^k}^{2p}+\|\theta(0)\|_{H^k_y}^{2p}\\&\qquad+\sigma^{4p}\mathfrak d_{\rm det}(T)^{2p}+\sigma^{2p}\mathfrak d_{\rm det}(T)^{2p}\mathfrak (p^p+\log(T)^p)
        +\eta^p\mathbb E\left[\sup_{0\leq t\leq t_{\rm st}(\eta;k)\wedge T}N_{\mu;k}(t)^p\right]\Bigg).\label{eq:cor:mix}
        \end{aligned}
        \end{equation}
\end{corollary}
\begin{proof}
    Inspecting \eqref{eq:sup:mixed:short},
    these bounds follow from Lemma \ref{lem:mix} together with
    \eqref{eq:ZX} and \eqref{eq:EB}.
\end{proof}

The perspicacious reader might wonder whether the bound \eqref{eq:cor:mix} is too crude. A more refined analysis could proceed by analysing the
16 mixed terms \begin{equation}
\begin{aligned}
\mathscr{M}_{\#_1;\#_2}(t)&=\int_0^t (1+t-s)^{-\mu}\| \mathcal{Z}_{\#_1}(s)\|_{H^{k}}\| \mathcal{E}_{\#_2}(s)\|_{H^{k}_y} \mathrm ds,
\end{aligned}
\end{equation}
with $\#_1\in \{0,F;\textrm{bnd},F;\textrm{nl},X\} $ and $\#_2\in \{0,G;\textrm{bnd},G;\textrm{nl},B\}. $
For example,
by taking $p=1$ for simplicity, 
we obtain the more efficient bound
\begin{equation}
\begin{array}{lcl}
\sigma^2 \mathbb E \sup_{0\le t \le T} \mathscr{M}_{X;B}(t)
& \le & \sigma^2 K \mathfrak{d}_{\rm det}(T)
\mathbb E\big[
\sup_{0 \le t \le T} \| \mathcal{E}_B(t) \|_{H^k_y} \sup_{0 \le t \le T} \|\mathcal{Z}_X(t)\|_{H^k}\big]
\\[0.2cm]
& \le & 
\sigma^2 K\mathfrak{d}_{\rm det}(T)
\big( 
\mathfrak{d}_{\rm stc}(T)^{-1/2} \mathbb E
\sup_{0 \le t \le T} \| \mathcal{E}_B(t)\|_{H^k_y}^2 
\\[0.2cm]
& & \qquad \qquad 
+\, \mathfrak{d}_{\rm stc}(T)^{1/2} \mathbb E
\sup_{0 \le t \le T} \| \mathcal{Z}_X(t)\|_{H^k}^2 
\big) 
\\[0.2cm]
& \le & K \sigma^2 \mathfrak{d}_{\rm det}(T) \mathfrak{d}_{\rm stc}(T)^{1/2} \log(T).
\end{array}
\end{equation}
In addition, assuming the extra condition $\vartheta(0) \in L^1_y$, 
we obtain the pathwise bound
\begin{equation}
\begin{array}{lcl}
     \mathscr{M}_{X;0}(t)
    & \le & \sigma M_{\rm alg}
\big[\|\theta(0)\|_{H^k_y}+\|\theta(0)\|_{L^1_y}\big] \int_0^t(1+t-s)^{-\mu}(1+s)^{-\frac{d-1}4}\|\mathcal Z_{X}(s)\|_{H^k}\mathrm ds
\\[0.2cm]
& \le & 
\sigma M_{\rm alg}
\big[\|\theta(0)\|_{H^k_y}+\|\theta(0)\|_{L^1_y}\big]  \sup_{0 \le s \le T} \|\mathcal Z_{X}(s)\|_{H^k}
\\[0.2cm]
& & \qquad \qquad \times \int_0^t(1+t-s)^{-\mu}(1+s)^{-\frac{d-1}4} \,\mathrm ds  
\\[0.2cm]
& \le & 
\sigma K
\big[\|\theta(0)\|_{H^k_y}+\|\theta(0)\|_{L^1_y}\big]  \sup_{0 \le s \le T} \|\mathcal Z_{X}(s)\|_{H^k}
\big(1 + T^{1/2} 1_{d = 2} \big),
\end{array}
\end{equation}
in which we have used Corollary \ref{cor:algebraic:decay} and Lemma \ref{lem:z}.
This leads to the more efficient estimate
\begin{equation}
\begin{array}{lcl}
     \sigma \mathbb E \sup_{0 \le t \le T} \mathscr{M}_{X;0}(t)
     &\le& K
     \big[ \|\vartheta(0)\|_{H^k_y}^2 + \| \vartheta(0)\|_{L^1_y}^2 + 
     \sigma^2 \log(T) \big(1 + \mathfrak{d}_{\rm stc}(T)^2 1_{d=2} \big)
     \big]\\
      &\leq& K
     \big[ \|\vartheta(0)\|_{H^k_y}^2 + \| \vartheta(0)\|_{L^1_y}^2 + 
     \sigma^2 \mathfrak d_{\rm det}(T)\mathfrak d_{\rm stc}(T)^{1/2}\log(T) 
     \big].
\end{array}
\end{equation}
However, it seems that the term $\mathscr{M}_{X; G;\rm nl}$
admits no comparable path to improvement,
because we do not want $\sup_{0\leq t\leq t_{\rm st}(\eta;k)\wedge T} N_{\mu;k}(t)$ to be multiplied by $T$-dependent factors. Differently put, the growth term $\sigma^{2p}\mathfrak d_{\rm det}(T)^{2p}$ in \eqref{eq:cor:mix} cannot be relaxed; which is in line with our discussion in {\S}\ref{sec:intro}.

\begin{proof}[Proof of Proposition \ref{prop:E_stab}]
Collecting the results in Corollaries \ref{cor:v:theta}, \ref{cor:int:v:theta}, and \ref{cor:int:mix}, we conclude
\begin{equation}\begin{aligned}
            &\mathbb E\left[\sup_{0\leq t\leq t_{\rm st}(\eta;k)\wedge T}N_{\mu;k}(t)^p\right]\leq K^p\Bigg(\|v(0)\|_{H^{k}}^{2p}+ \|\theta(0)\|_{H^{k}_y}^{2p}+\sigma^{4p} \mathfrak d_{\rm det}(T)^{2p}
           \\&\qquad\qquad\qquad\quad\quad+\sigma^{2p}  \mathfrak d_{\rm det}(T)^{2p}(p^p+\log(T)^p)
            + \eta^p\mathbb E\left[\sup_{0\leq t\leq t_{\rm st}(\eta;k)\wedge T}N_{\mu;k}(t)^p\right]\Bigg).
            \end{aligned}
        \end{equation}
        We now obtain the assertion by
        restricting the size of ${\eta}$ and subsequently absorbing $\sigma^{4p} \mathfrak d_{\rm det}(T)^{2p}$ into the other explicit $\sigma^2$-dependent term.
\end{proof}
 \begin{proof}
     [Proof of Theorem \ref{thm:general}]
     The first part of the assertion is given by Proposition \ref{prop:E_stab}, noting that the generic setting is indeed satisfied according to {\S}\ref{sec:phase}--{\S}\ref{sec:evol:pert}.
    Now, introduce the random variable
\begin{equation}
    Z_T=\sup_{0\leq t\leq t_{\rm st(\eta;k)}\wedge T}|N_{\mu;k}(t)|
\end{equation}
and observe that $
    \mathbb P(t_{\rm st}(\eta;k)< T)=\mathbb P(Z_T\geq \eta) $ holds. By exploiting the exponential Markov-type inequality in \cite[Lem. B.1]{bosch2024multidimensional}
    with 
\begin{equation}\Theta_1=K\sigma^2\mathfrak d_{\rm det}(T)^2,\quad \Theta_2=Ku_0^2+K\sigma^2\log(T)\mathfrak d_{\rm det}(T)^2,\end{equation}
where $u_0^2=\|v(0)\|_{H^k}^2+\|\theta(0)\|_{H^k_y}^2,$
we obtain
   the probability tail estimate
    \begin{equation}
        \mathbb P(Z_T\geq \eta)\leq 3T^{1/2e}\exp\left(-\frac{\eta-2eKu_0^2}{2eK\sigma^2\mathfrak d_{\rm det}(T)^2}\right).
    \end{equation}
The bound \eqref{eq:25} now readily follows upon choosing $\kappa \leq (4eK)^{-1} $ and by noting that $3T^{1/2e}\leq 2T$ holds for $T\ge 2$.
\end{proof}

\appendix

\section{Preliminary estimates}\label{sec:prelim}
In this part we collect several  essential preliminary estimates required to establish the metastability of  planar travelling waves in our noisy framework.  We   start in {\S}\ref{sec:travel:heat} by recalling well-established facts concerning the linearisation about the travelling wave (in one dimension) together with  properties of the heat semigroup (in any spatial dimension). 
In {\S}\ref{sec:stoch:convolution} we  provide fundamental maximal inequalities for stochastic convolutions that underpin our stability estimates.

\subsection{Semigroup bounds}\label{sec:travel:heat}\label{lem:decay:SL}
Let us denote by $S_{\rm tw}(t):H^k_x\to H^k_x$ 
 the analytic semigroup generated by  $\mathcal L_{\rm tw}$. Recall that $P_{\rm tw}:H_x^k\to H_x^k$ is the spectral projection on the neutral eigenvalue of $\mathcal L_{\rm tw}$ and we write $P^\perp _{\rm tw}=I-P_{\rm tw}.$
\begin{lemma}Pick a sufficiently large $M \geq 1$. Then for every $t>0$ we have the bounds
\begin{equation}\label{eq:Stw:bounds}
\begin{array}{ll}
\|S_{\rm tw}(t) P^\perp_{\rm tw}\|_{\mathscr{L}\left(H^k_x, H^k_x\right)} & \leq M e^{-\beta t}, \\
\|S_{\rm tw}(t) P^\perp_{\rm tw}\|_{\mathscr{L}\left(H^k_x, H^{k+1}_x\right)} & \leq M \max\{t^{-\frac{1}{2}},1\} e^{-\beta t}.
\end{array}
\end{equation}
\end{lemma}
\begin{proof}
These bounds can be deduced from the general theory in \cite[Prop. 5.2.1]{lunardi2004linear}. Alternatively, we refer to \cite[Lem. 2.1]{hamster2020expstability} for the case $k=0$ and to the comments in \cite[Sec. 3]{bosch2024multidimensional}.
\end{proof}
 Further, we write $S_L(t):H^k\to H^k$ for the semigroup generated by  $\mathcal L=\mathcal L_{\rm tw}+\Delta_y$ together with $S_H(t):H^k_y\to H^k_y$ which is the analytic semigroup generated by $\Delta_y.$ We can view both $S_{\rm tw}(t)$ and $S_H(t)$ as semigroups 
 on $H^k$ by letting them act as
 \begin{equation}
     [S_{\rm tw}(t)v](x,y)=[S_{\rm tw}(t)v(\cdot,y)](x),\qquad [S_{H}(t)v](x,y)=[S_{H}(t)v(x,\cdot)](y).
 \end{equation}
  This directly implies the commutation relation $S_L(t)=S_{\rm tw}(t)S_H(t)=S_H(t)S_{\rm tw}(t).$ 
Finally, throughout this paper, our convention is to write again $P_{\rm tw}^\perp =I-P_{\rm tw}:H^k\to H^k$ for the complement of the (spectral)  projection that is trivially extended to the coordinate $y$. 
\begin{lemma}Pick a sufficiently large $M \geq 1$. Then for every $t>0$ we have the bounds
\begin{equation}\label{eq:SL:bounds}
\begin{array}{ll}
\|S_L(t) P^\perp_{\rm tw}\|_{\mathscr{L}\left(H^k, H^k\right)} & \leq M e^{-\beta t}, \\
\|S_L(t) P^\perp_{\rm tw}\|_{\mathscr{L}\left(H^k, H^{k+1}\right)} & \leq M \max\{t^{-\frac{1}{2}},1\} e^{-\beta t}.
\end{array}
\end{equation}
\end{lemma}
\begin{proof}
The estimates in \eqref{eq:Stw:bounds} readily carry over to the semigroup $S_L(t)$ upon a Fourier computation
as in \cite[Lem. 3.3]{kapitula1997} or just by directly exploiting the  relation $S_{L}(t)P_{\rm tw}^\perp=S_H(t)S_{\rm tw}(t)P_{\rm tw}^\perp$.
\end{proof}
    It is important to recall that the travelling wave operator $\mathcal L$ actually has no spectral gap on $\mathbb R^{d}$ for any dimension $d>1$. We have at most algebraic decay  in contrast to exponential decay in one dimension \cite{kapitula2013spectral} or on a cylindrical domain \cite{bosch2025conditionalspeedshapecorrections,bosch2024multidimensional}. This can be  seen easily from a Fourier argument\footnote{Tacitly exploiting the fact that we  focus on the equal diffusion  case, so that there is no need to worry about the possibility of having transverse instabilities \cite{carter2024stabilizing,carter2023criteria,hoffman2015multi}.}, utilising the facts below  regarding the Laplacian $\Delta_y$. 
    In particular,  we point out that operator $P_{\rm tw}$, seen as bounded operator on $H^k,$ is not the spectral projection of $\mathcal L.$  


The  bounds below tell us  that the heat semigroup $S_H(t)$ decays algebraically in $H^k_y$  with rate $(d-1)/4,$ provided  the initial data is also in $L^1_y.$ 
\begin{lemma}\label{lem:theta}
    For any $k\geq 0$ there exists a constant $M>0$ such that for   $t>0$ we have 
    \begin{equation}
        \begin{aligned}
\|S_H(t)w\|_{H^k_y}&\leq M\|w\|_{H^k_y},\\
\|S_H(t)w\|_{H^k_y}&\leq M(1+t)^{-(d-1)/4}\|w\|_{L^1_y}+Me^{-\alpha t}\|w\|_{H^k_y},
\\
\|\partial_{x_i}S_H(t)w\|_{H^{k}_y}
&\leq Mt^{-1/2}\|w\|_{H^k_y},\\
\|\partial_{x_i}S_H(t)w\|_{H^{k}_y}&\leq M(1+t)^{-(d+1)/4}\|w\|_{L^1_y}+Mt^{-1/2}e^{-\alpha t}\|w\|_{H^k_y},
\label{eq:inequalities:heatsemigroup}
        \end{aligned}
    \end{equation}
    for some $\alpha>0$ and $i=2,\ldots,d.$
    In addition, we can take $M=1$ in the first inequality of \eqref{eq:inequalities:heatsemigroup}, showing that  $S_H(t)$ is a contractive semigroup in $H^k_y.$
\end{lemma}
\begin{proof}
    These bounds can be found in \cite[Lem. 3.3]{kapitula1997} and  we specifically refer to \cite[Lem. 2.2]{xin1992multidimensional}; see also \cite{chern1987convergence,kapitula1994stability,kapitula1997,kawashima1987large}. Utilising  the representations of the heat semigroup and the Sobolev norm in terms of the Fourier transform yields \begin{equation}\|S_H(t)w\|_{H^k_y}^2=\int_{\mathbb R^{d-1}}(1+|\xi|^2)^ke^{-2t|\xi|^2}|\widehat w(\xi)|^2\mathrm d\xi\leq \int_{\mathbb R^{d-1}}(1+|\xi|^2)^k|\widehat w(\xi)|^2\mathrm d\xi=\|w\|_{H^k_y}^2,\end{equation} showing that $S_H(t)$ is indeed a contractive semigroup in $H^k_y.$
\end{proof}
\begin{corollary}\label{cor:algebraic:decay}
   For any $k\geq 0$ there exists a constant $M_{\rm alg}>0$ such that for  $t>0$ we have
    \begin{equation}
        \|S_H(t)w\|_{H^k_y}\leq M_{\rm alg}(1+t)^{-(d-1)/4}\big[\|w\|_{L^1_y}+\|w\|_{H^k_y}\big],
    \end{equation}
    while for $t\geq 1$ we have
    \begin{equation}
        \|\nabla_y S_H(t)w\|_{H^k_y}\leq M_{\rm alg}(1+t)^{-(d+1)/4}\big[\|w\|_{L^1_y}+\|w\|_{H^k_y}\big].
    \end{equation}
\end{corollary}

 In \cite{bosch2025conditionalspeedshapecorrections,bosch2024multidimensional} we achieved and utilised exponential decay of the heat semigroup by projecting out the kernel of the generator $\Delta_y$, which was possible since we were working on a compact spatial domain. 
 Since the heat semigroup only admits algebraic decay on the whole space, one must be careful when estimating convolutions with expressions that also decay at algebraic rates.
The following inequality provides the bounds that we need in this regard.
\begin{lemma}\label{lem:z} For every $\lambda\in\mathbb R$ and $\mu,\nu>0$ that satisfies  $\lambda\leq \min\{\mu,\nu\}$,  there exists a constant  $C>0$ such that
    \begin{equation}
    \int_0^z(1+z-z')^{-\mu}(1+z')^{-\nu}\mathrm dz'\leq C(1+z)^{-\lambda},\quad z\geq 0, \label{eq:z}
\end{equation}
whenever  $\lambda \leq \mu+\nu-1$ if $\mu,\nu\neq 1$ or $\lambda<\mu$ if $\nu=1$ or $\lambda<\nu$ if $\mu=1.$ 
\end{lemma}
\begin{proof}
    This inequality can be found in, e.g., \cite[Lem 3.2]{chern1987convergence} and \cite[Lem. 5.3]{hoffman2015multi}.
\end{proof}
\begin{example}
    In specific cases, the integral in \eqref{eq:z} has a closed form. For $d\in\{2,4\},$ we have
    \begin{equation}
        \int_0^t(1+t-s)^{-(d-1)/4}(1+s)^{-(5-d)/4}\,\mathrm ds=2\sqrt 2\arctan\left(\sqrt{\frac{t/2+1}{\sqrt{t+1}}-1}\right)\leq \pi\sqrt2,
    \end{equation}
    while for $d=3$ the  integral above is equal to $2\arctan\left(\frac{t/2}{\sqrt{t+1}}\right)$, which is bounded by $\pi.$
\end{example}
Let
us  recall that a similar phenomena arises with exponential functions. We have
\begin{equation}
    \int_0^z e^{-\mu (z-z')}e^{-\nu z'}\mathrm dz'\leq Ce^{-\lambda z},\quad z\geq 0,\label{eq:est:exponentials}
\end{equation}
where $\lambda\leq \min\{\mu,\nu\}$ for $\mu\neq \nu$; the integral evaluates to $z e^{-\mu z}$ whenever $\mu=\nu,$ which can then be bounded by $Ce^{-\lambda z}$ for $\lambda<\mu$. In addition, it is not difficult to show that algebraic decay dominates exponential decay    in the sense that 
\begin{equation}
    \int_0^z e^{-\mu (z-z')}(1+z')^{-\nu}\mathrm dz'\leq C(1+z)^{-\lambda},\quad z\geq 0,\label{eq:est:exp:poly}
\end{equation}
holds for $\lambda \leq \nu$ and any $\mu,\nu>0.$ By symmetry, we also have
\begin{equation}
    \int_0^z (1+z-z')^{-\mu}e^{-\nu z'}\mathrm dz'\leq C(1+z)^{-\lambda},\quad z\geq 0,\label{eq:est:exp:poly2}
\end{equation}
for $\lambda \leq \mu$ and every $\mu,\nu>0.$
These estimates follow  from Lemma \ref{lem:z} after noting  that  exponential decay can be absorbed into  algebraic decay; for every $\mu>0$ and every $\nu>0$ there exists a constant $C(\mu,\nu)>0$ such that
\begin{equation}
    e^{-\mu z}\leq C(\mu,\nu)(1+z)^{-\nu},\quad z\geq 0.\label{eq:absorb:decay}
\end{equation}
Indeed, an easy computation shows that the optimal constant $C(\mu,\nu)$ is given by
\begin{equation}
C(\mu,\nu)=\begin{cases}
    1,& \nu \leq \mu,\\
       e^{\mu-\frac{\nu}{\mu}}\left(\frac\nu\mu\right)^\nu,&\nu>\mu.
    
    \end{cases}
\end{equation}

\subsection{Stochastic convolutions}\label{sec:stoch:convolution}

Typically,  maximal inequalities (i.e., BDG-type inequalities) rely on the assumption that the semigroup is either contractive \cite{hausenblas2001note} or that it admits an $H^\infty$-calculus \cite{weis2006h}; see also \cite{Neerven2020MaximalEF,veraar2011note}. 
 In our case, we
 exploit the fact that both $S_{\rm tw}(t)$ and therefore $S_L(t)$ admit an $H^\infty$-calculus after projecting out the isolated simple eigenvalue, and that the heat semigroup $S_H(t)$ is contractive.

\begin{lemma}\label{lem:semigroup:bnds:SL}Pick a  sufficiently large $K_{L} \geq 1$. Then for any $T>0$, any integer $p \geq 1$, and any process $X \in \mathcal{N}^2\left([0, T] ;\mathbb F ; H S(\mathcal W_Q, H^k)\right)$ we have the bound
\begin{equation}
\mathbb E\left\|\int_0^t S_L(t-s) X(s) \,\mathrm d W_s^Q\right\|_{H^k}^{2 p} \leq K_{L}^{2 p} p^p \mathbb E\left[\int_0^t\|S_L(t-s) X(s)\|_{H S(\mathcal W_Q;H^k)}^2 \mathrm d s\right]^p,\label{eq:decay}
\end{equation}
for all $0\leq t\leq T,$
together with the maximal inequality
\begin{equation}\label{eq:maximal:ineq:SL}
\mathbb E \sup _{0 \leq t \leq T}\left\|\int_0^t S_L(t-s) X(s) \,\mathrm d W_s^Q\right\|_{H^k}^{2 p} \leq K_{L}^{2 p} p^p \mathbb E\left[\int_0^T\|X(s)\|_{H S(\mathcal W_Q;H^k)}^2 \mathrm d s\right]^p .
\end{equation}
\end{lemma}  
\begin{proof} With regards to the notation in \cite{ kalton2015hinftyfunctional, van2007stochastic,veraar2011note},
we point out that $\gamma(\mathbb R_+;H^k)=L^2(\mathbb R_+;H^k)$, since  $H^k$ is a Hilbert space which is of type 2 as well as cotype 2 \cite[Rem. 4.7]{kalton2015hinftyfunctional}. The inequality \begin{equation}\label{eq:ito}
\mathbb E \sup _{0 \leq t \leq T}\left\|\int_0^t X(s) \,\mathrm d W_s^Q\right\|_{H^k}^{2 p} \leq K_{\mathrm{bdg}}^{2 p} p^p \mathbb E\left[\int_0^T\|X(s)\|_{H S(\mathcal W_Q;H^k)}^2 \mathrm d s\right]^p,
\end{equation} for some $K_{\mathrm{bdg}}>0,$ can be extracted from  \cite[Prop. 2.1 and Rem. 2.2]{veraar2011note}. Estimate \eqref{eq:decay} readily follows from the computation
\begin{equation}
    \mathbb E \left\|\int_0^t S_L(t-s) X(s) \,\mathrm d W_s^Q\right\|_{H^k}^{2 p}\leq \mathbb E \sup _{0 \leq \tau \leq t}\left\|\int_0^\tau S_L(t-s) X(s) \,\mathrm d W_s^Q\right\|_{H^k}^{2 p},\quad 0\leq t\leq T,
\end{equation}
and appealing to the bound in \eqref{eq:ito}.

Turning to the maximal inequality,  we proceed as in \cite[Sec. 3.2]{bosch2024multidimensional} and split the process $X$ as
\begin{equation}
\label{eq:fw:decomp:B}
 X(s)=P_{\rm tw}X(s)+P_{\rm tw}^\perp X(s),   
\end{equation}
noting that the spectral projection $P_{\rm tw}$ acts on processes $X$ via
\begin{equation}
P_{\rm tw}X(s)[\xi]=\langle X(s)[\xi],\psi_{\rm tw}\rangle_{L^2_x}\Phi_0',
\end{equation}
for all $\xi\in\mathcal W_Q$.
Since $S_{\rm tw}(t)\Phi_0'=\Phi_0'$, for all $t\geq 0,$ we can use \eqref{eq:ito} to deal with the process $P_{\rm tw}X.$
For the complementary process $P_{\rm tw}^\perp B$,
we shall make us of the subspaces
\begin{equation}
    H^k_{x,\perp}=\{v\in H^k_x:P_{\rm tw}v=0\},\quad k\geq 0,
\end{equation}
 which are again Hilbert spaces when endowed with the norm $\|\cdot\|_{H^k}$. In \cite[Sec. 3.2]{bosch2024multidimensional}
we concluded that $-\mathcal L_{\rm tw}$ is sectorial in $H^k_{x,\perp}$ and admits a bounded $H^\infty$-calculus of angle strictly smaller then $\pi/2$. 
In view of \cite[Prop. 3.1]{veraar2011note}, there exists 
an equivalent norm $\vvvert \cdot \vvvert_{H^k_x}$
on $H^k_{x,\perp}$, which is given by  
\begin{equation}
    \vvvert v\vvvert_{H^k_x}=
    \Bigg[ \int_0^\infty \| \mathcal L_{\rm tw}^{1/2}S_{\rm tw}(r)v \|^2_{H^k_x} \, \mathrm dr \Bigg]^{1/2},
\end{equation}
ensuring that $S_{\rm tw}(t)$ restricted to $H^k_{x,\perp}$
is contractive with respect to this norm. Indeed, we have
\begin{equation}
    \vvvert S_{\rm tw}(t)v\vvvert_{H^k_x}^2\leq \int_0^\infty \|\mathcal L_{\rm tw}^{1/2}S_{\rm tw}(t+r)v\|_{H^k_x}^2\mathrm dr = \int_{t}^\infty \|\mathcal L_{\rm tw}^{1/2}S_{\rm tw}(s)v\|_{H^k_x}^2\mathrm ds\leq \vvvert v\vvvert_{H^k_x}^2.
\end{equation}
Turning to the full spatial domain, we now introduce  the notation
\begin{equation}
H^k_\perp=\{v\in H^k:P_{\rm tw}v=0\},    
\end{equation}
and endow it with the norm
\begin{equation}
\vvvert v \vvvert_{H^k} 
= \left[ \sum_{ |\beta| \le k  } 
\int_{\mathbb R^{d-1}} \vvvert \partial_y^\beta v( \cdot, y) \vvvert_{H^{k - |\beta|}_{x}}^2 \, \mathrm d y \right]^{1/2},
\end{equation}
where the sum is with respect to multi-indices $\beta \in \mathbb Z^{d-1}_{\ge 0}$. This norm is again equivalent to the usual $\| \cdot \|_{H^k}$ norm, which means that there exist constants $C,c>0$ such that
\begin{equation}
\label{eq:fw:norm:eqv}
    c\|v\|_{H^k}\leq \vvvert v\vvvert_{H^k}\leq C\|v\|_{H^k},
\end{equation}
for all $v \in H^k_\perp$. Since it is not difficult to see that $S_L(t)$ restricted to $H^k_\perp$  is a contraction semigroup with respect to the norm $\vvvert\cdot\vvvert_{H^k}$, i.e., $\vvvert S_L(t) v\vvvert_{H^k}\leq \vvvert v\vvvert_{H^k}$ for any $v\in H^k_\perp,$ we obtain
\begin{equation}
    \begin{aligned}
        \mathbb E \sup _{0 \leq t \leq T}\left\|\int_0^t S_L(t-s) X(s) \,\mathrm d W_s^Q\right\|_{H^k}^{2 p}&\leq C^{2p}\mathbb E \sup _{0 \leq t \leq T}\left\vvvert\int_0^t S_L(t-s) X(s) \,\mathrm d W_s^Q\right\vvvert_{H^k}^{2 p}\\&\leq C^{2p}K_{\rm cnv}^{2 p} p^p \mathbb E\left[\int_0^T\vvvert X(s)\vvvert_{H S(\mathcal W_Q;H^k)}^2 \mathrm d s\right]^p\\
        &\leq C^{2p}K_{\rm cnv}^{2 p}c^{-2p} p^p \mathbb E\left[\int_0^T\|X(s)\|_{H S(\mathcal W_Q;H^k)}^2 \mathrm d s\right]^p,
    \end{aligned}
\end{equation}
with $K_L=CK_{\rm cnv}c^{-1}$ for some $K_{\rm cnv}>0$. Notice that  the final computation above  utilises the fact that the maximal inequality holds for contractive semigroups \cite{hausenblas2001note}.
\end{proof}


\begin{lemma}
\label{prop:SH}Pick a sufficiently large $K_{H} \geq 1$. Then for any $T>0$, any integer $p \geq 1$, and any process $B \in \mathcal{N}^2\big([0, T] ;\mathbb F ; H S(\mathcal W_Q, H^k_y)\big)$ we have the bound
\begin{equation}
\mathbb E\left\|\int_0^t S_H(t-s) B(s) \,\mathrm d W_s^Q\right\|_{H^k_y}^{2 p} \leq K_{H}^{2 p} p^p \mathbb E\left[\int_0^t\|S_H(t-s) B(s)\|_{H S(\mathcal W_Q;H^k_y)}^2 \mathrm d s\right]^p,\label{eq:decay:H}
\end{equation}
for all $0\leq t\leq T,$
together with the maximal inequality
\begin{equation}\label{eq:maximal:ineq:SH}
\mathbb E \sup _{0 \leq t \leq T}\left\|\int_0^t S_H(t-s) B(s) \,\mathrm d W_s^Q\right\|_{H^k_y}^{2 p} \leq K_{H}^{2 p} p^p \mathbb E\left[\int_0^T\|B(s)\|_{H S(\mathcal W_Q;H^k_y)}^2 \mathrm d s\right]^p .
\end{equation}
\end{lemma}  
\begin{proof}
Proving the bounds for the heat semigroup is much less involved. Obtaining estimate \eqref{eq:decay:H} is completely analogous to the proof of Lemma \ref{lem:semigroup:bnds:SL}. The maximal inequality  \eqref{eq:maximal:ineq:SH} simply exploits the fact that the heat semigroup is contractive \cite{hausenblas2001note}.
\end{proof}

In general, the maximal inequality for either a contractive semigroup or a semigroup that admits an $H^\infty$-calculus can be obtained with a dilation argument \cite{Neerven2020MaximalEF,veraar2011note}. 
 Notice that the factorisation method developed by \cite{da1988regularity,da2014stochastic}, although  applicable to general $C_0$-semigroups, leads to a  bound which is less sharp. A discretisation-based approach can also establish   maximal inequalities for stochastic convolutions involving contractive semigroups, even allowing for (random) evolution families \cite{van2021maximal}.

We remark that the inequalities in  \eqref{eq:maximal:ineq:SL} and \eqref{eq:maximal:ineq:SH} are very powerful and useful for short
time intervals, but on longer timescales it is not immediately clear how to exploit the decay
properties of the semigroup. Indeed, the right-hand side grows linearly in
time for integrands that are constant. This
changes if one drops the supremum; see \eqref{eq:decay} and \eqref{eq:decay:H}. As a matter of fact, the maximal inequality can also be used to extract the decay of the semigroup if one splits the integral into the sum of integrals with length one. In particular, we do this in {\S}\ref{sec:est:stoch:conv} and \cite[Sec. 3.3 and 3.4]{bosch2024multidimensional}, yet in slightly different ways. 

 \section{Explicit examples for the  stochastic forcing}\label{sec:about:noise}
  In this appendix we discuss three types of stochastic forcing that can be encoded
  by the stochastic term in \eqref{eq:u} 
  and that fit the assumptions (HgQ) and (H$L^1$) in {\S}\ref{sec:main:result}. 
As already mentioned in \cite{bosch2024multidimensional}, 
we will need to weaken the translational invariance with respect to the transverse direction in order to maintain the desired Hilbert-Schmidt properties. We will utilise weighted and/or localised noise to achieve this,
sacrificing more and more translational invariance along the way.

  

\subsection{Preliminaries}

For all of our examples we have $\mathcal{W} = L^2(\mathbb R^d; \mathbb R^m)$ and  $Q$ is defined via the convolution
\begin{equation}
\label{eq:nt:def:Q}
    [Q w](x,y) = \int_{\mathbb R} \int_{\mathbb R^{d-1}} q(x,y; x', y') w(x', y') \, \mathrm d y'\, \mathrm d x',
\end{equation}
for some appropriate kernel $q$. For each $\xi \in \mathcal{W}_Q$, the expression
$g(u)[\xi]$ is defined via the pointwise multiplication
\begin{equation}
\label{eq:nt:def:g:u:xi:varpi}
    g(u)[\xi](x,y) = \varpi(y) \tilde{g}(u(x,y) ) \xi(x,y)
\end{equation}
for some scalar weight-function $\varpi:\mathbb R^{d-1}\to\mathbb R$ and nonlinearity $\widetilde{g}: \mathbb R^n \to \mathbb R^{n \times m}$.
Notice that the commutation relation $T_{\gamma}g(u)[\xi]=g(T_\gamma u)[T_{\gamma}\xi]$
in (HgQ) is satisfied by construction.

 We first impose a pointwise local Lipschitz condition on the nonlinearity $\tilde{g}$, similar to (Hf). In particular, we require  $\tilde{g}$ to vanish at the limits of the waveprofile $\Phi_0.$ As an immediate consequence, $\tilde{g}$ can be viewed as a locally Lipschitz Nemytskii operator that maps $H^k$ into $H^k(\mathbb R^d; \mathbb R^{n \times m})$.  

\begin{itemize}
    \item[(A$\tilde{g}$)] We have   $\tilde{g}\in C^{k}(\mathbb R^n;\mathbb R^{n\times m})$
    with
    \begin{equation}
        \tilde{g}(u_-)=\tilde{g}(u_+)=0
    \end{equation}
    for the pair $u_\pm\in\mathbb R^n$ appearing in (Hf) and (HTw). In addition, for any $N>0$ there is a  constant  $C_{\tilde{g}}^N > 0$ so that
    \begin{equation}
        |\tilde{g}(u_A)-\tilde{g}(u_B)|+\ldots+|D^{k}\tilde{g}(u_A)-D^{k}\tilde{g}(u_B)|\leq C_{\tilde{g}}^N|u_A-u_B|\label{eq:mr:lipschitz:g}
    \end{equation}
     holds for all  $u_A,u_B\in\mathbb R^{n}$ with $|u_A|\leq N$ and $|u_B|\leq N$.
\end{itemize}

\begin{lemma}\label{lem:Hk:func}
Pick $k > d/2$ and suppose that \textnormal{(Hf)}, \textnormal{(HTw)} and \textnormal{(A$\tilde{g}$)} are satisfied. For all $N>0$ there exists a constant $K_{\tilde{g}}^N > 0$ so that 
for all $\gamma \in \mathbb R$, and every $v_A, v_B \in H^k$
with $\|v_A\|_{H^k} \le N$ and $\|v_B\|_{H^k} \le N$, we
have the bound
\begin{equation}
\label{eq:nl:bnd:delta:theta}
    \|\tilde{g}(T_\gamma \Phi_0+v_A)-\tilde{g}(T_{\gamma}\Phi_0+v_B)\|_{H^k(\mathbb R^d;\mathbb R^{n \times m })}\leq K_{\tilde{g}}^N\|v_A-v_B\|_{H^k} .
\end{equation}
\end{lemma}
\begin{proof}
    We may proceed as in the proof of \cite[Lem. 4.1]{bosch2024multidimensional}, but with one important adjustment since  $\tilde{g}(\Phi_0)$ is no longer contained in $H^k$. Instead, we exploit the fact that $\Phi_0$
    is bounded and satisfies
    \begin{equation}
        \|\partial_x^\ell \Phi_0\|_{\infty}\leq \|\Phi_0'\|_{H^{k}_x} < \infty,\quad \text{for all $1\leq \ell\leq k.$}
    \end{equation}
     This means that  the components $\mathcal I_I$ and $\mathcal I_{II}$, defined within the proof of \cite[Lem. 4.1]{bosch2024multidimensional}, need to be decomposed even further, treating the derivatives of $\Phi_0$ separately from those involving  $v_A$ and $v_B$. Observe that this can be done in a similar fashion as  the proof of Lemma \ref{lem:d2:f:v}.
\end{proof}

In view of  definition \eqref{eq:nt:def:g:u:xi:varpi},
one of the key questions throughout this section
is whether a function $z: \mathbb R^d \to \mathbb R^{n \times m}$ can be viewed as a Hilbert-Schmidt operator
from $\mathcal W_Q$ into $H^k$ via the pointwise multiplication
\begin{equation}
\label{eq:nt:def:multp:op:z}
    z [\xi](x,y) = z(x,y) \xi(x,y).
\end{equation}
To address this, we first note that in all our examples
we have the representation
\begin{equation}
\label{eq:nt:def:sqrt:Q}
    [\sqrt{Q} w](x,y) = \int_{\mathbb R} \int_{\mathbb R^{d-1}} p(x,y; x', y') w(x', y') \, \mathrm d y'\, \mathrm d x',
\end{equation}
for some appropriate kernel $p$.
We recall that any orthonormal basis of $\mathcal{W}_Q$ can be written in the form $(\sqrt{Q} e_{j})_{j \ge 0}$
for some set $(e_{j})_{j \ge 0}$ that is orthonormal in $\mathcal{W}$. Extending this (if necessary) to an orthonormal basis $(e_{j})_{j \ge 0}$ for $\mathcal{W} = L^2(\mathbb R^d; \mathbb R^{m})$,
allows us to introduce the functions
\begin{equation}
\label{eq:nt:def:p:ell}
    p^{j}(x,y) = [\sqrt{Q} e_{j}](x,y) = 
    \int_{\mathbb R} \int_{\mathbb R^{d-1}} p(x,y; x', y') e_j(x', y') \, \mathrm d y'\, \mathrm d x'.
\end{equation}
A straightforward computation involving the product rule|see, for example, the proof of \cite[Lem. 5.4]{bosch2024multidimensional}|implies that
\begin{equation}
\label{eq:nt:fund:bnd:z:hs}
\begin{aligned}
\|z\|_{HS(\mathcal W_Q;H^k)}^2
\le 
(k!)^3
\sum_{j=0}^\infty \sum_{|\alpha|\leq k} \sum_{\beta\le \alpha} \int_{\mathbb R} \int_{\mathbb R^{d-1}}
|\partial^\beta z(x,y)|^2 
|\partial^{\alpha-\beta} p^{j}(x,y) |^2
\, \mathrm dy \, \mathrm dx ,
\end{aligned}\end{equation}
provided that the right-hand side converges. This motivates the definition
\begin{equation}
\label{eq:nt:def:cal:p}
    \mathcal{P}_k(x,y) = \sum_{j=0}^\infty \sum_{|\gamma| \le k}
     | \partial^{\gamma} p^{j}(x,y) |^2=\sum_{|\gamma| \le k}\|\partial^\gamma_{(x,y)} p(x,y;\cdot,\cdot)\|_{L^2(\mathbb R^d;\mathbb R^{m \times m})}^2.
\end{equation}
We now provide three separate conditions on the pair $(\mathcal{P}_k, z)$ that guarantee Hilbert-Schmidt bounds for the multiplication operator \eqref{eq:nt:def:multp:op:z}, which all 
follow directly by inspecting \eqref{eq:nt:fund:bnd:z:hs}.

\begin{corollary}
\label{cor:nt:p:unif:bnd}
Pick $k \ge 0$ and let  $\sup_{(x,y)\in\mathbb R\times\mathbb R^{d-1}}\mathcal{P}_k(x,y)<\infty$. There exists a constant $K > 0$ so that for any $z \in H^k(\mathbb R^d;\mathbb R^{n \times m})$ the associated multiplication operator \eqref{eq:nt:def:multp:op:z} admits the bound
\begin{equation}
    \| z \|_{HS(L^2_Q;H^k)} \le K \| z \|_{H^k(\mathbb R^d;\mathbb R^{n \times m})} .
\end{equation}
\end{corollary}

\begin{corollary}
\label{cor:nt:p:intg}
Pick $k \ge 0$ and suppose that there exists a constant $C > 0$ so that
\begin{equation}
     \int_{\mathbb R^{d-1}} \mathcal{P}_k(x,y)\, \mathrm dy < C
\end{equation}
holds for all $x \in \mathbb R$. There exists a constant $K > 0$ so that for any $z \in H^k(\mathbb R;\mathbb R^{n \times m})$ the associated multiplication operator \eqref{eq:nt:def:multp:op:z}|with the convention $z[\xi](x,y) = z(x)\xi(x,y)$|admits the bound
\begin{equation}
    \| z \|_{HS(L^2_Q;H^k)} \le K \| z \|_{H^k(\mathbb R;\mathbb R^{n \times m})}.
\end{equation}
\end{corollary}

\begin{corollary}
\label{cor:nt:p:intg:full}
Pick $k \ge 0$ and suppose that there exists a constant $C > 0$ so that
\begin{equation}
     \int_{\mathbb R}\int_{\mathbb R^{d-1}} \mathcal{P}_k(x,y)\, \mathrm dy \, \mathrm dx < C.
\end{equation}
There exists a constant $K > 0$ so that for any $z \in W^{k,\infty}(\mathbb R^d;\mathbb R^{n \times m})$ the associated multiplication operator \eqref{eq:nt:def:multp:op:z} admits the bound
\begin{equation}
    \| z \|_{HS(L^2_Q;H^k)} \le K \| z \|_{W^{k,\infty}(\mathbb R;\mathbb R^{n \times m})}.
\end{equation}
\end{corollary}

We now provide a streamlined approach to establish 
the bounds in (HgQ).  
Corollaries \ref{cor:nt:p:unif:bnd}--\ref{cor:nt:p:intg:full} shall be used in the cases below
to check
the conditions
\eqref{eq:nt:bnd:hs:z:hk:1d} and \eqref{eq:nt:bnd:hs:z:hk:full}. Moreover, we note that
\begin{equation}
    \tilde{g}(\Phi_0) \in H^k(\mathbb R; \mathbb R^{n \times m}) \cap W^{k,\infty}(\mathbb R; \mathbb R^{n \times m}).
\end{equation}

\begin{lemma}
\label{lem:nt:hgq:satisfied}
Suppose that \textnormal{(Hf)} and \textnormal{(HTw)} are satisfied and that the operator $Q$ defined
in \eqref{eq:nt:def:Q} is bounded, non-negative and symmetric.
Suppose furthermore that there exists a constant $K > 0$ so that 
for any $\gamma \in \mathbb R$
the multiplication operator \eqref{eq:nt:def:multp:op:z}
associated to $\varpi \tilde{g}(T_\gamma \Phi_0)$ admits the bound
\begin{equation}
\label{eq:nt:bnd:hs:z:hk:1d}
    \| \varpi \tilde{g}(T_\gamma \Phi_0) \|_{HS(L^2_Q; H^k)} \le K,
\end{equation}
while for any $\tilde{z} \in H^k(\mathbb R^d; \mathbb R^{n \times m})$ 
the multiplication operator \eqref{eq:nt:def:multp:op:z}
associated to $\varpi \tilde{z}$ satisfies
\begin{equation}
\label{eq:nt:bnd:hs:z:hk:full}
    \| \varpi \tilde{z} \|_{HS(L^2_Q; H^k)} \le K \| \tilde{z}\|_{H^k(\mathbb R^d; \mathbb R^{n \times m})}.
\end{equation}
Then \textnormal{(HgQ)} is satisfied.
\end{lemma}
\begin{proof}
For any $\gamma \in \mathbb R$ 
and every $v_A, v_B \in H^k$ with $\max\{ \| v_A \|_{H^k}, \| v_B\|_{H^k} \} \le N$, we 
 use \eqref{eq:nt:bnd:hs:z:hk:full} and \eqref{eq:nl:bnd:delta:theta}, respectively,
to see that
\begin{align}
    \| g(T_{\gamma} \Phi + v_A) - g(T_{\gamma} \Phi + v_B) \|_{HS(\mathcal{W}_Q; H^k)}
    &\le \nonumber 
    K \| \tilde{g}( T_{\gamma} \Phi + v_A) - \tilde{g}( T_{\gamma} \Phi + v_B) ) \|_{H^k(\mathbb R^d; \mathbb R^{n \times m})}
    \\
    &\le K K_{\tilde{g}}^N \|v_A - v_B\|_{H^k},
\end{align}
which yields \eqref{eq:HgLip}.
To obtain the estimate in \eqref{eq:Hg}, it suffices
to note that
for any $v \in H^k$ with $\|v\|_{H^k} \le N$
we have
\begin{align}
\| g( T_{\gamma} \Phi_0 + v ) \|_{HS(\mathcal W_Q; H^k)}
& \le  \nonumber
\| g( T_{\gamma} \Phi_0  ) \|_{HS(\mathcal W_Q; H^k)}
+ \| g( T_{\gamma} \Phi_0 + v ) - g( T_{\gamma} \Phi_0) \|_{HS(\mathcal W_Q; H^k)}
\\ 
& \le  
\| \varpi \tilde{g}( T_{\gamma} \Phi_0  ) \|_{HS(\mathcal W_Q; H^k)}
+ K K_{\tilde{g}}^N N .
\end{align}
Indeed, we may now apply \eqref{eq:nt:bnd:hs:z:hk:1d}
to obtain
\begin{equation}
\| g( T_{\gamma} \Phi_0 + v ) \|_{HS(\mathcal W_Q; H^k)}
 \le 
 K  + K C_{\tilde{g}}^N N, 
\end{equation}
which proves the assertion.
\end{proof}

\subsection{Weighted noise}
\label{subsec:nt:wt}

In this setting the convolution kernel $q$ 
appearing in the definition \eqref{eq:nt:def:Q} for $Q$
is given by 
\begin{equation}
    q(x,y;x',y') = \bar{q}(x-x', y-y'),
\end{equation}
which implies that the covariance function is translationally invariant on the full space $\mathbb R^d$.
Writing $\hat {\bar q}$ for the Fourier transform of $\bar q$,
we impose the smoothness condition below. A key example to keep in mind for $m=1$ is the Gaussian $\bar{q}(x,y) = \mathrm{exp}( - (x^2 + |y|^2))$, which becomes another Gaussian after taking the Fourier transform. 

\begin{itemize}
    \item[(AqW)] We have  $\bar q\in H^{k'}(\mathbb R^d;\mathbb R^{m\times m})\cap L^1(\mathbb R^d;\mathbb R^{m\times m})$, for some integer $k' >2k+d/2,$  with $\bar q(-x,-y)= \bar q(x,y)$ and $\bar q^\top(x,y)=\bar q(x,y)$ for all $(x,y)\in \mathbb  R\times \mathbb R^{d-1}$. Furthermore, for any  $(\omega,\xi) \in  \mathbb R\times \mathbb R^{d-1}$ the $m\times m$ matrix $\hat {\bar q}(\omega,\xi)$ is   non-negative definite.   
\end{itemize}

Thanks to Young's inequality, the $L^1$-integrability
of $\bar{q}$
implies that the convolution operator $Q$ defined 
in \eqref{eq:nt:def:Q} is bounded
on $L^2(\mathbb R^d;\mathbb R^m)$. The symmetry properties
in (AqW) make $Q$  self-adjoint,
while the condition on $\hat{q}$ ensures that
\begin{equation}
    \langle Qw,w\rangle_{L^2(\mathbb R^d;\mathbb R^m)}\geq 0
\end{equation}
holds for all $w \in L^2(\mathbb R^d;\mathbb R^m)$. 

Exploiting the $H^{k'}$-smoothness of $\bar{q}$,
one can follow the proof of \cite[Lem. 4.4]{bosch2024multidimensional} to show that the integration kernel 
\eqref{eq:nt:def:sqrt:Q}
associated to $\sqrt{Q}$
admits the representation
\begin{equation}
    p(x,y;x', y') = \overline{p}(x - x', y-y')
\end{equation}
for some function $\overline{p} \in H^k(\mathbb R^d; \mathbb R^{m \times m}) $.
This can be used to obtain the following essential Hilbert-Schmidt bound for multiplication operators.

\begin{lemma}\label{lem:HS:z}
Pick $k \ge 0$ and suppose that \textnormal{(AqW)} is satisfied. There exists a constant $K > 0$ so that
for any $z\in H^k(\mathbb R^d;\mathbb R^{n\times m})$ 
the associated multiplication operator
\eqref{eq:nt:def:multp:op:z}
admits the bound
\begin{equation}
\label{eq:nt:trn:inv:hs:z}
    \| z\|_{HS(L^2_Q ; H^k)}
    \le K \| z\|_{H^k(\mathbb R^d; \mathbb R^{n \times m}) }.
\end{equation}
\end{lemma}
\begin{proof}
Recalling the quantity \eqref{eq:nt:def:cal:p}, we compute
\begin{equation}
\mathcal{P}_k(x,y) = 
\sum_{|\gamma| \le k}
\| \partial^\gamma \overline{p}(x  - \cdot, y - \cdot) \|_{L^2(\mathbb R^{d};\mathbb R^{m \times m})}^2
 = \| \overline{p} \|_{H^k(\mathbb R^d; \mathbb R^{m \times m})}^2,
 \end{equation}
 and observe  that its value is  independent of $(x,y).$
 The result  follows from Corollary \ref{cor:nt:p:unif:bnd}.
\end{proof}

Since $\tilde{g}(\Phi_0)$ is not in $H^k$, 
we will need the weight-function $\varpi$ in \eqref{eq:nt:def:g:u:xi:varpi}
to decay as $|y|\to \infty$ at a rate
that allows the use of
\eqref{eq:nt:trn:inv:hs:z} to show that the noise term
$g$ is well-defined. More precisely, we impose the following condition, which is satisfied for example by 
$\varpi(y)=e^{-|y|^2}$. We note that the $L^1$-bound
is only needed to verify (H$L^1$).

\begin{itemize}
\item[(A$\varpi$W)]
We have $\varpi\in H^{k''}(\mathbb R^{d-1};\mathbb R) \cap L^1(\mathbb R^{d-1};\mathbb R)$ for some integer $k''>k+(d-1)/2$.
\end{itemize}

\begin{lemma}
Pick $k > d/2$ and assume that \textnormal{(Hf)}, \textnormal{(HTw)}, \textnormal{(A$\tilde{g}$)}, \textnormal{(AqW)}, and \textnormal{(A$\varpi$W)} all hold. Then \textnormal{(HgQ)} is satisfied.
\end{lemma}
\begin{proof}
Note first that there exists a constant $C_1 > 0$ so that 
\begin{equation}
\| \partial^\alpha \varpi \|_{\infty } < C_1,
\end{equation}
for any multi-index $\alpha \in \mathbb Z^{d-1}_{\ge 0}$ with $|\alpha| \le k$. In particular, for any
$z \in H^k(\mathbb R^d; \mathbb R^{n \times m})$
we have
\begin{equation}
  \| \varpi z \|_{H^k(\mathbb R^d; \mathbb R^{n \times m})} \le C_2   \|z\|_{H^k(\mathbb R^d; \mathbb R^{n \times m})},
\end{equation}
for some $C_2 > 0$. In addition,
there exists a constant $C_3 > 0$ so that
for any $\gamma \in \mathbb R$
we have
\begin{equation}
\| \varpi \tilde{g}(T_{\gamma} \Phi_0)\|_{H^k(\mathbb R^d; \mathbb R^{n \times m})}
\le
C_3 \| \varpi \|_{H^k(\mathbb R^{d-1};\mathbb R)} \| \tilde{g}( \Phi_0)\|_{H^k(\mathbb R; \mathbb R^{n \times m} ) } ,
\end{equation}
because the weight-function $\varpi$ only depends on $y$ while 
$\tilde{g}(T_{\gamma} \Phi_0)$
 only depends on $x$.
In particular, the conditions
\eqref{eq:nt:bnd:hs:z:hk:1d}
and \eqref{eq:nt:bnd:hs:z:hk:full}
follow from \eqref{eq:nt:trn:inv:hs:z},
allowing us to apply
Lemma \ref{lem:nt:hgq:satisfied}.
\end{proof}

\begin{lemma}
Pick $k > d/2$ and assume that \textnormal{(Hf)}, \textnormal{(HTw)}, \textnormal{(A$\tilde{g}$)}, \textnormal{(AqW)}, and \textnormal{(A$\varpi$W)} all hold. Then \textnormal{(H$L^1$)} is satisfied.
\end{lemma}
\begin{proof}
Let $(\sqrt{Q}e_j)_{j \geq 0}$ denote an arbitrary orthonormal basis of $L^2_Q.$ For any $\gamma \in \mathbb R$
and $v \in H^k$, an application of Cauchy-Schwarz 
together with Lemma \ref{lem:HS:z}  with $k=0$
shows that
    \begin{equation}\begin{aligned}
        \textstyle \sum_{j =0}^\infty \|\varpi& (\tilde{g}(T_\gamma \Phi_0+v)\sqrt Qe_j)^\top T_\gamma \psi_{\rm tw}\|_{L^1(\mathbb R^d;\mathbb R)}^2\\&\textstyle \leq \sum_{j =0}^\infty \|T_\gamma \psi_{\rm tw} \sqrt{|\varpi|} \|_{L^2}^2\|\sqrt{|\varpi|}\tilde{g}(T_\gamma \Phi_0+v)\sqrt Q e_j\|_{L^2}^2\\
        &=\textstyle \|\psi_{\rm tw}\|_{L^2_x}^2\|\varpi\|_{L^1(\mathbb R^{d-1};\mathbb R)}\| \sqrt{|\varpi|} \tilde{g}(T_\gamma \Phi_0+v)\|_{HS(L^2_Q;L^2)}^2\\
        &\textstyle \le K \|\varpi\|_{L^1(\mathbb R^{d-1};\mathbb R)} \| \sqrt{|\varpi|} \tilde{g}(T_\gamma \Phi_0+v)\|_{L^2(\mathbb R^{d}; \mathbb R^{n\times m})}^2
    \end{aligned}\end{equation}
    for some $K > 0$ that does not depend on the choice of basis. The assertion now follows from
    \begin{equation}
        \begin{aligned}
            \|\sqrt{\varpi}&g(T_{\gamma} \Phi_0+v)\|_{L^2(\mathbb R^d;\mathbb R^{n\times m})}\\& \textstyle \leq \|\sqrt{|\varpi|}[\tilde{g}(T_\gamma \Phi_0+v)-\tilde{g}(T_\gamma \Phi_0)] \|_{L^2(\mathbb R^d;\mathbb R^{n\times m})}
            + \|\sqrt{|\varpi|}\tilde{g}(T_\gamma \Phi_0)] \|_{L^2(\mathbb R^d;\mathbb R^{n\times m})}
            \\&\leq\textstyle  \|\varpi\|_{\infty}^{1/2}\|\tilde{g}(T_\gamma \Phi_0+v)-\tilde{g}(T_\gamma \Phi_0)\|_{L^2(\mathbb R^d;\mathbb R^{n\times m})}
            +\|\varpi\|_{L^1(\mathbb R^{d-1};\mathbb R)}^{1/2}\|\tilde{g}(\Phi_0)\|_{L^2(\mathbb R;\mathbb R^{n\times m})}
            \\&\leq  K_{\tilde g}^N\|\varpi\|_{H^k (\mathbb R^{d-1};\mathbb R)}^{1/2}\|v\|_{L^2(\mathbb R;\mathbb R^{n\times m})}+\|\varpi\|_{L^1(\mathbb R^{d-1};\mathbb R)}^{1/2}\|\tilde{g}(\Phi_0)\|_{L^2(\mathbb R;\mathbb R^{n\times m})},
        \end{aligned}
    \end{equation}
    which holds on account of \eqref{eq:nl:bnd:delta:theta}  whenever $\|v \|_{H^k} \le N$.
\end{proof}

\subsection{Localised noise in the transverse direction}
\label{subsec:nt:loc}

Setting the weight-function in \eqref{eq:nt:def:g:u:xi:varpi} to $\varpi = 1$,
we will now consider the case where the integration kernel in the definition \eqref{eq:nt:def:Q} for $Q$
is given by
\begin{equation}
\label{eq:nt:loc:trv:q:repr}
    q(x,y;x', y') =  \bar{q}(x-x')\sum_{\ell = 0}^\infty \lambda_\ell   \mu_\ell(y) \mu_\ell(y')^\top.
\end{equation}
In particular, $q$ exhibits translational invariance in the $x$-direction only. Explicit examples can be obtained by combining those given in {\S}\ref{subsec:nt:wt} 
and {\S}\ref{subsec:nt:loc:full}. Note that our requirements on $\bar{q}$, which
we choose to be scalar in order to ensure that it commutes with the matrices in \eqref{eq:nt:loc:trv:q:repr},
are the scalar counterparts of those in {\S}\ref{subsec:nt:wt}.

\begin{itemize}
    \item[(AqL1)] We have  $\bar q\in H^{k'}(\mathbb R;\mathbb R)\cap L^1(\mathbb R;\mathbb R)$, for some integer $k' \geq 2k+1,$ with $\bar q(-x)=\bar q(x)$ 
    for all $x\in \mathbb  R$. Furthermore, we have $\hat{\bar q}(\omega) \ge 0$ for all $\omega\in\mathbb R$.
\end{itemize}

We remark that the $(y,y')$-dependence in the representation \eqref{eq:nt:loc:trv:q:repr} should be seen as an explicit expansion in terms of the  eigenfunctions of the integral operator $Q$. 
This is encoded in the following structural and summability conditions.
\begin{itemize}
     \item[(AqL2)]   
     We have $\lambda_\ell \ge 0$ for all $\ell \ge 0$, together with the orthogonality properties
     \begin{equation}
         \langle \mu_\ell, \mu_{\ell'} \rangle_{L^2(\mathbb R^{d-1} ; \mathbb R^{m})} = \delta_{\ell\ell'},\quad 0\leq \ell,\ell'<\infty,
     \end{equation}
     and the summability conditions
     \begin{equation}
     \label{eq:nt:aql1:sum:cond}
         \sum_{\ell=0}^\infty \lambda_\ell \| \mu_\ell \|^2_{H^{k''}(\mathbb R^{d-1}; \mathbb R^{m})} < \infty,
         \qquad 
         \sum_{\ell=0}^\infty \lambda_\ell^{1/2} \| \mu_\ell\|_{L^1(\mathbb R^{d-1}; \mathbb R^{m})} < \infty,
     \end{equation}
     for some integer $k'' > k + (d-1)/2$.
\end{itemize}

We note that the first bound in \eqref{eq:nt:aql1:sum:cond} directly implies the trace condition
\begin{equation}
    \mathrm{Tr} \, q(x, \cdot\,; x', \cdot) = \bar q(x-x')\sum_{\ell=0}^\infty \lambda_\ell
    < \infty,
\end{equation}
in which $\mathrm{Tr}$ stands for the trace operator
that acts as
\begin{equation}
    \mathrm{Tr} \, h( \cdot \,; \cdot) = \int_{\mathbb R^{d-1}} \sum_{i=1}^m [h(y; y)]_{ii} \, \mathrm dy
\end{equation}
on functions $h: \mathbb R^{d-1} \times \mathbb R^{d-1} \to \mathbb R^{m \times m}$.
In particular, for every $x, x' \in \mathbb R$ we have
\begin{equation}
    \| q(x, \cdot\,; x', \cdot) \|_{L^2(\mathbb R^{d-1} \times \mathbb R^{d-1};\mathbb R^{m \times m})}^2 = 
    \bar q(x-x')^2\sum_{\ell= 0} ^\infty \lambda_\ell^2 \le \left( \bar q(x-x')\sum_{\ell=0}^\infty \lambda_\ell\right)^2 < \infty.
\end{equation}
Combining Young's inequality for the $x$-direction with Minkowski's inequality for the $y$-direction, we see that
$Q$ is bounded on $\mathcal{W}$.
The fact that $\bar{q}$ is a scalar function, together with the symmetry conditions in (AqL1) and the explicit structure of the $(y,y')$-dependence in \eqref{eq:nt:loc:trv:q:repr}, readily show that $Q$ is self-adjoint. Together with the condition on $\hat{\bar q}$, this explicit structure can also be used to confirm the non-negativity of $Q$.

\begin{remark}\label{remark:B}
We remark here that the first condition in \eqref{eq:nt:aql1:sum:cond} can be reformulated as
\begin{equation}
    \sum_{|\gamma| \le k''} \mathrm{Tr} \,  \partial^\gamma_y \partial^\gamma_{y'} q(x, \cdot\,; x', \cdot ) < \infty,
\end{equation}
hence bypassing the need to consider the full eigenfunction expansion \eqref{eq:nt:loc:trv:q:repr}. The sign conditions on $\lambda_\ell$ can be verified in a semi-explicit fashion via Mercer's condition \cite{mercer1909xvi}, which requires $q(x, \cdot\,; x', \cdot)$ to be positive semi-definite. However, we are not aware of similar methods to 
express the $L^1$-condition in \eqref{eq:nt:aql1:sum:cond} directly in terms of $q$, which motivates our choice to work with \eqref{eq:nt:loc:trv:q:repr}.
\end{remark}

We now note that (AqL1) allows us to construct a function $\overline{p} \in H^k(\mathbb R; \mathbb R)$ satisfying $\hat{p} = \sqrt{ \hat{q} }$; see for example \cite[Lem 4.4]{bosch2024multidimensional}. In particular, we have $\bar{q} = \bar{p} * \bar{p}$. This can be used
to show that the integration kernel \eqref{eq:nt:def:sqrt:Q} associated
to $\sqrt{Q}$ can be written in the form
\begin{equation}
\label{eq:nt:def:p:trv:loc}
   p(x,y;x', y') = \overline{p}(x - x') \sum_{\ell= 0}^\infty \lambda_\ell^{1/2} \mu_\ell(y) \mu_\ell(y')^\top .
\end{equation}
Picking an arbitrary orthonormal basis $(e_{j})_{j \ge 0}$ for $L^2(\mathbb R^d; \mathbb R^m)$, 
we introduce the notation
\begin{equation}
 \tilde{p}_{j;\ell}(x)  =    
 \int_{\mathbb R} \int_{\mathbb R^{d-1}} \bar p(x - x')\mu_\ell(y')^\top  e_{j}(x', y') \, \mathrm  d y'\, \mathrm dx',
\end{equation}
which allows the functions \eqref{eq:nt:def:p:ell}
to be expressed as 
\begin{equation}
    p^{j}(x, y) = \sum_{\ell=0}^\infty \lambda_{\ell}^{1/2}\mu_{\ell}(y) \tilde{p}_{j;\ell}(x).
\end{equation}
The key observation is that
\begin{equation}
\label{eq:nt:sum:ptilde:j:full}
\begin{array}{lcl}
    \sum_{j=0}^\infty \partial^{\gamma_x}_x \tilde{p}_{j;\ell}(x) 
    \partial^{\gamma_x}_x\tilde{p}_{j;\ell'}(x)
    & = & \langle \partial^{\gamma_x}_x \bar p(x - \cdot) \mu_\ell(\cdot) ,
    \partial^{\gamma_x}_x \bar p(x - \cdot) \mu_{\ell'}(\cdot) 
    \rangle_{L^2(\mathbb R^d;\mathbb R^{m})}
    \\[0.2cm]
    & = & \|\partial^{\gamma_x}_x \bar p \|_{L^2(\mathbb R ; \mathbb R)}^2
    \delta_{\ell\ell'}
\end{array}
\end{equation}
holds for any $x \in \mathbb R$, any $\ell \ge 0$, any $\ell'\ge 0$, and any integer
$0 \le \gamma_x \le k$. Indeed, this allows us to
express \eqref{eq:nt:def:cal:p} in the
form\footnote{An alternative computation that bypasses the basis $(e_j)_{j \ge 0}$
proceeds by computing the $\mathbb R^{m \times m}$-norm of 
\eqref{eq:nt:def:p:trv:loc}, but the functions $p^j$ are needed anyway in the proof of Lemma \ref{lem:nt:trv:loc:hl1}.}
\begin{align}
\mathcal{P}_k(x,y)
    &\textstyle\nonumber =  \sum_{\gamma_x, \gamma_y: |\gamma_x| + |\gamma_y| \le k}
    \sum_{j=0}^\infty \sum_{\ell , \ell'=0}^\infty \lambda_\ell^{1/2} \lambda_{\ell'}^{1/2}\partial_x^{\gamma_x} \tilde{p}_{j;\ell}(x)\partial_x^{\gamma_x} \tilde{p}_{j;\ell'}(x)
    \langle \partial_y^{\gamma_y} \mu_\ell(y), \partial_y^{\gamma_y} \mu_{\ell'}(y) \rangle_{\mathbb R^m}
\\
&\textstyle =  
\sum_{\gamma_x, \gamma_y: |\gamma_x| + |\gamma_y| \le k}  \sum_{j=0}^\infty \sum_{\ell=0 }^\infty
    \lambda_\ell\partial_x^{\gamma_x} \tilde{p}_{j;\ell}(x)^2  |\partial_y^{\gamma_y} \mu_\ell(y)|^2
\\
&\textstyle\nonumber = 
\sum_{\gamma_x, \gamma_y: |\gamma_x| + |\gamma_y| \le k}
\|\partial_x^{\gamma_x} \bar{p}\|_{L^2(\mathbb R; \mathbb R)}^2\sum_{\ell=0 }^\infty
    \lambda_\ell
|\partial_y^{\gamma_y} \mu_\ell(y)|^2,
\end{align}
where $\gamma_y\in\mathbb Z^{d-1},$ which leads to the bound
\begin{equation}
\mathcal{P}_k(x,y)
\leq 
(k+1)\|\bar p\|_{H^k(\mathbb R; \mathbb R)}^2
    \sum_{\ell=0}^\infty \lambda_\ell \sum_{|\gamma_y| \le k}
    | \partial^{\gamma_y}_y \mu_\ell(y) |^2.    
\end{equation}
In particular, we see that there exists a constant $K > 0$ so that
\begin{equation}
\label{eq:nt:loc:bnd:p:k:unif}
    \mathcal{P}_k(x,y) \le 
    K \|\bar p\|_{H^k(\mathbb R; \mathbb R)}^2 
    \sum_{\ell=0}^\infty \lambda_\ell \| \mu_\ell \|_{H^{k''}(\mathbb R^{d-1}; \mathbb R^{m})}^2, 
\end{equation}
holds for any $(x,y) \in \mathbb R^d$,
using the fact  $k'' > k + (d-1)/2$ 
to obtain a uniform bound on all the derivatives. In addition, we observe that for all $x \in \mathbb R$ we have
\begin{equation}
\label{eq:nt:loc:bnd:p:k:intg:y}
    \int_{\mathbb R^{d-1}} \mathcal{P}_k(x,y) \, \mathrm dy 
    \le 
    K \|\bar p\|_{H^k(\mathbb R; \mathbb R)}^2 
    \sum_{\ell=0}^\infty \lambda_\ell \| \mu_\ell \|_{H^{k}(\mathbb R^{d-1}; \mathbb R^{m})}^2 .
\end{equation}

\begin{corollary}
Pick $k > d/2$ and assume that \textnormal{(Hf)}, \textnormal{(HTw)}, \textnormal{(A$\tilde{g}$)}, \textnormal{(AqL1)}, and \textnormal{(AqL2)} all hold. Then \textnormal{(HqQ)} is satisfied.    
\end{corollary}
\begin{proof}
In view of the bounds \eqref{eq:nt:loc:bnd:p:k:unif} and \eqref{eq:nt:loc:bnd:p:k:intg:y},
this follows directly from
Corollaries \ref{cor:nt:p:unif:bnd}--\ref{cor:nt:p:intg}
and Lemma \ref{lem:nt:hgq:satisfied}.
\end{proof}

\begin{lemma}
\label{lem:nt:trv:loc:hl1}
Pick $k > d/2$ and assume that \textnormal{(Hf)}, \textnormal{(HTw)}, \textnormal{(A$\tilde{g}$)}, \textnormal{(AqL1)}, and \textnormal{(AqL2)} all hold. Then \textnormal{(H$L^1$)} is satisfied.
\end{lemma}
\begin{proof}
Observe first that there exists a constant $C_N >0$ so that
\begin{equation}
    \|\tilde{g}(T_{\gamma} \Phi_0 + v)\|_\infty \le C_N,
\end{equation}
for every $\gamma \in \mathbb R$ and $v \in H^k$
with $\|v\|_{H^k} \le N$, on account of 
\eqref{eq:mr:lipschitz:g}
and the fact that $\Phi_0$ is bounded.
Writing $w_\ell(x,y) = |\psi_{\rm tw}(x)| |\mu_\ell(y)|$ 
for convenience,
this allows us to compute
\begin{align}
\label{eq:nt:l1:comp:loc:bnd:g:psi:w:l1}
   \nonumber \textstyle \|(g(T_{\gamma} \Phi_0+v) p^{j})^\top T_{\gamma} \psi_{\rm tw}\|_{L^1(\mathbb R^d;\mathbb R)}^2
    & \textstyle \le  C_N^2
    \big[
    \sum_{\ell=0}^\infty \lambda_\ell^{1/2} 
     \int_{\mathbb R} \int_{\mathbb R^{d-1}} T_{\gamma} w_\ell(x,y)
      |\tilde{p}_{j;\ell}(x)| \, \mathrm dy \, \mathrm dx \big]^2
\\
& \le \textstyle  
C_N^2
\big[
\sum_{\ell=0}^\infty \lambda_\ell^{1/2}
\| T_{\gamma} w_\ell \|_{L^1(\mathbb R^d; \mathbb R  ) }^{1/2}
\| (T_{\gamma} w_\ell)
     \tilde{p}_{j;\ell}^2 \|_{L^1(\mathbb R^d; \mathbb R  )}^{1/2}
\big]^2
\\
& \le \textstyle  
C_N^2  
\sum_{\ell=0}^\infty \lambda_\ell^{1/2} \| w_\ell\|_{L^1(\mathbb R^d; \mathbb R  )}
\sum_{\ell=0}^\infty \lambda_\ell^{1/2} \| (T_{\gamma} w_\ell) \tilde{p}_{j;\ell}^2 \|_{L^1(\mathbb R^d; \mathbb R  )}.\nonumber
\end{align}
Applying \eqref{eq:nt:sum:ptilde:j:full}, we hence find
\begin{equation}
\label{eq:nt:loc:bnd:l1}
\begin{aligned}
 \textstyle \sum_{j=0}^\infty   \|(g&(T_{\gamma}\Phi_0+v) p^{j})^\top T_{\gamma} \psi_{\rm tw}\|_{L^1(\mathbb R^d;\mathbb R)}^2
    \\& \le \textstyle 
    C_N^2 \| \bar{p}\|_{L^2(\mathbb R; \mathbb R)}^2
\sum_{\ell=0}^\infty \lambda_\ell^{1/2} \| w_\ell \|_{L^1(\mathbb R^d;\mathbb R)}
\sum_{\ell=0}^\infty \lambda_\ell^{1/2} \| T_\gamma w_\ell   \|_{L^1(\mathbb R^d;\mathbb R)}
\\
& \le \textstyle 
C_N^2 \| \bar{p}\|_{L^2(\mathbb R; \mathbb R)}^2
\| \psi_{\rm tw} \|_{L^1(\mathbb R; \mathbb R^n)}^2
\big( \sum_{\ell=0}^\infty \lambda_\ell^{1/2} \| \mu_\ell\|_{L^1(\mathbb R^{d-1}; \mathbb R^{ m})} \big)^2,
\end{aligned}
\end{equation}
which completes the proof in view of \eqref{eq:nt:aql1:sum:cond}.
\end{proof}

It is insightful to consider an orthonormal basis 
$(e_{j,\ell})_{j, \ell \ge 0}$ for $L^2(\mathbb R^d; \mathbb R^m)$
of the product form
\begin{equation}
 e_{j,\ell}(x,y) = e_j(x) \mu_\ell(y), 
\end{equation}
in which $(e_j)_{j \ge 0}$ is some orthonormal basis for $L^2(\mathbb R; \mathbb R)$ and the set $(\mu_\ell)_{\ell \ge 0}$ forms|possibly after extension|an orthonormal basis for $L^2(\mathbb R^{d-1}; \mathbb R^{m})$.
In this case, we have
$\tilde{p}_{(j,\ell');\ell}(x) = \delta_{\ell \ell'} \tilde{p}_{(j,\ell);\ell}(x)$, which implies
\begin{equation}
    p^{(j,\ell)}(x,y) = \lambda_\ell^{1/2} \mu_\ell(y)
    \tilde{p}_{(j,\ell);\ell}(x).
\end{equation}
Using \eqref{eq:nt:sum:ptilde:j:full},
we observe that
\begin{equation}
\label{eq:nt:special:basis:sum:i}
    \sum_{j=0}^\infty \partial^{\gamma_x}_x \tilde{p}_{(j,\ell); \ell}(x)^2 = 
    \sum_{j,\ell'=0}^\infty \partial^{\gamma_x}_x \tilde{p}_{(j,\ell'); \ell}(x)^2
    = 
    \| \partial^{\gamma_x}_x \bar p \|_{L^2(\mathbb R ; \mathbb R)}^2,
\end{equation}
which, as a side remark,  provides us with an alternative yet more direct path to the identity
\begin{equation}
\mathcal{P}_k(x,y)
=
    \sum_{\gamma_x, \gamma_y: |\gamma_x| + |\gamma_y| \le k}
    \|\partial_x^{\gamma_x} \bar{p}\|_{L^2(\mathbb R; \mathbb R)}^2 \sum_{\ell=0 }^\infty
    \lambda_\ell |\partial_y^{\gamma_y} \mu_\ell(y)|^2.
\end{equation}

On the other hand, the sum in the second line
of \eqref{eq:nt:l1:comp:loc:bnd:g:psi:w:l1} 
reduces to a single term, leading to
\begin{equation}\textstyle
    \|(g(T_\gamma \Phi_0+v) p^{(j,\ell)})^\top T_{\gamma} \psi_{\rm tw}\|_{L^1(\mathbb R^d;\mathbb R)}^2
    \le  C_N^2   \lambda_\ell \|w_\ell\|_{L^1(\mathbb R^d; \mathbb R)}  \| (T_{\gamma} w_\ell) \tilde{p}_{(j,\ell);\ell}^2 \|_{L^1(\mathbb R^d; \mathbb R)}.
\end{equation}
Applying \eqref{eq:nt:special:basis:sum:i} with $\gamma_x = 0$,
we hence see that
\begin{equation}
\begin{aligned}
\textstyle \sum_{j,\ell=0}^\infty \|(g&(T_\gamma \Phi_0+v) p^{(j,\ell)})^\top T_{\gamma} \psi_{\rm tw}\|_{L^1(\mathbb R^d;\mathbb R)}^2
\\& \textstyle \le  
 C_N^2   \| \bar p \|_{L^2(\mathbb R ; \mathbb R)}^2 
 \| \psi_{\rm tw} \|_{L^1(\mathbb R; \mathbb R^n)}^2
 \sum_{\ell=0}^\infty \lambda_\ell \|\mu_\ell\|_{L^1(\mathbb R^{d-1}; \mathbb R^m)}^2 .
\end{aligned}
\end{equation}
Since sequence space norms admit the ordering $\| \cdot \|_{\ell^2} \le \|\cdot \|_{\ell^1}$, we see that 
this bound is sharper than \eqref{eq:nt:loc:bnd:l1}.
This (potential) basis-dependent estimate points out an interesting subtlety related to the use of the 2-summing norm. Indeed, if one simply uses
the fixed basis $(e_{j,\ell})_{j,\ell\ge 0}$ throughout this section, the second condition in \eqref{eq:nt:aql1:sum:cond}
can be relaxed to
\begin{equation}
    \sum_{\ell=0}^\infty \lambda_\ell \|\mu_\ell\|_{L^1(\mathbb R^{d-1}; \mathbb R^m)}^2 < \infty.
\end{equation}

\subsection{Fully localised noise}
\label{subsec:nt:loc:full}
Continuing
to use $\varpi = 1$ for the weight-function in 
\eqref{eq:nt:def:g:u:xi:varpi}, our final  setting is where the integration kernel in  the definition \eqref{eq:nt:def:Q} for $Q$ is given by
\begin{equation}
\label{eq:nt:full:loc:repr:q}
    q(x,y;x', y') =  \sum_{\ell = 0} ^\infty\lambda_\ell   \mu_\ell(x,y) \mu_\ell(x',y')^\top,
\end{equation}
which means that we have fully lost translational invariance. 
In this case, the statistical properties of the translated noise process differ from those of the original stationary noise. Consequently, it is now really necessary to keep on writing $T_{-\gamma}\mathrm dW_t^Q$  in \eqref{eq:pert:system}   instead of replacing it by $\mathrm dW_t^Q$.

Choosing $m=1$ and writing $\mathbf x=(x,y)$, two motivating examples are given by
\begin{equation}
 q(\mathbf{x};\mathbf{x}')=\begin{cases}
        e^{-1/(R-|\mathbf x|^2)}e^{-1/(R-|\mathbf x'|^2)},&\max\{|\mathbf x|,|\mathbf x'|\}\leq R,\\
        0&\rm else,
    \end{cases}
\end{equation}
for any $R>0$, or  
\begin{equation}q(\mathbf{x};\mathbf x')=\frac{1}{(1+|\mathbf x|^2)^\zeta(1+|\mathbf x'|^2)^\zeta},\end{equation} for  large enough $\zeta>0$ to make $q$ sufficiently smooth.
These examples are both of rank one, in the sense that only one of the scalars $\lambda_\ell$ in 
\eqref{eq:nt:full:loc:repr:q} is non-zero.

As in {\S}\ref{subsec:nt:loc}, we view
the representation \eqref{eq:nt:full:loc:repr:q}
for the convolution kernel as an explicit expansion
in terms of  eigenfunctions of $Q$. In particular, we impose the following structural and summability conditions,
which are the direct analogues of those in (AqL2) above; see also Remark \ref{remark:B}

\begin{itemize}
     \item[(AqT)]   
     We have $\lambda_\ell \ge 0$ for all $\ell\ge 0$, together with the orthogonality properties
     \begin{equation}
         \langle \mu_\ell, \mu_{\ell'} \rangle_{L^2(\mathbb R^{d} ; \mathbb R^{m})} = \delta_{\ell \ell'},\quad 0\leq \ell,\ell'<\infty,
     \end{equation}
     and the summability conditions
     \begin{equation}
     \label{eq:nt:aqt:sum:cond}
         \sum_{\ell=0}^\infty \lambda_\ell \| \mu_\ell \|^2_{H^{k'}(\mathbb R^{d}; \mathbb R^{m})} < \infty,
         \qquad 
         \sum_{\ell=0}^\infty \lambda_\ell^{1/2} \| \mu_\ell\|_{L^1(\mathbb R^{d}; \mathbb R^{m})} < \infty,
     \end{equation}
     for some integer $k' > k + d/2$.
\end{itemize}

By repeating the arguments in {\S}\ref{subsec:nt:loc} almost
 verbatim, we find that these conditions
imply that  $Q$ is bounded,  self-adjoint and positive semi-definite, and
\begin{equation}
    \mathrm{Tr} \, q := \int_{\mathbb R} \int_{\mathbb R^{d-1}} \sum_{i=1}^m[q(x,y;x, y)]_{ii} \, \mathrm dy \, \mathrm dx = \sum_{\ell=0}^\infty \lambda_\ell < \infty.
\end{equation}
Extending the sequence $(\mu_\ell)_{\ell \ge 0}$ (if necessary) to form an orthonormal basis for $\mathcal{W} = L^2(\mathbb R^{d}; \mathbb R^{m})$, we furthermore compute
\begin{equation}
    \mathrm {Tr}(Q) :=\sum_{\ell=0}^\infty\langle Q\mu_\ell,\mu_\ell\rangle_{\mathcal W}=\sum_{\ell=0}^\infty \|\sqrt{Q} \mu_\ell\|_{\mathcal W}^2=\|\sqrt{Q}\|_{HS(\mathcal W)}^2 = 
    \sum_{\ell = 0}^\infty \lambda_\ell <\infty.
\end{equation}
In particular,  $Q$ is trace class in $\mathcal W$, 
$\sqrt{Q}$ is a Hilbert-Schmidt operator on $\mathcal{W}$, and  $(W^Q_t)_{t \ge 0}$ is a (regular) $Q$-Wiener process taking values in $\mathcal W_Q$. In this context it is worthwhile to recall that any trace class operator is necessarily an integral operator \cite{treves2016topological}.

The integration kernel \eqref{eq:nt:def:sqrt:Q} associated to $\sqrt{Q}$ can be written in the form
\begin{equation}
   p(x,y;x', y') = \sum_{\ell=0}^\infty \lambda_\ell^{1/2}  \mu_\ell(x,y) \mu_\ell(x',y')^\top.
\end{equation}
By fixing an arbitrary basis $(e_{j})_{j \ge 0}$ for $L^2(\mathbb R^d; \mathbb R^m)$,
we introduce the notation
\begin{equation}
 \tilde{p}_{j;\ell}  =    
 \int_{\mathbb R}\int_{\mathbb R^{d-1}} \mu_\ell(x', y')^\top e_{j}(x', y') \, \mathrm d y'\, \mathrm dx', 
\end{equation}
which allows the functions \eqref{eq:nt:def:p:ell} to be expressed as
\begin{equation}
    p^{j}(x, y) = \sum_{\ell=0}^\infty \lambda_{\ell}^{1/2}\mu_{\ell}(x,y) \tilde{p}_{j;\ell}.
\end{equation}
The appropriate modification of \eqref{eq:nt:sum:ptilde:j:full} to this setting reads
\begin{equation}
\label{eq:nt:full:loc:sum:ptilde:j:full}
\begin{array}{lcl}
    \sum_{j=0}^\infty  \tilde{p}_{j;\ell} 
    \tilde{p}_{j;\ell'}
    =\langle  \mu_\ell(\cdot) ,
     \mu_{\ell'}(\cdot) 
    \rangle_{L^2(\mathbb R^d;\mathbb R^{m})}
    =
    \delta_{\ell \ell'},
\end{array}
\end{equation}
which holds for any $\ell\ge 0$ and $\ell'\ge 0$.
Recalling \eqref{eq:nt:def:cal:p}, this allows us to
compute
\begin{equation}
    \mathcal{P}_k(x,y)
    = \sum_{|\gamma|\le k} \sum_{\ell=0}^\infty\lambda_\ell | \partial^\gamma \mu_\ell(x,y)|^2 .
\end{equation}
Using the fact that $k' > k + d/2$ holds, we obtain a uniform bound on all the derivatives,
and see that there exists a constant $K > 0$ so that
\begin{equation}
    \mathcal{P}_k(x,y) \le 
    K \sum_{\ell=0}^\infty \lambda_\ell \| \mu_\ell \|_{H^{k'}(\mathbb R^{d-1}; \mathbb R^{m})}^2 \label{eq:70}
\end{equation}
holds for every $(x,y) \in \mathbb R^d$,
mimicking \eqref{eq:nt:loc:bnd:p:k:unif}.
However, \eqref{eq:nt:loc:bnd:p:k:intg:y}
is now replaced by the full integral bound
\begin{equation}
    \int_{\mathbb R}\int_{\mathbb R^{d-1}} \mathcal{P}_k(x,y) \, \mathrm dy \,\mathrm dx
    \le 
    K \sum_{\ell=0}^\infty \lambda_\ell \| \mu_\ell \|_{H^{k}(\mathbb R^{d}; \mathbb R^{m})}^2 .\label{eq:71}
\end{equation}

\begin{lemma}
Pick $k > d/2$ 
and assume that \textnormal{(Hf)}, \textnormal{(HTw)}, \textnormal{(A$\tilde{g}$)}, 
and \textnormal{(AqT)} hold. Then \textnormal{(HgQ)} is satisfied.
\end{lemma}
\begin{proof}
In view of the bounds \eqref{eq:70} and \eqref{eq:71},
this follows directly from
Corollaries \ref{cor:nt:p:unif:bnd} and \ref{cor:nt:p:intg:full}
together with Lemma \ref{lem:nt:hgq:satisfied}.
\end{proof}

\begin{lemma}
Pick $k > d/2$ 
and assume that \textnormal{(Hf)}, \textnormal{(HTw)}, \textnormal{(A$\tilde{g}$)}, 
and \textnormal{(AqT)} hold. Then  \textnormal{(H$L^1$)} is satisfied.
\end{lemma}
\begin{proof}
First, observe that there exists a constant $C_N >0$ with
\begin{equation}
    \|\tilde{g}(T_{\gamma} \Phi_0 + v)\|_\infty
    \| T_{\gamma} \psi_{\rm tw} \|_\infty \le C_N,
\end{equation}
for every $\gamma \in \mathbb R$ and $v \in H^k$
with $\|v\|_{H^k} \le N$, on account of 
\eqref{eq:mr:lipschitz:g}
and the fact that $\Phi_0$ is bounded.
This allows us to compute
\begin{equation}
\begin{aligned}
    \|(g(T_{\gamma} \Phi_0+&v) p^{j})^\top T_{\gamma} \psi_{\rm tw}\|_{L^1(\mathbb R^d;\mathbb R)}^2
    \\& \le \textstyle C_N^2
    \big[
    \sum_{\ell=0}^\infty \lambda_\ell^{1/2} \|\mu_\ell\|_{L^1(\mathbb R^d; \mathbb R^m)}|\tilde{p}_{j;\ell}| \big]^2
\\ 
& \le \textstyle 
C_N^2  
\sum_{\ell=0}^\infty \lambda_\ell^{1/2} \|\mu_\ell\|_{L^1(\mathbb R^d; \mathbb R^m)}
\sum_{\ell=0}^\infty \lambda_\ell^{1/2} \|\mu_\ell\|_{L^1(\mathbb R^d; \mathbb R^m)} \tilde{p}_{j;\ell}^2 .
\end{aligned}
\end{equation}
Applying \eqref{eq:nt:full:loc:sum:ptilde:j:full}, we hence find
\begin{equation}
\begin{array}{lcl}
 \sum_{j=0}^\infty   \|(g(T_{\gamma}\Phi_0+v) p^{j})^\top T_{\gamma} \psi_{\rm tw}\|_{L^1(\mathbb R^d;\mathbb R)}^2
     \le  
C_N^2 \big( \sum_{\ell=0}^\infty \lambda_\ell^{1/2} \| \mu_\ell\|_{L^1(\mathbb R^{d}; \mathbb R^{ m})} \big)^2,
\end{array}
\end{equation}
which completes the proof in view of \eqref{eq:nt:aqt:sum:cond}.
\end{proof}

\printbibliography

\end{document}